\documentclass[10pt,a4paper,reqno]{amsart}
\usepackage{amsthm}
\usepackage{amsmath}
\usepackage{amssymb}
\usepackage{graphicx,color}

\usepackage[font=small]{caption}
\usepackage{subcaption}

\usepackage{multirow}
\usepackage[shortlabels]{enumitem}

\usepackage{tikz}
\usepackage{pgf,tikz}
\usepackage{mathrsfs}
\usetikzlibrary{arrows}
\usetikzlibrary{snakes}
\usetikzlibrary{arrows}
\usetikzlibrary{calc}
\usetikzlibrary{patterns}
\usepackage{tkz-fct}

\makeatletter 
\@namedef{subjclassname@2020}{%
  \textup{2020} Mathematics Subject Classification}
\makeatother

\theoremstyle{plain}
\newtheorem{thm}{Theorem}
  \theoremstyle{definition}
  \newtheorem*{thm*}{Theorem}
  
  \theoremstyle{remark}
  \newtheorem{rem}[thm]{Remark}
  \theoremstyle{plain}
  \newtheorem{prop}[thm]{Proposition}
  \theoremstyle{plain}
  \newtheorem{lem}[thm]{Lemma}
  \theoremstyle{plain}
  
 \theoremstyle{definition}
  
  \theoremstyle{remark}
  \newtheorem*{rem*}{Remark}

  \theoremstyle{definition}

\usepackage{amsfonts}
\usepackage{mathrsfs}

\newtheorem*{question*}{\it{QUESTION}}

\theoremstyle{plain}

\newcommand{\N}{\mathbb{N}}
\newcommand{\R}{{\mathbb{R}}}
\newcommand{\C}{{\mathbb{C}}}

\newcommand{\dd}{{\rm d}}
\newcommand{\ii}{{\rm i}}

\newcommand{\diag}{\mathop\mathrm{diag}\nolimits}

\newcommand{\sign}{\mathop\mathrm{sign}\nolimits}
\renewcommand{\Re}{\mathop\mathrm{Re}\nolimits}
\renewcommand{\Im}{\mathop\mathrm{Im}\nolimits}
\newcommand{\supp}{\mathop\mathrm{supp}\nolimits}

\def\pFq#1#2#3#4#5{ 
  {}_{#1}F_{#2}\biggl(\genfrac..{0pt}{}{#3}{#4}\biggl|\,#5\biggr)
}

\usepackage[normalem]{ulem}
\definecolor{DarkGreen}{rgb}{0,0.5,0.1} 
\definecolor{DarkRed}{rgb}{0.54,0,0} 

\newcommand\soutP{\bgroup\markoverwith
{\textcolor{DarkGreen}{\rule[.5ex]{2pt}{1pt}}}\ULon}

\makeatother

\begin{document}

\title[]{Asymptotic root distribution of Charlier polynomials with large negative parameter}

\author{Petr Blaschke}
\address[Petr Blaschke]{Mathematical Institute in Opava, Silesian University in Opava, Na Rybn{\' i}{\v c}ku~626/1, 746 01 Opava, Czech Republic
	}
\email{Petr.Blaschke@math.slu.cz}

\author{Franti\v sek \v Stampach}
\address[Franti{\v s}ek {\v S}tampach]{
	Department of Mathematics, Faculty of Nuclear Sciences and Physical Engineering, Czech Technical University in Prague, Trojanova~13, 12000 Praha~2, Czech Republic
	}	
\email{stampach@fjfi.cvut.cz}



\date{\today}

\begin{abstract}
We analyze the asymptotic distribution of roots of Charlier polynomials with negative parameter depending linearly on the index. The roots cluster on curves in the complex plane. We determine implicit equations for these curves and deduce the limiting density of the root distribution supported on these curves. The proof is based on a determination of the limiting Cauchy transform in a specific region and a careful application of the saddle point method. The obtained result represents a solvable example of a
more general open problem. 
\end{abstract}

\maketitle

\section{Introduction}

We study an asymptotic distribution of roots of a classical family of orthogonal polynomials, namely the Charlier polynomials, in a regime when they are not orthogonal with respect to a positive measure supported on $\R$ as their parameter assumes non-standard values. This means that the studied family of polynomials cannot be identified with characteristic polynomials of a Hermitian Jacobi matrix and their roots are not confined to the real line. Rather than that they cluster in complex curves. We give an exact description of these curves, sometimes referred to as the zero attractors for the studied polynomial sequence, and deduce the asymptotic density of the distribution of the roots.

The Charlier polynomials are defined by the formula
\[
 C_{n}^{(a)}(x):=\pFq{2}{0}{-n,-x}{-}{-\frac{1}{a}}
 =\sum_{k=0}^{n}\binom{n}{k}\frac{(-x)_{k}}{a^{k}},
\]
where $n\in\N_{0}$ and $a\neq0$ is a parameter. For basic definitions and properties of the hypergeometric functions and Charlier polynomials, we refer the reader to the Askey scheme~\cite{koe-les-swa_10}. If the parameter $a$ is positive, Charlier polynomials are orthogonal with respect to the Poisson probability measure, see~\cite[Eq~9.14.2]{koe-les-swa_10}. We assume the parameter to be negative and, in addition, dependent linearly on the index. In such a setting, we further scale and slightly transform the main variable to keep roots in a rectangular window with a side of unit length for any index $n$. As a~result, we consider polynomials
\begin{equation}
 P^{\mathrm{C}}_{n}(z;a):=C_{n}^{(-an)}\left(zn-an-1\right),
\label{eq:def_pol_C}
\end{equation}
for $n\in\N$ and $a>0$. Since all coefficients of $P_{n}^{\mathrm{C}}(\,\cdot\,;a)$ are real, the roots of $P_{n}^{\mathrm{C}}(\,\cdot\,;a)$ are distributed symmetrically in $\C$ with respect to the real line. The main result on the asymptotic distribution of roots of $P^{\mathrm{C}}_{n}(\,\cdot\,;a)$, as $n\to\infty$, is given in Theorem~\ref{thm:charlier}. This result represents a solvable model of a general problem whose solution is currently out of reach; more details given in Section~\ref{sec:sampling_jacobi}.

The problem on the asymptotic distribution of roots of orthogonal polynomials with variable and non-standard parameters has been extensively studied for many families. Our setting is reminiscent to a similar problem concerning the Laguerre polynomials with variable negative parameter that received a lot of attention~\cite{mar-fin-mar-gon-ori_01, kui-mcl_01, kui-mcl_04, dai-won_08, dia-ori_11, ati-mar-fin-mar-gon-tha_14}. In fact, the case of Laguerre polynomials has been studied in a greater generality, as the parameter $a_{n}$ in the studied sequence of Laguerre polynomials $L_{n}^{(a_{n})}(nz)$ is allowed to be a negative multiple of $n$ only asymptotically, i.e., $\lim_{n\to\infty}a_{n}/n=a<0$. In such a setting, the asymptotic distribution of roots exhibits remarkable complexity. Moreover, as a non-Hermitian orthogonality has been established for this family, it is a nice example for a demonstration of the powerful steepest descent Riemann--Hilbert problem method. 

As far as the asymptotic properties of Charlier polynomials are concerned, asymptotic expansions for $C_{n}^{(a)}(nz)$, as $n\to\infty$, with the fixed parameter $a>0$ and various restrictions on $z$, were found in~\cite{goh_ca98, rui-wong_94, lop-tem_04} using primarily the saddle point method. For a more recent application of the turning point theory in asymptotic analysis of the Charlier polynomials, see~\cite{hua-lin-zha_21}. As it is impossible to cover all related works, we mention at least few more papers on asymptotic behavior of classical orthogonal polynomials with varying non-standard parameters~\cite{mar-fin-mar-gon-ori_99,  kui-mar-fin_04, won-zha_06, mar-fin-ori_05, wan-qui-won_13, wan-won_jmaa16}, hypergeometric polynomials~\cite{dur-gui_01, dri-jor_03, dri-joh_07, zhou-sri-wan_12, aba-bog_16, aba-bog_18}, and other interesting families~\cite{boy-goh_07, boy-goh_08, boy-goh_10, son-won_17, sta_21}.

The paper is organized as follows. In Section~\ref{sec:main}, the main result is stated as Theorem~\ref{thm:charlier} and several related remarks as well as numerical illustrations are added. Next, general steps of the applied method are summarized. The proof of Theorem~\ref{thm:charlier} is gradually worked out in Section~\ref{sec:charlier}. First, roots of~$P_{n}^{\mathrm{C}}$ are shown to be confined to a rectangular domain in~$\C$ uniformly in $n\in\N$, which simplifies the subsequent analysis. Second, by using a contour integral representation for $P_{n}^{\mathrm{C}}$ together with the saddle point method, a formula for the limiting Cauchy transform of a sequence of the root counting measures of $P_{n}^{\mathrm{C}}$ is obtained in Theorem~\ref{thm:cauchy_transf_charlier}. By inspection of analytic properties of the limiting Cauchy transform, the limiting measure is deduced and proved to be determined completely. Lastly, a rigorous but technical justification of the application of the saddle point method in the proof of Theorem~\ref{thm:cauchy_transf_charlier} is postponed to the last subsection of Section~\ref{sec:charlier}. Finally, Section~\ref{sec:sampling_jacobi} interprets Theorem~1 as a solvable model for a general open problem on the asymptotic eigenvalue distribution of complex Jacobi sampling matrices, discusses several related works, and closes with few concluding remarks.

\section{The main result and the method}\label{sec:main}

\subsection{Main results}

First, we need to introduce a notation. The asymptotic distribution of roots of polynomials $P^{\mathrm{C}}_{n}(\,\cdot\,;a)$, for $n\to\infty$, is encoded in the complex function
\begin{equation}
f(\xi;z):=(a-z)\log(1+\xi)+\log\xi-a\xi.
\label{eq:def_f_charlier}
\end{equation}
Let us denote its two critical points, i.e., solutions of the equation $\partial_{\xi} f(\xi;z)=0$, by
\begin{equation}
\xi_{\pm}=\xi_{\pm}(z;a):=\frac{1}{2a}\left(1-z\pm\sqrt{(1-z)^{2}+4a}\right).
\label{eq:xi_plus_minus_charlier}
\end{equation}
If not stated otherwise, the multi-valued functions, such as the logarithm and the square root, assume their principal values.

The parameter $z$ is restricted to the rectangular domain $(0,1)+2\ii\sqrt{a}(-1,1)$, in which the equation
\begin{equation}
\Re f(\xi_{+};z)=\Re f(\xi_{-};z)
\label{eq:re_f_eq}
\end{equation}
determines an arc of a simple smooth curve connecting a unique point in the interval $(0,1)$ with the upper-left corner $1+2\ii\sqrt{a}$. This claim is proven in Proposition~\ref{lem:Omega0_basic_prop_charlier} and is illustrated in Figure~\ref{fig:cut_thm1}. Let $\gamma:[0,1]\to\C$ denotes a parametrization of the arc with the starting point $\gamma(0)=1+2\ii\sqrt{a}$ and the end point $\gamma(1)\in(0,1)$. Then the asymptotic root distribution of polynomials $P^{\mathrm{C}}_{n}(\,\cdot\,;a)$, for $n\to\infty$, works in two regimes depending on the value of $a$ and is as follows.

\begin{thm}\label{thm:charlier}
 Let $a>0$, $f$ and $\xi_{\pm}$ functions given by \eqref{eq:def_f_charlier} and \eqref{eq:xi_plus_minus_charlier}, and $\gamma$ as defined above. Then, for $n\to\infty$, the sequence of root counting measures of polynomials $P^{\mathrm{C}}_{n}(\,\cdot\,;a)$ converges weakly to a probability measure $\mu=\mu_{1}+\mu_{2}$, for which the following holds true:
  \begin{enumerate}[{\upshape i)}]
  \item Measure $\mu_{1}$ is symmetric with respect to the real line, supported on the image of $\gamma$ and its complex conjugate $\bar{\gamma}$, and absolutely continuous. For $0\leq\alpha<\beta\leq 1$, the measure of the arc $\gamma([\alpha,\beta])$ of the curve $\gamma$ connecting the points $\gamma(\alpha)$ and $\gamma(\beta)$ is given by the difference
  \[
   \mu_{1}\left(\gamma([\alpha,\beta]\right)=\rho\left(\gamma(\beta)\right)-\rho\left(\gamma(\alpha)\right),
  \]
where 
  \[
   \rho(z):=\frac{1}{2\pi\ii}\left(f(\xi_{+}(z;a);z)-f(\xi_{-}(z;a);z)\right).
  \]
 \item If $a\geq\gamma(1)$, then $\mu_{2}=0$. If $a<\gamma(1)$, then $\mu_{2}$ is supported on the real interval $[a,\gamma(1)]$, is absolutely continuous, and its density reads
  \[
    \frac{\dd\mu_{2}}{\dd x}(x)=1, \quad x\in[a,\gamma(1)].
  \]
 \end{enumerate}
\end{thm}

\begin{figure}[htb!]
    \centering
        \includegraphics[width=0.99\textwidth]{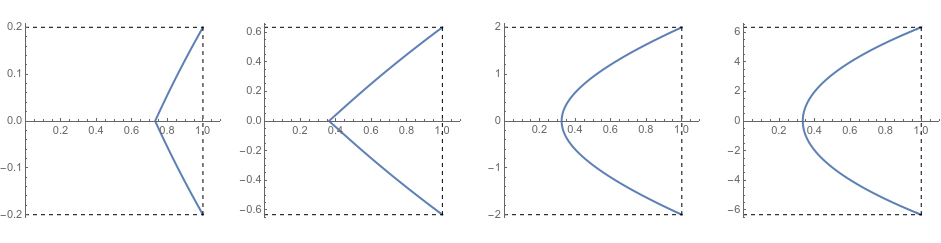}
        \caption{The curve determined by the equation $\Re f(\xi_{+};z)=\Re f(\xi_{-};z)$ in the rectangular window $(0,1)+2\ii\sqrt{a}(-1,1)$ for $a=0.01, 0.1,1,10.$}
    \label{fig:cut_thm1}
\end{figure}

The proof of Theorem~\ref{thm:charlier} is worked out in Section~\ref{sec:charlier}. The two scenarios of Theorem~\ref{thm:charlier} are illustrated in Figures~\ref{fig:thm1a} and~\ref{fig:thm1b}.

\begin{figure}[htb!]
    \centering
    \begin{subfigure}[c]{0.49\textwidth}
        \includegraphics[width=0.95\textwidth]{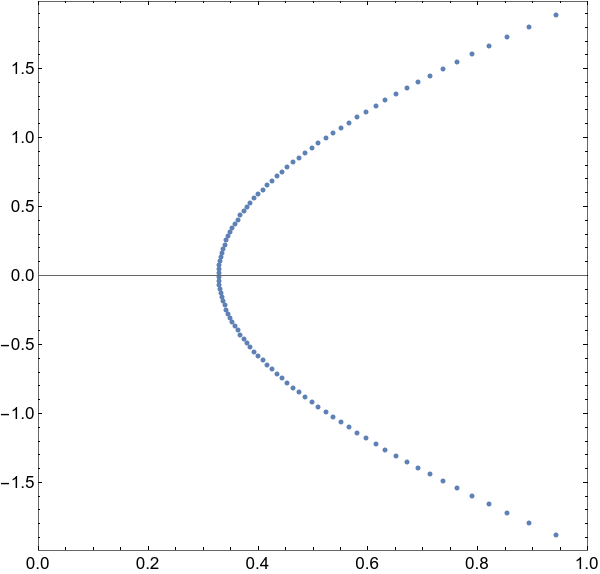}
        \caption{The roots of $P_{100}^{C}$}
    \end{subfigure}
    \begin{subfigure}[c]{0.49\textwidth}
        \includegraphics[width=0.95\textwidth]{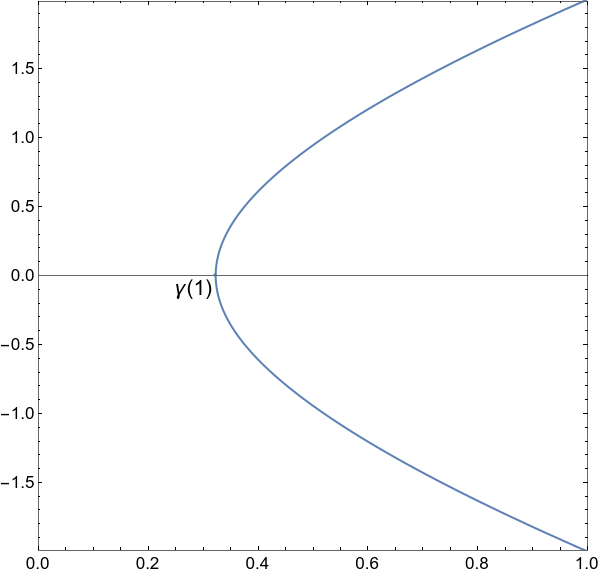}
        \caption{The support of $\mu$.}
    \end{subfigure}
    \caption{An illustration of Theorem~\ref{thm:charlier} with $a=1$ when $\mu_{2}$ is trivial.}
    \label{fig:thm1a}
\end{figure}

\begin{figure}[htb!]
    \centering
        \begin{subfigure}[c]{0.49\textwidth}
        \includegraphics[width=0.95\textwidth]{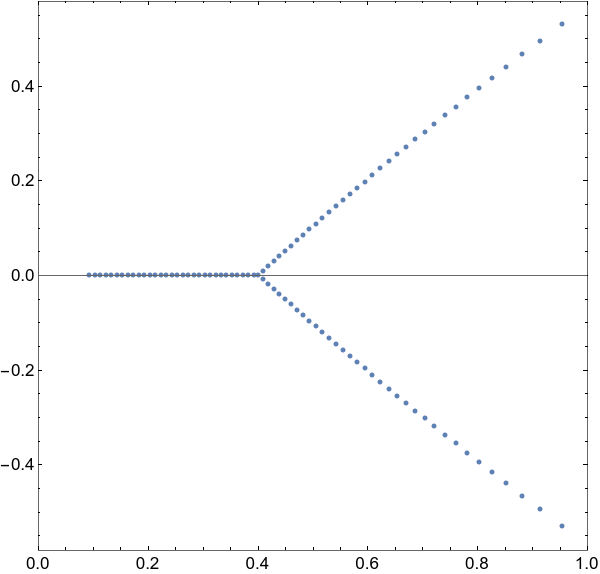}
        \caption{The roots of $P_{100}^{C}$}
    \end{subfigure}
    \begin{subfigure}[c]{0.49\textwidth}
        \includegraphics[width=0.95\textwidth]{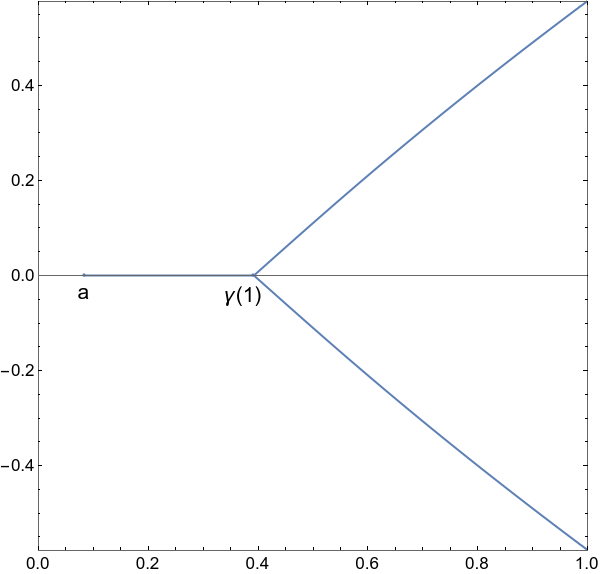}
        \caption{The support of $\mu$.}
    \end{subfigure}
    \caption{An illustration of Theorem~\ref{thm:charlier} with $a=1/12$ when $\mu_{2}$ is nontrivial}
    \label{fig:thm1b}
\end{figure}

\begin{rem}
The point $\gamma(1)$ depends on $a$ and coincides with the unique solution of the equation $\Re f(\xi_{+}(x;a);x)=\Re f(\xi_{-}(x;a);x)$ for $x\in(0,1)$. The equation does not seem to be explicitly solvable, however. For example, if we substitute $x=1-2\sqrt{a}\sinh t$, we arrive at the equation
\[
 2t-\sqrt{a}\left(e^{t}+e^{-t}\right)+\left[a-1+\sqrt{a}\left(e^{t}-e^{-t}\right)\right]\ln\frac{\sqrt{a}+e^{t}}{|\sqrt{a}-e^{-t}|}=0,
\]
which determines implicitly a positive solution $t_{0}=t_{0}(a)$, and $\gamma(1)=1-2\sqrt{a}\sinh t_{0}$.
Without going into details, let us remark the asymptotic behavior of $\gamma(1)$ for $a$ small and $a$ large is as follows:
\[
 \gamma(1)=1-\frac{2\sqrt{a}}{\sqrt{y_{0}^{2}-1}}+O(a), \quad a\to0+,
\]
and
\[
 \gamma(1)=\frac{1}{3}-\frac{4}{405a}+O\left(\frac{1}{a^{2}}\right), \quad a\to\infty,
\]
where $y_{0}\approx1.1997$ is the unique positive solution of the transcendental equation $y=\coth y$.
The dependence of $\gamma(1)$ on the parameter $a$ is nontrivial. Based on a numerical analysis, it seems that $\gamma(1)$, as function of $a>0$, first decays from the value $1$ to its global minimum of approximate value $0.2693$ attained at $a_{0}\approx0.2342$, and then, for $a>a_{0}$, $\gamma(1)$ is increasing to $1/3$.
\end{rem}

\begin{rem}
 The threshold value $a$, which distinguishes between $\mu_{2}=0$ and $\mu_{2}\neq0$ in Theorem~\ref{thm:charlier}, occurs when $a=\gamma(1)$. Since $\xi_{+}(a;a)=1/a$ and $\xi_{-}(a;a)=-1$ the equation $\Re f(\xi_{+}(a;a);a)=\Re f(\xi_{-}(a;a);a)$ simplifies to the transcendental equation
 \[
  ae^{1+a}=1,
 \]
 whose solution is approximately $a\approx0.2785$ and can be identified with the value of the Lambert $W$ function at the point $1/e$.
\end{rem}

\begin{rem}\label{rem:dens_more_explicit_charlier}
If we write $z=\gamma(t)$ in~\eqref{eq:re_f_eq}, differentiate with respect to $t$, and make use of formulas $\partial_{\xi} f(\xi_{\pm};z)=0$ and $\partial_{z} f(\xi_{\pm};z)=-\log(1+\xi_{\pm})$, we deduce that the curve $\gamma$ is a~solution of the first order differential equation
\[
\Re\left[\gamma'(t)\left(\log(1+\xi_{+}(\gamma(t);a))-\log(1+\xi_{-}(\gamma(t);a))\right)\right]=0, \quad t\in(0,1),
\]
In a slightly more explicit terms, if we parametrize $\gamma$ as $\gamma(t)=x(t)+2\ii\sqrt{a}(1-t)$, we obtain the ordinary differential equation 
\[
x'(t)\log\left|\frac{1+\xi_{+}}{1+\xi_{-}}\right|+2\sqrt{a}\arg\left(\frac{1+\xi_{+}}{1+\xi_{-}}\right)=0,
\]
where $\xi_{\pm}=\xi_{\pm}(x(t)+2\ii\sqrt{a}(1-t);a)$. The existence of such a~parametrization for the curve $\gamma$ follows from the proof of Proposition~\ref{lem:Omega0_basic_prop_charlier}. With the aid of this observation, the density of the limiting measure from the claim~(i) of Theorem~\ref{thm:charlier} can be expressed as
\begin{equation}
 \frac{\dd\mu_{1}}{\dd t}(t)=\frac{\sqrt{a}}{\pi}\left|\log\left(\frac{1+\xi_{+}}{1+\xi_{-}}\right)\right|^{2}\bigg/\log\left|\frac{1+\xi_{+}}{1+\xi_{-}}\right|, \quad t\in(0,1),
\label{eq:dens_charlier}
\end{equation}
where $\xi_{\pm}=\xi_{\pm}(x(t)+2\ii\sqrt{a}(1-t);a)$. Expression~\eqref{eq:dens_charlier} shows that $\mu_{1}$ is a positive measure, indeed. Density~\eqref{eq:dens_charlier} is plotted in Figure~\ref{fig:density_thm1} for several values of $a$.
\end{rem}

\begin{figure}[htb!]
    \centering
        \includegraphics[width=0.95\textwidth]{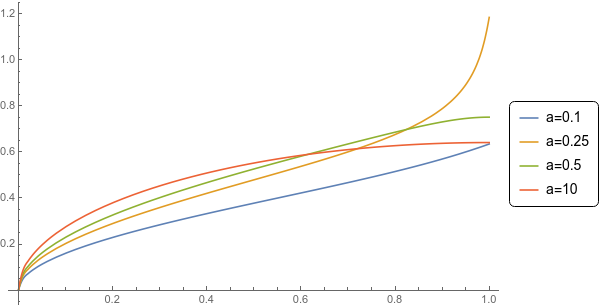}
        \caption{Density~\eqref{eq:dens_charlier} for $a\in\{0.1,0.25,0.5,10\}$.}
    \label{fig:density_thm1}
\end{figure}

\subsection{The method}

We summarize steps of the strategy used to deduce the asymptotic root distributions for polynomials $P_{n}^{\mathrm{C}}$, for $n\to\infty$. Recall that, in general, to deduce the asymptotic root distribution of a sequence of polynomials $P_{n}$ of degree $n$ means to show that the sequence of the root counting measures
\[
 \mu_{n}:=\sum_{\lambda\in P_{n}^{-1}(\{0\})}\frac{\nu(\lambda)}{n}\,\delta_{\lambda}
\]
is weakly convergent, for $n\to\infty$, and provide a description of the limiting measure such as the support, absolute continuity, density, etc. In the above formula, $\nu(\lambda)$ is the multiplicity of the root $\lambda$ of $P_{n}$. For $P_{n}=P_{n}^{\mathrm{C}}$, we apply the following strategy which combines several conventional methods.

First, we prove that the roots of $P_{n}^{\mathrm{C}}$ are localized in a fixed rectangular domain in $\C$ for all $n\in\N$. This allows us to restrict the analysis to this domain which is, although not immediately obvious, a significant simplification in the applied approach. 

Second goal is a determination of the limiting Cauchy transform of the sequence of root counting measures. 
Recall the Cauchy (or Stieltjes) transform $C_{\mu}$ of a Borel measure $\mu$ supported in $\C$ is defined as the convolution $C_{\mu}:=1/z\ast\mu$ understood in the sense of generalized functions on $\C$. For $z\in\C\setminus\supp\mu$, the Cauchy transform of $\mu$ is given by the integral
\[
 C_{\mu}(z):=\int_{\C}\frac{\dd\mu(\xi)}{z-\xi}.
\]
If measures $\mu_{n}$ are supported in a compact subset of $\C$ for all $n$ and the sequence of their Cauchy transforms $C_{\mu_{n}}$ tends to a function $C$ almost everywhere in $\C$ (with respect to the 2D Lebesgue measure), then $\mu_{n}$ converges weakly to a~measure $\mu$ and $C=C_{\mu}$ is its Cauchy transform. If $\mu_{n}$ is the root counting measure of a polynomial $P_{n}$, we have the useful formula
\begin{equation}
 C_{\mu_{n}}(z)=\frac{P_{n}'(z)}{nP_{n}(z)}.
\label{eq:cauchy_transf_root_count}
\end{equation}
Then the derivation of the limiting Cauchy transform is typically a matter of an application of a conveniently chosen asymptotic method for $P_{n}$ and its derivative.

To recover $\mu$ from its Cauchy transform $C_{\mu}$, one uses the generalized formula $\pi\mu=\partial_{\bar{z}}C_{\mu}$. In all the cases investigated below, $C_{\mu}$ will be an analytic function everywhere except several branch cuts occurring on a finite number of smooth closed arcs. Then, if $\gamma:(\alpha,\beta)\to\C$ denotes a simple oriented open smooth curve on which $C_{\mu}$ has the branch cut, $\mu$ is absolutely continuous there and its density is given by the Plemelj--Sokhotski formula
\begin{equation}
\frac{\dd\mu}{\dd t}(t)=-\frac{\gamma'(t)}{2\pi\ii}\left(C_{\mu}(\gamma(t)+)-C_{\mu}(\gamma(t)-)\right), \quad t\in(\alpha,\beta),
\label{eq:plemelj-sokhotski}
\end{equation}
provided that the jump of $C_{\mu}$ on the image of $\gamma$ is an integrable function. Here $C_{\mu}(\gamma(t)\pm)$ denote the non-tangential limits to $\gamma(t)$ from the left/right side of $\gamma$ induced by the chosen orientation. A~more detailed description of this approach is given in~\cite[Sec.~2.1]{bla-sta_jmaa20} and general aspects on the potential theory and generalized functions can be found in books~\cite{saf-tot_97,vla_84}.

The indicated method relying on the Cauchy transform is quite general and has been successfully applied many times in various problems. Usually, the most limiting aspect is the unspecified part concerning the asymptotic analysis of~\eqref{eq:cauchy_transf_root_count}, for $n\to\infty$, that may differ from case to case depending on what is known about the polynomials $P_{n}$ (generating functions, difference equations, non-Hermitian orthogonality relations, etc.). In this work, the classical saddle point method turns out to be applicable. For a reference, we recall a~formulation of the saddle point method which is taken from~\cite[Thm.~7.1, p.~127]{olv_97} specialized slightly to our needs; see also Perron's method in~\cite[Sec.~II.5]{won_01}.

\begin{thm}[saddle point method]\label{thm:saddle-point}
 Let the following assumptions hold:
 \begin{enumerate}[{\upshape i)}]
  \item Let $f$ and $g$ be functions independent of $n$, single valued, and analytic in a domain $M\subset\C$.
  \item The integration path~$\gamma$ is independent of $n$ and its range is located in~$M$ with a possible exception of the end-points.
  \item There is a point $\xi_{0}$ located on the path~$\gamma$ which is not an end-point such that $f'(\xi_{0})=0$ and $f''(\xi_{0})\neq0$ (i.e., $\xi_{0}$ is the saddle point of $\Re f$).
  \item The integral
  \[
   \int_{\gamma}g(\xi)e^{-n f(\xi)}\dd\xi
  \]
  converges absolutely for all $n$ sufficiently large.
  \item One has
  \[
   \Re\left(f(\xi)-f(\xi_{0})\right)>0,
  \]
  for all $\xi\neq\xi_{0}$ that lie on the range of~$\gamma$.
 \end{enumerate}
 Then 
 \begin{equation}
  \int_{\gamma}g(\xi)e^{-n f(\xi)}\dd\xi=g(\xi_{0})e^{-nf(\xi_{0})}\sqrt{\frac{2\pi}{nf''(\xi_{0})}}\left(1+O\left(\frac{1}{\sqrt{n}}\right)\!\right)\!, \quad \mbox{ as } n\to\infty.
 \label{eq:saddle-point}
 \end{equation}
\end{thm}

\begin{rem}
The branch of the square root in~\eqref{eq:saddle-point} has to be appropriately chosen. The choice depends only on the behavior of the phase of the function~$f''$ and the curve $\gamma$ in a~neighborhood of $\xi_{0}$, see~\cite[Chp.~4, Sec.~7]{olv_97} for a detailed explanation.
\end{rem}

The saddle point method will be applied to an integral representation for $P_{n}^{\mathrm{C}}$ deduced from suitable generating function formulas and the Cauchy integral formula. The verification of assumption (v) of Theorem~\ref{thm:saddle-point} represents the most challenging task of the asymptotic analysis and is treated in separate Section~\ref{subsec:spm_charlier}.

We may summarize the main steps of the applied strategy:
\begin{enumerate}[{\upshape 1)}]
\item We show that roots of $P_{n}^{\mathrm{C}}$ are confined to a fixed rectangular domain for all $n\in\N$.
\item We deduce a contour integral representation for $P_{n}^{\mathrm{C}}$ from a generating function formula and apply the saddle point method to obtain the limiting Cauchy transform in the interior of the rectangular domain.
\item We show that the limiting Cauchy transform has a branch cut in the the open rectangular domain and possibly on one of its sides. We apply the Plemelj--Sokhotski formula to recover the limiting measure $\mu$.
\item We check that the obtained limiting measure $\mu$ was determined completely by showing that $\mu$ is a probability measure, i.e., by direct verification of the identity $\int\dd\mu=1$.
\end{enumerate}

\section{The Charlier polynomials}\label{sec:charlier}

\subsection{An approximate localization of roots}\label{subsec:rectang_charlier}

As an initial step, we roughly localize the roots of polynomials $P_{n}^{\mathrm{C}}(\,\cdot\,;a)$ for all $n\in\N$.

It readily follows from the three-term recurrence for Charlier polynomials~\cite[Eq.~9.14.3]{koe-les-swa_10}
\[
 -xC_{n}^{(a)}(x)=aC_{n+1}^{(a)}(x)-(n+a)C_{n}^{(a)}(x)+nC_{n-1}^{(a)}(x), \quad n\in\N,
\]
that the polynomials
\[
 p_{k}^{(n)}(z;a):=a^{k}C_{k}^{(-an)}\left(zn-an-1\right)
\]
fulfill the recurrence
\[
 p_{k+1}^{(n)}(z;a)=\left(z-\frac{k+1}{n}\right)p_{k}^{(n)}(z;a)+\frac{ak}{n}\,p_{k-1}^{(n)}(z;a),
\]
for all $k,n\in\N$. It implies that $p_{n}^{(n)}(z;a)$, which coincides with $a^{n}P_{n}^{\mathrm{C}}(z;a)$, is the characteristic polynomial of the $n\times n$ tridiagonal matrix $J_{n}(a)$, whose entries are given by formulas
\begin{equation}
 \left(J_{n}(a)\right)_{k,k}=\frac{k}{n} \quad\mbox{ and }\quad
 \left(J_{n}(a)\right)_{k+1,k}=\left(J_{n}(a)\right)_{k,k+1}=\ii\sqrt{\frac{ak}{n}},
\label{eq:jacobi_matrix_entries_charlier}
\end{equation}
for $k=1,2\dots,n$ and $k=1,2,\dots,n-1$, respectively, i.e.,
\begin{equation}
 a^{n}P_{n}^{\mathrm{C}}(z;a)=\det\left(z-J_{n}(a)\right),
\label{eq:char_pol_charlier}
\end{equation}
for all $n\in\N$, $a>0$, and $z\in\C$. This connection with Jacobi matrices allows us to localize the roots of $P_{n}^{\mathrm{C}}(\,\cdot\,;a)$ as follows.

\begin{lem}\label{lem:localization_charlier}
 For $a>0$ and all $n\in\N$, roots of $P_{n}^{\mathrm{C}}(\,\cdot\,;a)$ are localized in the rectangular domain $(0,1]+2\ii\sqrt{a}(-1,1)$.
\end{lem}

\begin{proof}
 It follows from~\eqref{eq:char_pol_charlier} that roots of $P_{n}^{\mathrm{C}}(\,\cdot\,;a)$ coincide with eigenvalues of $J_{n}(a)$. In the proof, we show that the numerical range of $J_{n}(a)$ is localized in $(0,1]+2\ii\sqrt{a}(-1,1)$, i.e.,
\[
  0<\Re(x,J_{n}(a)x)\leq1 \quad \mbox{ and }\quad  -2\sqrt{a}<\Im(x,J_{n}(a)x)<2\sqrt{a},
\]
for all $x\in\C^{n}$ with $\|x\|=1$.
  As the spectrum of a matrix is a subset of its numerical range, the statement will readily follow.
 
Notice that, for any $x\in\C^{n}$, one has 
 \[
 \Re(x,J_{n}(a)x)=(x,(\Re J_{n}(a))x) \quad \mbox{ and }\quad \Im(x,J_{n}(a)x)=(x,(\Im J_{n}(a))x),
 \]
where
\[ 
 \Re J_{n}(a):= \frac{J_{n}(a)+J_{n}^{*}(a)}{2} \quad \mbox{ and }\quad \Im J_{n}(a):= \frac{J_{n}(a)-J_{n}^{*}(a)}{2\ii}.
\]
 Fix $n\in\N$ and $x\in\C^{n}$ such that $\|x\|=1$. Since $\Re J_{n}(a)=\diag(\frac{1}{n},\frac{2}{n},\dots,1)$, it clearly holds $0<(x,(\Re J_{n}(a))x)\leq 1$.
 Next, one has
 \[
  \Im J_{n}(a) = \sqrt{\frac{a}{n}}\left(W+W^{*}\right)\!,
 \]
 where $W=\diag(1,\sqrt{2},\dots,\sqrt{n})\,U$, $Ue_{1}=0$, and $Ue_{k}=e_{k-1}$, for $k\in\{2,\dots,n\}$, where $\{e_{1},\dots,e_{n}\}$ stands for the standard basis of $\C^{n}$.
 Then it is easy to see that
 \[
 \|\Im J_{n}(a)\|\leq2\sqrt{\frac{a}{n}}\,\|W\|\leq 2\sqrt{\frac{a(n-1)}{n}}<2\sqrt{a}.
 \]
 The last inequality implies that $|(x,(\Im J_{n}(a))x)|<2\sqrt{a}$.
\end{proof}

\subsection{The limiting Cauchy transform}

An application of the Cauchy integral formula to the generating function for the Charlier polynomials~\cite[Eq.~9.4.11]{koe-les-swa_10}
\[
e^{t}\left(1-\frac{t}{a}\right)^{x}=\sum_{n=0}^{\infty}C_{n}^{(a)}(x)\frac{t^{n}}{n!}, \quad |t|<a,
\]
yields the contour integral representation 
\[
 C_{n}^{(a)}(x)=\frac{1}{2\pi\ii}\oint_{\gamma}e^{z}\left(1-\frac{z}{a}\right)^{x}\frac{\dd z}{z^{n+1}},
\]
where $\gamma$ is a positively oriented Jordan curve with $0$ in its interior located in the domain of analyticity of the
integrated function. By using~\eqref{eq:def_pol_C}, we get
\[
 P_{n}^{\mathrm{C}}(z;a)=\frac{1}{2\pi}\oint_{\gamma}e^{v}\left(1+\frac{v}{an}\right)^{zn-an-1}\frac{\dd v}{v^{n+1}},
\]
where $\gamma$ is a positively oriented Jordan curve with $0$ in its interior located in the cut-plane $\C\setminus(-\infty,-an]$. Making the substitution $v=an\xi$ in the last integral results in the formula
\[
 P_{n}^{\mathrm{C}}(z;a)=\frac{1}{2\pi\ii a^{n}n^{n}}\oint_{\gamma}e^{an\xi}\left(1+\xi\right)^{zn-an-1}\frac{\dd \xi}{\xi^{n+1}},
\]
where $\gamma$ is a positively oriented Jordan curve with $0$ in its interior located in $\C\setminus(-\infty,-1]$.
Thus, we have arrived at the integral representation
\begin{equation}
 P_{n}^{\mathrm{C}}(z;a)=\frac{1}{2\pi\ii a^{n} n^{n}}\oint_{\gamma}g(\xi)e^{-nf(\xi,z)}\dd\xi,
\label{eq:integr_repre_spm_charlier}
\end{equation}
where
\[
 g(\xi):=\frac{1}{\xi(1+\xi)} \quad \mbox{ and } \quad f(\xi;z)=(a-z)\log(1+\xi)+\log\xi-a\xi,
\]
and $\gamma$ as above. The integral formula~\eqref{eq:integr_repre_spm_charlier} is in a suitable form for the application of the saddle point method.

We make use of formula~\eqref{eq:integr_repre_spm_charlier} and its derivative with respect to $z$, where the function $g$ is simply replaced by $g\cdot(\partial_{\xi}f(\,\cdot\,;z))$, apply Theorem~\ref{thm:saddle-point} and formula~\eqref{eq:cauchy_transf_root_count} to deduce the limiting Cauchy transform. Due to Lemma~\ref{lem:localization_charlier} and the symmetry
\begin{equation}
\overline{P_{n}^{\mathrm{C}}(z;a)}=P_{n}^{\mathrm{C}}(\bar{z};a),
\label{eq:symmetry_charlier}
\end{equation}
it will be sufficient to consider only $z\in(0,1)+2\ii\sqrt{a}(0,1)$. It is easy to see that the assumptions (i) and (ii) of Theorem~\ref{thm:saddle-point} are fulfilled. Concerning the assumption (iii), we have
\[
 \frac{\partial f}{\partial\xi}(\xi;z)=\frac{a-z}{1+\xi}+\frac{1}{\xi}-a.
\]
Hence, critical points of $f(\,\cdot\,;z)$ are roots of the quadratic equation
\[
a\xi^{2}+(z-1)\xi-1=0,
\]
which are $\xi_{\pm}(z;a)$ given by~\eqref{eq:xi_plus_minus_charlier}. The roots $\xi_{\pm}(z;a)$ coincide if and only if $z=1\pm 2\ii\sqrt{a}$.

A verification of the assumption (iv) of Theorem~\ref{thm:saddle-point} is again straightforward.
On the other hand, the last assumption~(v) is crucial and not easy to establish. One needs to check that the curve $\gamma$ from the integral~\eqref{eq:integr_repre_spm_charlier} is homotopic to a Jordan curve located in $\C\setminus((-\infty,-1]\cup\{0\})$ with the property that for all $\xi$ from the image of this curve
\[
\Re\left(f(\xi;z)-f(\xi_{\pm}(z;a);z)\right)>0
\]
with the only exception when either $\xi=\xi_{+}(z;a)$ or $\xi=\xi_{-}(z;a)$. It is by no means clear, whether it is possible. A detailed justification, which is based on a discussion of all possible configurations, is laborious and the most technical part of the applied method. In order not to distract from our primary intention, which is the derivation of the asymptotic root distribution, we postpone the detailed justification to separate Section~\ref{subsec:spm_charlier}. We point out that, when Theorem~\ref{thm:saddle-point} is applied without a rigorous justification of the assumptiton (v), one obtains the right formula for the limiting Cauchy transform in $(0,1)+2\ii\sqrt{a}(0,1)$.

It is natural to introduce two subsets of the rectangular domain $(0,1)+2\ii\sqrt{a}(0,1)$ that identify the regions where either the saddle point $\xi_{+}(z;a)$ or $\xi_{-}(z;a)$ is dominant:
\begin{equation}
 \Omega_{\pm}:=\left\{z\in(0,1)+2\ii\sqrt{a}\,(0,1) \mid \Re f(\xi_{+}(z;a),z) \gtrless \Re f(\xi_{-}(z;a),z)\right\}.
\label{eq:def_omega_pm}
\end{equation}
Further, the common boundary of $\Omega_{+}$ and $\Omega_{-}$ is denoted as
\begin{equation}
 \Omega_{0}:=\left\{z\in(0,1)+2\ii\sqrt{a}\,(0,1) \mid \Re f(\xi_{+}(z;a),z)=\Re f(\xi_{-}(z;a),z)\right\},
\label{eq:def_omega_0}
\end{equation}
see Figure~\ref{fig:omega_sets_charlier}.

\begin{figure}[htb!]
  \includegraphics[width=0.84 \textwidth]{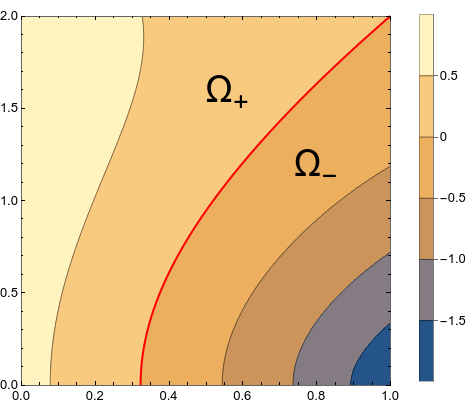}
        \caption{Values of $\Re f(\xi_{+}(z;a),z)- \Re f(\xi_{-}(z;a),z)$ for $z\in(0,1)+2\ii\sqrt{a}\,(0,1)$ and $a=1$. The common boundary~\eqref{eq:def_omega_0} of $\Omega_{\pm}$ is made in red.}
        \label{fig:omega_sets_charlier}
\end{figure}

We will show that $\Omega_{0}$ is an arc of a simple curve which connects a real point in $(0,1)$ with the corner point $1+2\ii\sqrt{a}$. To this end, we need the following auxiliary result.

\begin{lem}\label{lem:basic_prop_aux_charlier}
 For the roots $\xi_{\pm}=\xi_{\pm}(z;a)$ given by~\eqref{eq:xi_plus_minus_charlier}, the following implications hold true.
 \begin{enumerate}[{\upshape i)}]
  \item 
  \[
   \log\left|\frac{1+\xi_{+}}{1+\xi_{-}}\right|=0 \quad\Rightarrow\quad  |z-a|=1+a, \; \Re z\geq 1.
  \]
  \item
  \[
   \arg\left(\frac{1+\xi_{+}}{1+\xi_{-}}\right)\in\{-\pi,0\} \quad\Rightarrow\quad \Im z=0 \;\;\mbox{ or }\;\; |z-a|=1+a, \; \Re z\leq 1.
  \]
 \end{enumerate}
\end{lem}

\begin{proof}
i) One has 
 \[
 \log\left|\frac{1+\xi_{+}}{1+\xi_{-}}\right|=0 \quad\Leftrightarrow\quad
 \left|\frac{1+\xi_{+}}{1+\xi_{-}}\right|=1 \quad\Leftrightarrow\quad
 \frac{1+\xi_{+}}{1+\xi_{-}}=\frac{1+\ii s}{1-\ii s},
 \]
 for a parameter $s\in\R$. Using~\eqref{eq:xi_plus_minus_charlier}, the latter condition yields the quadratic equation
 \[
 s^{2}(2a+1-z)^{2}+(1-z)^{2}+4a=0,
 \]
 whose solutions read
 \[
  z=\frac{1+s^{2}(1+2a)\pm2\ii\sqrt{a+s^{2}a(1+a)}}{1+s^{2}}.
 \]
 For these roots, one obtains 
 \[
  |z-a|^{2}=\frac{\left(1+s^{2}(1+2a)-a(1+s^{2})\right)^{2}+4(a+s^{2}a(1+a))}{(1+s^{2})^{2}}=(1+a)^{2}
 \]
 and
 \[
  \Re z =1+\frac{2as^{2}}{1+s^{2}}\geq 1.
 \]

 ii) One has 
 \[
 \arg\left(\frac{1+\xi_{+}}{1+\xi_{-}}\right)\in\{-\pi,0\} \quad\Leftrightarrow\quad \frac{1+\xi_{+}}{1+\xi_{-}}=t,
 \]
 for a parameter $t\in\R\setminus\{0\}$. Using definition~\eqref{eq:xi_plus_minus_charlier}, the latter equation yields
 \[
 (1+t)^{2}\left((1-z)^{2} + 4a\right)= (1-t)^{2}(2a + 1 - z)^{2},
 \]
 which is again a quadratic equation in $z$. Its solutions read
 \[
  z=\frac{2t-a(1-t)^{2}\pm|1+t|\sqrt{a^2(1-t)^2-4ta}}{2t}.
 \]
 If $a(1-t)^{2}\geq 4t$, i.e., the expression in the square root is non-negative, $\Im z=0$. On the other hand, if $a(1-t)^{2}<4t$, the two roots are
 \[ 
 z=\frac{2t-a(1-t)^{2}\pm\ii|1+t|\sqrt{4ta-a^2(1-t)^2}}{2t},
 \] 
 for which one gets 
 \[
  |z-a|^{2}=\frac{\left(2t-a(1-t)^{2}-2at\right)^{2}+(1+t)^{2}\left(4ta-a^2(1-t)^2\right)}{4t^{2}}=(1+a)^{2}
 \]
 and
 \[
  \Re z =1-\frac{a(1-t)^{2}}{2t}\leq 1
 \]
 since the assumption $a(1-t)^{2}<4t$ implies $t>0$.
\end{proof}

\begin{prop}\label{lem:Omega0_basic_prop_charlier}
 The set $\Omega_{0}$ is an open connected arc of a simple smooth curve with one end-point located in the interval $(0,1)$ and the second end-point $1+2\ii\sqrt{a}$.
\end{prop}

\begin{proof}
We prove the statement by showing that the function 
\[
 g(z):=\Re f(\xi_{+}(z;a),z)-\Re f(\xi_{-}(z;a),z)
\]
has the following properties:
\begin{enumerate}[{\upshape a)}]
 \item For all $y\in[0,2\sqrt{a}]$, the function $x\mapsto g(x+\ii y)$ is strictly monotone in $(0,1)$.
 \item For all $y\in[0,2\sqrt{a})$, $g(1+\ii y)<0$, and $g(1+2\ii\sqrt{a})=0$. 
 \item For all $y\in[0,2\sqrt{a}]$, $g(\ii y)>0$.
\end{enumerate}

The property (a) means that the function $g$ is either strictly increasing, or strictly decreasing on each horizontal line segment in the rectangle $[0,1]+2\ii\sqrt{a}[0,1]$. By (b) and (c), this function is positive at the left end-point of the horizontal line segment and negative at the right end-point except the upper-right corner $1+2\ii\sqrt{a}$, where it vanishes. It follows that $g$ is actually strictly decreasing function with a unique zero in each of the horizontal line segments. It means that the intersection of $\Omega_{0}$ and the horizontal line segment is always a~one-point set. Taking also into account the smoothness of $g$ in $(0,1)+2\ii\sqrt{a}(0,1)$, all claims from the statement readily follow. In fact, it shows even more. Namely, that $\Omega_{0}$ is the graph of a smooth function $x=x(y)$, for $y\in[0,2\sqrt{a}]$, or in other words, that the curve $\gamma$, whose image coincides with the closure of $\Omega_{0}$, can be parametrized by the imaginary part of the variable $z=x+\ii y$.

\textbf{The verification of property~(a):} For brevity, we write $\xi_{\pm}=\xi_{\pm}(z;a)$. By claim (i) of Lemma~\ref{lem:basic_prop_aux_charlier}, the partial derivative
\[
 \frac{\partial g}{\partial x}(x+\ii y)=-\log\left|\frac{1+\xi_{+}}{1+\xi_{-}}\right|
\]
never vanishes for all $x\in(0,1)$ and $y\in[0,2\ii\sqrt{a}]$, from which the property~(a) follows. 

Let us also remark that $g=g(x+\ii y)$, while smooth as function of $x$ for all $y\in(0,2\sqrt{a}]$, need not be smooth in the real interval $(0,1)$, i.e., when $y=0$ (it is the case if $a<1$). It follows from the observation that $\xi_{-}=-1$, if and only if $z=a$, which implies $\partial_{x}g(a)=-\infty$. Nevertheless, $g$ is always continuous in $[0,1]$ as its value $g(a)=-\log(a)-a-1$ is well defined by the respective limit.

\textbf{The verification of property~(b):} The partial derivative 
\begin{equation}
 \frac{\partial g}{\partial y}(x+\ii y)=\arg\left(\frac{1+\xi_{+}}{1+\xi_{-}}\right)
\label{eq:g_partial_der_y_inproof}
\end{equation}
is non-vanishing for $x=1$ and all $y\in(0,2\sqrt{a})$ by claim~(ii) of Lemma~\ref{lem:basic_prop_aux_charlier}. Indeed, the intersection of the vertical line $\Re z=1$ and the circle $|z-a|=a+1$ occurs at points $z=\pm2\ii\sqrt{a}$. Consequently, the continuous function $y\mapsto\ g(1+\ii y)$ is strictly monotone in $(0,2\sqrt{a})$. 

Next, since $\xi_{+}=\xi_{-}$, for $z=1+2\ii\sqrt{a}$, it is clear that $g(1+2\ii\sqrt{a})=0$. One also readily checks that
\[
 g(1)=(a-1)\log\left|\frac{1+\sqrt{a}}{1-\sqrt{a}}\right|-2\sqrt{a}<0,
\]
for all $a>0$. Thus, $y\mapsto\ g(1+\ii y)$ is strictly increasing in $[0,2\sqrt{a}]$, negative in $[0,2\sqrt{a})$, and vanishing at $y=2\sqrt{a}$. This is the property~(b).

\textbf{The verification of property~(c):} 
We distinguish two cases. If $a\leq 1/2$, the entire rectangle $(0,1)+2\ii\sqrt{a}(0,1)$ is located inside the disk $|z-a|<a+1$. Then the claim (ii) of Lemma~\ref{lem:basic_prop_aux_charlier} implies that partial derivative~\eqref{eq:g_partial_der_y_inproof} does not change the sign in $(0,1)+2\ii\sqrt{a}(0,1)$. As we already know its sign to be positive in $1+2\ii\sqrt{a}(0,1)$, it is also positive in $2\ii\sqrt{a}(0,1)$, and hence the function $y\mapsto g(\ii y)$ is strictly increasing in $(0,1)$. It is an easy exercise to check that 
\[
 g(0)=(2a+1)\log\frac{\sqrt{1+4a}+1}{\sqrt{1+4a}-1}-\sqrt{1+4a}>0,
\]
for all $a>0$. It follows the property (c) for $a\leq 1/2$.

If $a>1/2$, the circle $|z-a|=a+1$ intersects $2\ii\sqrt{a}(0,1)$ at the point $\ii\sqrt{1+2a}$. Hence, we similarly infer from the claim (ii) of Lemma~\ref{lem:basic_prop_aux_charlier} that $y\mapsto g(\ii y)$ is strictly increasing this time only in $(0,\sqrt{2a+1})$ and that the partial derivative~\eqref{eq:g_partial_der_y_inproof} can vanish for $x=0$ at most at the point $y=\sqrt{2a+1}$ (which is actually the case). In the last part of the proof, we will show that $g(2\ii\sqrt{a})>0$. Now it suffices to realize that, if the function $g$ has a zero in the segment $\ii(\sqrt{2a+1},2\sqrt{a})$, at which end points it has positive values, $\partial_{y}g$ has to vanish at a~point therein. This would contradict the claim~(ii) of Lemma~\ref{lem:basic_prop_aux_charlier}, however. Thus, there is no zero of $g$ in $\ii(\sqrt{2a+1},2\sqrt{a})$ and so $g$ stays positive in the entire segment $2\ii\sqrt{a}(0,1)$.

It remains to show that $g(2\ii\sqrt{a})>0$. A direct inspection of the value does not seem to be as elementary as in the cases of the remaining corners of the rectangle. Therefore we use an indirect argument. A tedious but straightforward calculation yields
\[
 \log\frac{1+\xi_{+}\left(1+2\ii\sqrt{a}-z^{2};a\right)}{1+\xi_{-}\left(1+2\ii\sqrt{a}-z^{2};a\right)}=\frac{\sqrt{2}\left(1+\sqrt{a}+\ii(1-\sqrt{a})\right)}{\sqrt[4]{a}(1+a)}z+O(z^{2}), \quad z\to0.
\]
It follows that 
\[
 \frac{\partial g}{\partial x}\left(1-\varepsilon^{2}+2\ii\sqrt{a}\right)=-\Re\left(\log\frac{1+\xi_{+}\left(1+2\ii\sqrt{a}-\varepsilon^{2};a\right)}{1+\xi_{-}\left(1+2\ii\sqrt{a}-\varepsilon^{2};a\right)}\right)=
 -\frac{\sqrt{2}\left(1+\sqrt{a}\right)}{\sqrt[4]{a}(1+a)}\varepsilon+O(\varepsilon^{2}),
\]
for $\varepsilon\to0+$. Hence $\partial_{x}g\left(1-\varepsilon^{2}+2\ii\sqrt{a}\right)<0$ for all $\varepsilon>0$ sufficiently small. Since we already know that the partial derivative does not change its sign in the entire line segment $(0,1)+2\ii\sqrt{a}$ we conclude that the function $x\mapsto g(x+2\ii\sqrt{a})$ is strictly decreasing in $[0,1]$. Recalling also that $g(1+2\ii\sqrt{a})=0$, we see that $g(2\ii\sqrt{a})>0$.
\end{proof}

Next, we formulate a statement on the limiting Cauchy transform. Its proof is based on the application of formula~\eqref{eq:cauchy_transf_root_count}, integral representation~\eqref{eq:integr_repre_spm_charlier}, and the saddle point method. It readily follows that
\[
 \lim_{n\to\infty}\frac{\partial_{z}P_{n}^{\mathrm{C}}(z;a)}{nP_{n}^{\mathrm{C}}(z;a)}=-\frac{\partial f}{\partial z}(\xi_{\pm}(z;a);a), \quad\mbox{ for } z\in\Omega_{\mp}.
\]
The rigorous justification of the assumption (v) of Theorem~\ref{thm:saddle-point}, which we skip at this point, is postponed to Section~\ref{subsec:spm_charlier}. Recalling~\eqref{eq:def_f_charlier}, we arrive at the following formula.

\begin{thm}\label{thm:cauchy_transf_charlier}
The limiting Cauchy transform of the sequence of root counting measures $\mu_{n}^{\mathrm{C}}$ of polynomials $P_{n}^{\mathrm{C}}(\,\cdot\,;a)$ in $(0,1)+2\ii\sqrt{a}\, (0,1)$ reads
\[
 C_{\mu}(z):=\lim_{n\to\infty}C_{\mu_{n}^{\mathrm{C}}}(z)=-\frac{\partial f}{\partial z}(\xi_{\pm}(z;a);a)=\log\left(1+\xi_{\pm}(z;a)\right), \quad\mbox{ for } z\in\Omega_{\mp}.
\]
\end{thm}

\subsection{The limiting measure}

We identify branch cuts of the limiting Cauchy transform $C_{\mu}$ of the sequence of root counting measures of polynomials $P_{n}^{\mathrm{C}}(\,\cdot\,;a)$ and, by an application of the Plemelj--Sokhotski formula, reconstruct the limiting measure. After checking that the induced measure is a probability measure, Theorem~\ref{thm:charlier} will be established.

By Lemma~\ref{lem:localization_charlier}, roots of $P_{n}^{\mathrm{C}}(\,\cdot\,;a)$ are located in $[0,1]+2\ii\sqrt{a}[-1,1]$. Taking also symmetry~\eqref{eq:symmetry_charlier} into account, we can restrict the analysis of singularities of $C_{\mu}$ to the open rectangle $(0,1)+2\ii\sqrt{a}(0,1)$ and its boundary.

Concerning the open rectangle $(0,1)+2\ii\sqrt{a}(0,1)$, a discontinuity of $C_{\mu}$ can occur only on the cut $\Omega_{0}$ as it follows from Theorem~\ref{thm:cauchy_transf_charlier}. An application of the Plemejl--Sokhotski formula~\eqref{eq:plemelj-sokhotski} yields a measure $\mu_{1}$, induced by the jump of $C_{\mu}$ on the cut $\Omega_{0}$, is absolutely continuous and, if $\gamma:[0,1]\to\C$ is an injective parametrization which maps $(0,1)$ onto $\Omega_{0}$ with the starting point $\gamma(0)=1+2\ii\sqrt{a}$ and the end point $\gamma(1)\in(0,1)$, see Proposition~\ref{lem:Omega0_basic_prop_charlier}, the density of $\mu_{1}$ reads
\[
\frac{\dd\mu_{1}}{\dd t}(t)=\frac{\gamma'(t)}{2\pi\ii}\left(C_{\mu}(\gamma(t)+)-C_{\mu}(\gamma(t)-)\right), \quad t\in(0,1),
\]
where $C_{\mu}(\gamma(t)\pm)$ denotes the non-tangential limits at $\gamma(t)$ from regions $\Omega_{\pm}$. Using Theorem~\ref{thm:cauchy_transf_charlier} together with the identity $\partial_{\xi}f(\xi_{\pm}(z;a);a)=0$, the density of $\mu_{1}$ can be written as
\[
\frac{\dd\mu_{1}}{\dd t}(t)=\frac{1}{2\pi\ii}\frac{\dd}{\dd t}\left[f(\xi_{+}(\gamma(t);a);\gamma(t))-f(\xi_{-}(\gamma(t);a);\gamma(t))\right].
\]
It follows that the measure of $\gamma([\alpha,\beta])$, for $0\leq\alpha<\beta\leq1$, can be expressed as the difference
\[
 \mu_{1}\!\left(\gamma([\alpha,\beta])\right)=\rho(\gamma(\beta))-\rho(\gamma(\alpha)),
\]
for $\rho(z)=\left[f(\xi_{+}(z;a);z)-f(\xi_{-}(z;a);z)\right]/(2\pi\ii)$. Thus, $\mu_{1}$ coincides with the measure from the claim (i) of Theorem~\ref{thm:charlier}. 

It turns out that the only side of the rectangle $[0,1]+2\ii\sqrt{a}[0,1]$, where $C_{\mu}$ can be discontinuous, is the real segment $[0,1]$. It follow from formulas~\eqref{eq:cauchy_transf_root_count} and~\eqref{eq:symmetry_charlier} that the limiting Cauchy transform has the symmetry
\begin{equation}
 \overline{C_{\mu}(z)}=C_{\mu}(\overline{z}).
\label{eq:symmetry_cauchy_charlier}
\end{equation}
Consequently, Theorem~\ref{thm:cauchy_transf_charlier} provides us with the limiting Cauchy transform also in the rectangle $(0,1)+2\ii\sqrt{a}\,(-1,0)$ and we can check possible discontinuities of $C_{\mu}$ in $[0,1]$ by inspection of the limits from the upper and the lower half-plane.

First note that $C_{\mu}$ continuous analytically through the interval $(\gamma(1),1)$. Indeed, denoting $\Omega_{-}^{*}$ the complex conjugate of $\Omega_{-}$ and taking into account Theorem~\ref{thm:cauchy_transf_charlier} together with the symmetry~\eqref{eq:symmetry_cauchy_charlier}, we have, for $x\in(\gamma(1),1]$, the coinciding side limits
\[
 \lim_{\substack{z\to x \\ z\in\Omega_{-}}} C_{\mu}(z)=\log(1+\xi_{+}(x;a))=\lim_{\substack{z\to x \\ z\in\Omega_{-}}}\overline{C_{\mu}(z)}=\lim_{\substack{z\to x \\ z\in\Omega_{-}^{*}}}C_{\mu}(z)
\]
because $1+\xi_{+}(x;a)>0$.

On the other hand, if $x\in(0,\gamma(1))$, one has to be careful since $1+\xi_{-}(x;a)$ can be negative and one needs to take the right branch of the logarithm in the formula for $C_{\mu}(z)$ when $z$ approaches $x$ from the upper or the lower half-plane. Notice that $1+\xi_{-}(x;a)>0$ if and only if $x<a$. Consequently, one has to distinguish two cases. If $a\geq\gamma(1)$, then, similarly as above, the corresponding side limits coincide,
\[
 \lim_{\substack{z\to x \\ z\in\Omega_{+}}} C_{\mu}(z)=\log(1+\xi_{-}(x;a))=\lim_{\substack{z\to x \\ z\in\Omega_{+}}}\overline{C_{\mu}(z)}=\lim_{\substack{z\to x \\ z\in\Omega_{+}^{*}}}C_{\mu}(z),
\]
for all $x\in[0,\gamma(1))$. However, if $a<\gamma(1)$, $\Im(1+\xi_{-}(z;a))<0$ for $z\in\Omega_{+}$ approaching $x\in(a,\gamma(1))$. Therefore 
\[
 \lim_{\substack{z\to x \\ z\in\Omega_{+}}} C_{\mu}(z)=\log|1+\xi_{-}(x;a)|-\ii\pi,
\]
while
\[
\lim_{\substack{z\to x \\ z\in\Omega_{+}^{*}}} C_{\mu}(z)=\lim_{\substack{z\to x \\ z\in\Omega_{+}}} \overline{C_{\mu}(z)}=\log|1+\xi_{-}(x;a)|+\ii\pi.
\]
For $x<a$, we again have equal side limits. The discontinuity of $C_{\mu}$ on the real interval $[a,\gamma(1)]$ induces an absolutely continuous measure $\mu_{2}$, whose density reads
\[
 \frac{\dd\mu_{2}}{\dd x}(x)=-\frac{1}{2\pi\ii}\left(C_{\mu}(x+)-C_{\mu}(x-)\right)=1, \quad x\in(a,\gamma(1)),
\]
by the Plemejl--Sokhotski formula, where $C_{\mu}(x+)$ and $C_{\mu}(x-)$ are the side limits from $\Omega_{+}$ and $\Omega_{+}^{*}$, respectively. Measure $\mu_{2}$ coincides with the measure from the claim (ii) of Theorem~\ref{thm:charlier}.

At this point, it is not clear whether a significant number of roots of $P_{n}(\,\cdot\,;a)$ can cluster on the boundary of the rectangle $[0,1]+2\ii\sqrt{a}[-1,1]$, for $n\to\infty$, such that they would create another part of the limiting measure supported in these boundary sides. We will show that it is not the case. One possibility to prove it would be an analysis of discontinuities of the limiting Cauchy transform on the sides of the rectangle. However, we have no more symmetry than~\eqref{eq:symmetry_cauchy_charlier} and we did not derive a formula for $C_{\mu}$ outside the rectangle (it seems possible but the rigorous justification of the assumption (v) done in Section~\ref{subsec:spm_charlier} would be even more tedious). Another and much more straightforward possibility is a direct verification that the already obtained measure $\mu_{1}+\mu_{2}$ ($\mu_{2}:=0$, if $a>\gamma(1)$) is a probability measure, which guarantees that the limiting measure has been determined completely. This is verified in the next lemma and completes the proof of Theorem~\ref{thm:charlier}.

\begin{lem}
Suppose measures $\mu_{1}$ and $\mu_{2}$ are defined as in the claims (i) and (ii) of Theorem~\ref{thm:charlier}, respectively. Then $\mu:=\mu_{1}+\mu_{2}$ is a probability measure, i.e., $\mu(\C)=1$.
\end{lem}

\begin{proof}
 Suppose first that $a\geq\gamma(1)$. Then $\mu=\mu_{1}$ and, by claim (i) of Theorem~\ref{thm:charlier}, we have
 \begin{equation}
  \mu(\C)=2\mu_{1}\!\left(\gamma([0,1])\right)=2\left[\rho(\gamma(0))-\rho(\gamma(1))\right],
 \label{eq:mu_mu_1_inproof}
 \end{equation}
 where 
 \[
  \rho(z)=\frac{1}{2\pi\ii}\left[f(\xi_{+}(z;a);z)-f(\xi_{-}(z;a);z)\right].
 \]
 Since $\gamma(0)=1+2\ii\sqrt{a}$ and $\xi_{+}(1+2\ii\sqrt{a};a)=\xi_{-}(1+2\ii\sqrt{a};a)$, one observes that 
 \begin{equation}
   \rho(\gamma(0))=0.
    \label{eq:rho_gamma_0_inproof}
 \end{equation}  
 
  Next, we evaluate $\rho(\gamma(1))$. Recall that $\gamma(1)$ is the unique solution of the equation 
  \[  
  \Re f(\xi_{+}(x;a);x)=\Re f(\xi_{-}(x;a);x),
  \]
  for $x\in(0,1)$. Consequently, we may write
  \[
   \rho(\gamma(1))=\frac{1}{2\pi}\Im\left[f(\xi_{+}(\gamma(1);a);\gamma(1))-f(\xi_{-}(\gamma(1);a);\gamma(1))\right].
  \]
  Since $\xi_{+}(x;a)>0$, for all $x\in(0,1)$, it readily follows from definition~\eqref{eq:def_f_charlier} that 
 \begin{equation}
  \Im f(\xi_{+}(\gamma(1);a);\gamma(1))=0.
 \label{eq:xi_plus_gamma_1_inproof}
 \end{equation}
 On the other hand, $-1<\xi_{-}(x;a)<0$, for $x\in(0,1)$, and hence
 \[  
  \Im f(\xi_{-}(\gamma(1);a);\gamma(1))=\pi.
 \]
 In total, we see that
 \begin{equation}
  \rho(\gamma(1))=-\frac{1}{2}.
 \label{eq:rho_gamma_1_inproof}
 \end{equation}
 Plugging~\eqref{eq:rho_gamma_0_inproof} and~\eqref{eq:rho_gamma_1_inproof} into~\eqref{eq:mu_mu_1_inproof} yields $\mu(\C)=1$.
 
 Second, suppose that $a<\gamma(1)$. Then we have a contribution from $\mu_{2}$ getting the entire mass
 \begin{equation}
  \mu(\C)=2\mu_{1}\!\left(\gamma([0,1])\right)+\mu_{2}\!\left([a,\gamma(1)]\right).
 \label{eq:mu_mu_1_mu_2_inproof}
 \end{equation}
 In this case, both equations~\eqref{eq:rho_gamma_0_inproof} and~\eqref{eq:xi_plus_gamma_1_inproof} remain valid. The difference is that, if $a<\gamma(1)$, $\xi_{-}(x;a)<-1$, for $x\in(a,\gamma(1))$, an hence we have also a contribution from the first logarithm in~\eqref{eq:def_f_charlier}, which yields
  \[  
  \Im f(\xi_{-}(\gamma(1);a);\gamma(1))=(a-\gamma(1))\pi+\pi.
 \]
 It follows that 
  \begin{equation}
  \mu_{1}\!\left(\gamma([0,1])\right)=-\rho(\gamma(1))=\frac{1}{2\pi}\Im f(\xi_{-}(\gamma(1);a);\gamma(1))=\frac{a-\gamma(1)+1}{2}.
 \label{eq:mu_1_gamma_inproof}
 \end{equation}
 
 By claim (ii) of Theorem~\ref{thm:charlier}, $\mu_{2}$ is a uniform measure on $[a,\gamma(1)]$, which implies
 \begin{equation}
 \mu_{2}\!\left([a,\gamma(1)]\right)=\gamma(1)-a.
 \label{eq:mu_2_mass_inproof}
 \end{equation}
 By plugging~\eqref{eq:mu_1_gamma_inproof} and~\eqref{eq:mu_2_mass_inproof} into~\eqref{eq:mu_mu_1_mu_2_inproof}, one again verifies that $\mu(\C)=1$.
 
\end{proof}

\subsection{A justification of the application of the saddle point method}\label{subsec:spm_charlier}

The aim of this section is a justification of the applicability of the saddle point method to the integral representation of $P_{n}^{\mathrm{C}}$ given by~\eqref{eq:integr_repre_spm_charlier} and its derivative which yields the formula for the limiting Cauchy transform from Theorem~\ref{thm:cauchy_transf_charlier}. More concretely, it is the skipped rigorous treatment of the assumption (v) of Theorem~\ref{thm:saddle-point} which is to be filled in order to Theorem~\ref{thm:cauchy_transf_charlier} be proved. To this end, we will discuss possible configurations of level curves determined by the equation
\begin{equation}
 \Re f(\xi;z)=c, \quad c\in\R,
\label{eq:re f_eq_c}
\end{equation}
in the $\xi$-plane, with $z\in(0,1)+2\ii\sqrt{a}(0,1)$ a fixed parameter. These level curves are not arbitrary. Properties of the function 
\begin{equation}
\Re f(\xi;z)=(a-\Re z)\log|1+\xi|+(\Im z)\arg(1+\xi)+\log|\xi|-a\Re\xi,
\label{eq:re_f_expanded}
\end{equation}
first of all its harmoniticity in $\C\setminus((-\infty,-1]\cup\{0\})$, imply significant constraints. For a~better picture, we list selected properties of the level curves defined by~\eqref{eq:re f_eq_c} here:
\begin{enumerate}
\item Function $\Re f(\,\cdot\,;z)$ attains all real values, i.e., the level curves given by~\eqref{eq:re f_eq_c} are defined for all $c\in\R$.
\item The level curves have intersection points if and only if $c=\Re f(\xi_{+};z)$ or $c=\Re f(\xi_{-};z)$, where $\xi_{\pm}=\xi_{\pm}(z;a)$ are the saddle points~\eqref{eq:xi_plus_minus_charlier}, and they occur at the points $\xi_{\pm}$ exclusively. Since, for $z\in(0,1)+2\ii\sqrt{a}(0,1)$, the saddle points $\xi_{\pm}$ are simple the crossings occur at the right angle.
\item For $c\neq\Re f(\xi_{\pm};z)$, the solution of~\eqref{eq:re f_eq_c} is a union of simple smooth connected curves. Among these curves, the only possible Jordan curve is located in $\C\setminus((-\infty,-1]\cup\{0\})$ with zero in its interior. Other curves are simple with end points either at $\infty$ or located on the cut $(-\infty,-1)$.
\item Finally, as it follows, for example, from a basic fact of the Morse theory, see~\cite[Chp.~1]{mat_02} or~\cite[Part~I]{mil_63}, two level curves given by equations $\Re f(\xi;z)=c_{1}$ and $\Re f(\xi;z)=c_{2}$, for $c_{1}<c_{2}$ are homotopic to each other unless the interval $(c_{1},c_{2})$ contains one of the critical values $\Re f(\xi_{\pm};z)$, i.e., $c_{1}<\Re f(\xi_{+};z)<c_{2}$ or $c_{1}<\Re f(\xi_{-};z)<c_{2}$. In other words, as $c$ varies, topological properties of the curves given by~\eqref{eq:re f_eq_c} may change only at $c=\Re f(\xi_{\pm};z)$.
\end{enumerate}

We will study the evolution of possible configurations of the level curves given by~\eqref{eq:re f_eq_c}, as $c$ varies from $-\infty$ to $\infty$, and discuss the existence of the path satisfying the assumption (v) of Theorem~\ref{thm:saddle-point} in each case. To reduce the number of possible configurations we will use the following two auxiliary results.

\begin{lem}\label{lem:crossings_imag_line_charlier}
For all $z\in(0,1)+2\ii\sqrt{a}\,(0,1)$ and $c\in\R$, the function $h:\R\setminus\{0\}\to\R$ defined by
\[
 h(t):=\Re f(\ii t;z)-c
\]
possesses exactly two zeros, one positive and one negative.
\end{lem}

\begin{proof}
 Substituting for $\xi=\ii t$ in~\eqref{eq:re_f_expanded}, one gets the expression
 \[
  h(t)=(a-\Re z)\log\sqrt{1+t^{2}}+(\Im z)\arctan(t)+\log |t|-c,
 \]
 for $t\in\R\setminus\{0\}$. By differentiating the above expression, one obtains
 \[
  h'(t)=\frac{(a+1-\Re z)t^{2}+(\Im z)t+1}{t(1+t^{2})}
 \]
 for $t\in\R\setminus\{0\}$. The assumptions $0<\Re z <1$ and $0<\Im z<2\sqrt{a}$ guarantee that the nominator of the above formula is positive for all $t\in\R$. Consequently, $\sign h'(t)=\sign t$, which means that $h$ is strictly decreasing in $(-\infty,0)$ and strictly increasing in $(0,\infty)$. Taking also into account the limit values 
 \[
  \lim_{t\to\pm\infty}h(t)=\infty \quad\mbox{ and}\quad \lim_{t\to0}h(t)=-\infty,
 \]
 we conclude that $h$ has exactly one positive and one negative zero.
\end{proof}

\begin{lem}\label{lem:crossings_the_cut_charlier}
For all $z\in(0,1)+2\ii\sqrt{a}\,(0,1)$, with $\Re z<a$, and $c\in\R$, the functions $h_{\pm}:(0,\infty)\to\R$ defined by
\[
 h_{\pm}(t):=\lim_{\substack{\xi\to-1-t \\ \Im\xi\gtrless0}}\Re(\xi;z)-c
\]
possess at most two zeros each.
\end{lem}

\begin{proof}
 It follows from~\eqref{eq:re_f_expanded} that 
 \[
  h_{\pm}(t)=(a-\Re z)\log t\pm\pi\Im z+\log(1+t)+a(1+t)-c,
 \]
 which, if differentiated with respect to $t$, yields
 \[
 h_{\pm}'(t)=\frac{at^{2}+(2a+1-\Re z)t-\Re z}{t(1+t)},
 \]
 for $t>0$. One readily checks that the quadratic polynomial from the nominator has one positive and one negative root. Moreover, taking into account the assumption $\Re z<a$, one gets the limit values
 \[
  \lim_{t\to\infty}h_{\pm}(t)=\lim_{t\to0}h_{\pm}(t)=\infty.
 \]
 Consequently, both functions $h_{\pm}$ first decay from $\infty$ to their unique global minima and then grow to $\infty$, which implies the statement.
\end{proof}

In the forthcoming discussion, we denote the contour integral from~\eqref{eq:integr_repre_spm_charlier} by
\[
 I_{n}(z;a):=\oint_{\gamma}g(\xi)e^{-nf(\xi;z)}\dd\xi
\]
and the contour integral of its derivative with respect to $z$ by  
\[
I_{n}'(z;a):=\oint_{\gamma}\left(g(\xi)\,\frac{\partial f}{\partial z}(\xi;z)\right)e^{-nf(\xi;z)}\dd\xi.
\]
Then, by~\eqref{eq:cauchy_transf_root_count} and~\eqref{eq:integr_repre_spm_charlier}, the Cauchy transform of the root counting measure $\mu_{n}$ is the ratio 
\begin{equation}
 C_{\mu_{n}}(z)=-\frac{I_{n}'(z;a)}{I_{n}(z;a)}.
\label{eq:C_mu_n_I_n_charlier}
\end{equation}
For $z\in(0,1)+2\ii\sqrt{a}\,(0,1)$ given, we distinguish three cases:
\[
\mbox{A) } a>\Re z, \quad\;
\mbox{B) } a<\Re z, \quad\;
\mbox{C) } a=\Re z.
\]

A) \textbf{The case $a>\Re z$:}
Recall that our intention is the study of an evolution of the level curves given by~\eqref{eq:re f_eq_c} as $c$ varies from $-\infty$ to $\infty$.  Since $a>\Re z$ we have the limit values
\[
 \lim_{\xi\to-1} \Re f(\xi;z)=\lim_{\xi\to0} \Re f(\xi;z)=\lim_{\Re\xi\to\infty} \Re f(\xi;z)=-\infty
\]
and
\[
 \lim_{\Re\xi\to-\infty} \Re f(\xi;z)=\infty
\]
that are determined by the first, third, and forth term on the right-hand side of~\eqref{eq:re_f_expanded}. 

By inspection of~\eqref{eq:re_f_expanded}, we see that, for $c$ very small, i.e., $c\to-\infty$, solutions of~\eqref{eq:re f_eq_c} consist of three components which are small perturbations of the following curves: a spiral-like curve encircling the point $-1$, a circle centered at $0$, and a vertical line intersecting the point $-c/a$. These three limit components are denoted by $\ell_{-1}$, $\ell_{0}$, and $\ell_{\infty}$, respectively. Indeed, for $c\to-\infty$, the solutions of 
\[
 \Re f(\xi;z)=\Re\left[(a-z)\log(1+\xi)\right]+\log|\xi|-a\Re\xi=c
\]
is approximately determined by one of the terms depending on whether $\xi$ is close to $-1$, $0$ or of a large real part. The remaining terms can be neglected as a small perturbation. Thus, $\ell_{-1}$ is a curve given approximately by the equation
\[
 \Re\left[(a-z)\log(1+\xi)\right]=c,
\]
for $c\to-\infty$, which is equivalent to the equation
\[
 (a-\Re z)\log r+(\Im z)\phi=c,
\]
for $\xi=-1+r\exp(\ii\phi)$. Hence, as $c\to\-\infty$, $\ell_{-1}$ is close to the logarithmic spiral around the point $-1$ given in polar coordinates by the equation
\[
 r=\exp\left(\frac{c-(\Im z)\phi}{a-\Re z}\right).
\]
Similarly, for $c$ small, $\ell_{0}$ approaches the circle given by the equation $\log|\xi|=c$, i.e., $|\xi|=e^{c}$. Lastly, $\ell_{\infty}$ is, for $c\to-\infty$, close to the vertical line given by equation $-a\Re\xi=c$. The schematic picture of the solutions of $\Re f(z;\xi)=c$ ,for $c$ approaching $-\infty$, is depicted in Figure~\ref{fig:conf_c<<0_charlier}.
\begin{figure}[htb]
\begin{center}

\begin{tikzpicture}[x=0.75pt,y=0.75pt,yscale=-1,xscale=1]
\clip   (32.8,6) rectangle (571.8,228) ;
\draw  [color={rgb, 255:red, 65; green, 117; blue, 5 }  ,draw opacity=1 ][fill={rgb, 255:red, 65; green, 117; blue, 5 }  ,fill opacity=0.5 ][line width=2.25]  (183.8,101) .. controls (190.8,67) and (237.8,82) .. (239.8,103) .. controls (244.8,137) and (182.8,156) .. (167.8,104) .. controls (145.8,103) and (209.8,103) .. (183.8,101) -- cycle ;
\draw  [color={rgb, 255:red, 208; green, 2; blue, 27 }  ,draw opacity=1 ][fill={rgb, 255:red, 208; green, 2; blue, 27 }  ,fill opacity=0.5 ][line width=2.25]  (327.8,104) .. controls (332.8,50) and (415.8,66) .. (420.8,104) .. controls (425.8,142) and (322.8,158) .. (327.8,104) -- cycle ;
\draw    (110.8,103) -- (660.8,103) ;
\draw [color={rgb, 255:red, 139; green, 87; blue, 42 }  ,draw opacity=1 ][line width=2.25]    (25.8,102) .. controls (27.47,100.35) and (29.14,100.36) .. (30.8,102.03) .. controls (32.46,103.7) and (34.13,103.71) .. (35.8,102.06) .. controls (37.47,100.41) and (39.14,100.42) .. (40.8,102.09) .. controls (42.46,103.76) and (44.13,103.77) .. (45.8,102.12) .. controls (47.47,100.46) and (49.14,100.47) .. (50.8,102.14) .. controls (52.46,103.81) and (54.13,103.82) .. (55.8,102.17) .. controls (57.47,100.52) and (59.14,100.53) .. (60.8,102.2) .. controls (62.46,103.87) and (64.13,103.88) .. (65.8,102.23) .. controls (67.47,100.58) and (69.14,100.59) .. (70.8,102.26) .. controls (72.46,103.93) and (74.13,103.94) .. (75.8,102.29) .. controls (77.47,100.64) and (79.14,100.65) .. (80.8,102.32) .. controls (82.46,103.99) and (84.13,104) .. (85.8,102.35) .. controls (87.47,100.7) and (89.14,100.71) .. (90.8,102.38) .. controls (92.46,104.05) and (94.13,104.06) .. (95.8,102.4) .. controls (97.47,100.75) and (99.14,100.76) .. (100.8,102.43) .. controls (102.46,104.1) and (104.13,104.11) .. (105.8,102.46) .. controls (107.47,100.81) and (109.14,100.82) .. (110.8,102.49) .. controls (112.46,104.16) and (114.13,104.17) .. (115.8,102.52) .. controls (117.47,100.87) and (119.14,100.88) .. (120.8,102.55) .. controls (122.46,104.22) and (124.13,104.23) .. (125.8,102.58) .. controls (127.47,100.93) and (129.14,100.94) .. (130.8,102.61) .. controls (132.46,104.28) and (134.13,104.29) .. (135.8,102.64) .. controls (137.47,100.98) and (139.14,100.99) .. (140.8,102.66) .. controls (142.46,104.33) and (144.13,104.34) .. (145.8,102.69) .. controls (147.47,101.04) and (149.14,101.05) .. (150.8,102.72) .. controls (152.46,104.39) and (154.13,104.4) .. (155.8,102.75) .. controls (157.47,101.1) and (159.14,101.11) .. (160.8,102.78) .. controls (162.46,104.45) and (164.13,104.46) .. (165.8,102.81) .. controls (167.47,101.16) and (169.14,101.17) .. (170.8,102.84) .. controls (172.46,104.51) and (174.13,104.52) .. (175.8,102.87) .. controls (177.47,101.22) and (179.14,101.23) .. (180.8,102.9) .. controls (182.46,104.57) and (184.13,104.58) .. (185.8,102.92) .. controls (187.47,101.27) and (189.14,101.28) .. (190.8,102.95) .. controls (192.46,104.62) and (194.13,104.63) .. (195.8,102.98) -- (198.8,103) -- (198.8,103) ;
\draw    (360.8,12) -- (359.8,227) ;
\draw    (199.8,96) -- (199.8,110) ;
\draw  [color={rgb, 255:red, 74; green, 144; blue, 226 }  ,draw opacity=1 ][fill={rgb, 255:red, 74; green, 144; blue, 226 }  ,fill opacity=0.51 ][line width=2.25]  (534.8,2) .. controls (558.8,-23) and (626.8,-35) .. (658.8,0) .. controls (659.8,43) and (660.8,190) .. (659.8,232) .. controls (645.8,279) and (548.8,260) .. (525.8,233) .. controls (511.8,194) and (483.8,161) .. (478.8,118) .. controls (477.8,81) and (519.8,46) .. (534.8,2) -- cycle ;

\draw (186,115.4) node [anchor=north west][inner sep=0.75pt]    {$-1$};
\draw (342,111.4) node [anchor=north west][inner sep=0.75pt]    {$0$};
\draw (217,147.4) node [anchor=north west][inner sep=0.75pt]  [color={rgb, 255:red, 65; green, 117; blue, 5 }  ,opacity=1 ]  {$\ell _{-1}$};
\draw (394,142.4) node [anchor=north west][inner sep=0.75pt]  [color={rgb, 255:red, 208; green, 2; blue, 27 }  ,opacity=1 ]  {$\ell _{0}$};
\draw (54,21.4) node [anchor=north west][inner sep=0.75pt]    {$\Re f( \xi;z) =c\sim -\infty $};
\draw (460,149.4) node [anchor=north west][inner sep=0.75pt]  [color={rgb, 255:red, 74; green, 144; blue, 226 }  ,opacity=1 ]  {$\ell _{\infty }$};
\end{tikzpicture}
\end{center}
\caption{The topological configuration of the level curves~$\ell_{-1}$, $\ell_{0}$, and $\ell_{\infty}$ when $c$ approaches $-\infty$ and $\Re z<a$. The filled regions are the regions, where the value of $\Re f(\xi,z)$ is less than the level~$c$.} \label{fig:conf_c<<0_charlier}
\end{figure}
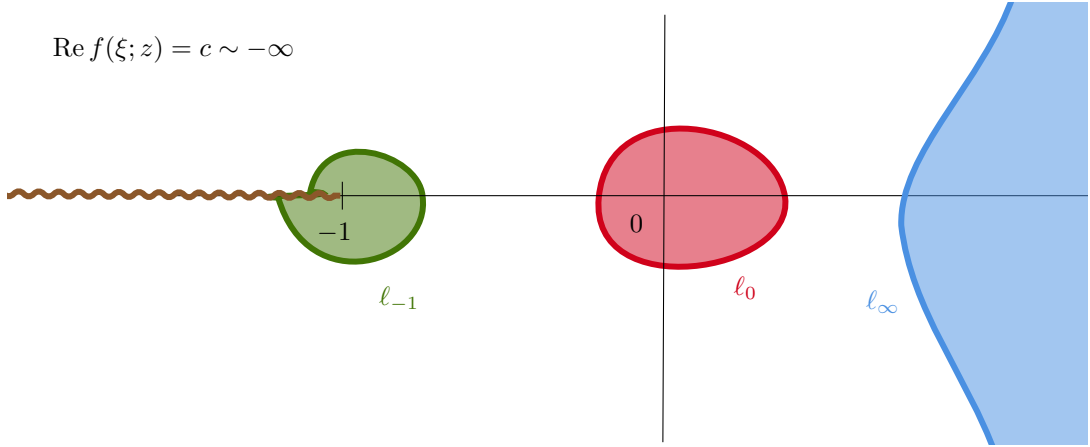

By a similar reasoning, we see that, for $c$ large, the level curve $\Re f(\xi;z)=c$ is approximately the vertical line $\ell_{-\infty}$ determined by the equation $-a\Re\xi=c$ (interrupted by the cut $(-\infty,-1]$).

Suppose $z\in\Omega_{-}$, which means that $\Re f(\xi_{+};a)<\Re f(\xi_{-};a)$, see~\eqref{eq:def_omega_pm}. Then, as $c$ varies from $-\infty$ to $\infty$, the three curves $\ell_{-1}$, $\ell_{0}$, and $\ell_{\infty}$ evolves without intersecting as long as $c<\Re f(\xi_{+};z)$. At $c=\Re f(\xi_{+};z)$, two of components of $\ell_{-1}$, $\ell_{0}$, and $\ell_{\infty}$ merge at the point $\xi=\xi_{+}$. As $c$ further increases, the two remaining components merge when $c=\Re f(\xi_{-};z)$ at $\xi=\xi_{-}$ to form the component $\ell_{-\infty}$. Clearly, if $z\in\Omega_{+}$, the situation is analogous with the roles of $\xi_{+}$ and $\xi_{-}$ interchanged. Next, we discuss the three possible combinations of which components of $\ell_{-1}$, $\ell_{0}$, and $\ell_{\infty}$ merge first.
a) \textbf{Components $\ell_{-1}$ and $\ell_{0}$ merge first:} 
Recall that we investigate whether the configuration admits the existence of a Jordan curve $\gamma$ with $0$ in its interior, located in $\C\setminus\left((-\infty,-1]\cup\{0\}\right)$, and such that $\Re f(\xi;z)>\Re f(\xi_{+};z)$, for all $\xi$ on the image of $\gamma$ except at $\xi=\xi_{+}$, which is exactly the assumption~(v) of Theorem~\ref{thm:saddle-point}. This is clearly possible in this configuration, see Figure~\ref{fig:conf_ell0_ell-1_merge_charlier} for an illustration. Hence the saddle point method is readily applicable to both $I_{n}(z;a)$ as well as $I_{n}'(z;a)$ and Theorem~\ref{thm:saddle-point} provides us with the asymptotic formulas
\[
 I_{n}(z;a)=g(\xi_{+})e^{-nf(\xi_{+};z)}\sqrt{\frac{2\pi}{n\partial_{\xi}^{2}f(\xi_{+};z)}}\left(1+O\left(\frac{1}{\sqrt{n}}\right)\!\right)\!, \quad \mbox{ for } n\to\infty,
\]
and 
\[
I_{n}'(z;a)=g(\xi_{+})(\partial_{z}f(\xi_{+};z))e^{-nf(\xi_{+};z)}\sqrt{\frac{2\pi}{n\partial_{\xi}^{2}f(\xi_{+};z)}}\left(1+O\left(\frac{1}{\sqrt{n}}\right)\!\right)\!, \quad \mbox{ for } n\to\infty.
\]
It follows that
\begin{equation}
 \lim_{n\to\infty}-\frac{I_{n}(z;a)}{I_{n}'(z;a)}=-\frac{\partial f}{\partial z}(\xi_{+};z),
\label{eq:ratio_I_n_asympt}
\end{equation}
which, together with~\eqref{eq:C_mu_n_I_n_charlier}, implies the desired formula of Theorem~\ref{thm:cauchy_transf_charlier}.
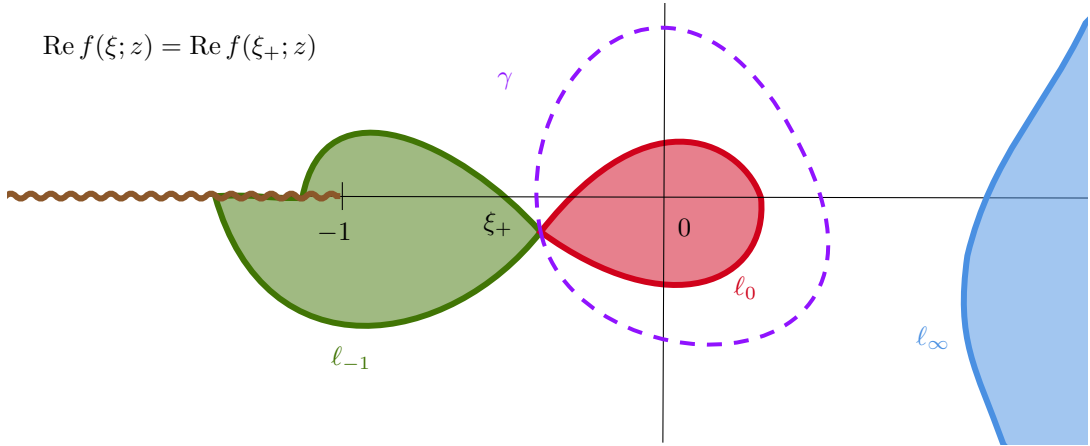
\begin{figure}[htb]
\begin{center}
\begin{tikzpicture}[x=0.75pt,y=0.75pt,yscale=-1,xscale=1]
\clip   (32.8,6) rectangle (571.8,228) ;
\draw  [color={rgb, 255:red, 65; green, 117; blue, 5 }  ,draw opacity=1 ][fill={rgb, 255:red, 65; green, 117; blue, 5 }  ,fill opacity=0.5 ][line width=2.25]  (179.8,104) .. controls (190.8,44) and (258.8,76) .. (298.8,121) .. controls (252.8,180) and (161.8,194) .. (136.3,103) .. controls (162.8,102) and (152.8,104) .. (179.8,104) -- cycle ;
\draw  [color={rgb, 255:red, 208; green, 2; blue, 27 }  ,draw opacity=1 ][fill={rgb, 255:red, 208; green, 2; blue, 27 }  ,fill opacity=0.5 ][line width=2.25]  (298.8,121) .. controls (354.8,49) and (400.8,78) .. (408.8,104) .. controls (412.8,147) and (362.8,168) .. (298.8,121) -- cycle ;
\draw    (110.8,103) -- (587.8,104) ;
\draw [color={rgb, 255:red, 139; green, 87; blue, 42 }  ,draw opacity=1 ][line width=2.25]    (21.8,103) .. controls (23.47,101.33) and (25.13,101.33) .. (26.8,103) .. controls (28.47,104.67) and (30.13,104.67) .. (31.8,103) .. controls (33.47,101.33) and (35.13,101.33) .. (36.8,103) .. controls (38.47,104.67) and (40.13,104.67) .. (41.8,103) .. controls (43.47,101.33) and (45.13,101.33) .. (46.8,103) .. controls (48.47,104.67) and (50.13,104.67) .. (51.8,103) .. controls (53.47,101.33) and (55.13,101.33) .. (56.8,103) .. controls (58.47,104.67) and (60.13,104.67) .. (61.8,103) .. controls (63.47,101.33) and (65.13,101.33) .. (66.8,103) .. controls (68.47,104.67) and (70.13,104.67) .. (71.8,103) .. controls (73.47,101.33) and (75.13,101.33) .. (76.8,103) .. controls (78.47,104.67) and (80.13,104.67) .. (81.8,103) .. controls (83.47,101.33) and (85.13,101.33) .. (86.8,103) .. controls (88.47,104.67) and (90.13,104.67) .. (91.8,103) .. controls (93.47,101.33) and (95.13,101.33) .. (96.8,103) .. controls (98.47,104.67) and (100.13,104.67) .. (101.8,103) .. controls (103.47,101.33) and (105.13,101.33) .. (106.8,103) .. controls (108.47,104.67) and (110.13,104.67) .. (111.8,103) .. controls (113.47,101.33) and (115.13,101.33) .. (116.8,103) .. controls (118.47,104.67) and (120.13,104.67) .. (121.8,103) .. controls (123.47,101.33) and (125.13,101.33) .. (126.8,103) .. controls (128.47,104.67) and (130.13,104.67) .. (131.8,103) .. controls (133.47,101.33) and (135.13,101.33) .. (136.8,103) .. controls (138.47,104.67) and (140.13,104.67) .. (141.8,103) .. controls (143.47,101.33) and (145.13,101.33) .. (146.8,103) .. controls (148.47,104.67) and (150.13,104.67) .. (151.8,103) .. controls (153.47,101.33) and (155.13,101.33) .. (156.8,103) .. controls (158.47,104.67) and (160.13,104.67) .. (161.8,103) .. controls (163.47,101.33) and (165.13,101.33) .. (166.8,103) .. controls (168.47,104.67) and (170.13,104.67) .. (171.8,103) .. controls (173.47,101.33) and (175.13,101.33) .. (176.8,103) .. controls (178.47,104.67) and (180.13,104.67) .. (181.8,103) .. controls (183.47,101.33) and (185.13,101.33) .. (186.8,103) .. controls (188.47,104.67) and (190.13,104.67) .. (191.8,103) .. controls (193.47,101.33) and (195.13,101.33) .. (196.8,103) -- (198.8,103) -- (198.8,103) ;
\draw    (360.8,4) -- (359.8,227) ;
\draw    (199.8,96) -- (199.8,110) ;
\draw  [color={rgb, 255:red, 74; green, 144; blue, 226 }  ,draw opacity=1 ][fill={rgb, 255:red, 74; green, 144; blue, 226 }  ,fill opacity=0.51 ][line width=2.25]  (570.8,16) .. controls (594.8,-9) and (626.8,-35) .. (658.8,0) .. controls (659.8,43) and (660.8,190) .. (659.8,232) .. controls (645.8,279) and (552.8,255) .. (529.8,228) .. controls (515.8,189) and (505.8,176) .. (511.8,133) .. controls (523.8,83) and (549.8,60) .. (570.8,16) -- cycle ;

\draw  [color={rgb, 255:red, 144; green, 19; blue, 254 }  ,draw opacity=1 ][dash pattern={on 5.63pt off 4.5pt}][line width=1.5]  (298.8,121) .. controls (315.8,200) and (489.8,205) .. (429.8,81) .. controls (369.8,-43) and (281.8,42) .. (298.8,121) -- cycle ;

\draw (186,115.4) node [anchor=north west][inner sep=0.75pt]    {$-1$};
\draw (366,113.4) node [anchor=north west][inner sep=0.75pt]    {$0$};
\draw (193,177.4) node [anchor=north west][inner sep=0.75pt]  [color={rgb, 255:red, 65; green, 117; blue, 5 }  ,opacity=1 ]  {$\ell _{-1}$};
\draw (394,142.4) node [anchor=north west][inner sep=0.75pt]  [color={rgb, 255:red, 208; green, 2; blue, 27 }  ,opacity=1 ]  {$\ell _{0}$};
\draw (49,20.4) node [anchor=north west][inner sep=0.75pt]    {$\Re f( \xi ;z) =\Re f( \xi _{+} ;z)$};
\draw (485,167.4) node [anchor=north west][inner sep=0.75pt]  [color={rgb, 255:red, 74; green, 144; blue, 226 }  ,opacity=1 ]  {$\ell _{\infty }$};
\draw (269,109.4) node [anchor=north west][inner sep=0.75pt]  [color={rgb, 255:red, 0; green, 0; blue, 0 }  ,opacity=1 ]  {$\xi _{+}$};
\draw (276,39.4) node [anchor=north west][inner sep=0.75pt]  [color={rgb, 255:red, 65; green, 117; blue, 5 }  ,opacity=1 ]  {$\textcolor[rgb]{0.56,0.07,1}{\gamma }$};
\end{tikzpicture}
\end{center}
\caption{The topological configuration when $\ell_{-1}$ and $\ell_{0}$ merge first. The dashed curve is a possible choice of the Jordan curve $\gamma$ satisfying the assumption~(v) of Theorem~\ref{thm:saddle-point}.} \label{fig:conf_ell0_ell-1_merge_charlier}
\end{figure}
b) \textbf{Components $\ell_{0}$ and $\ell_{\infty}$ merge first:} This configuration is completely analogous to the previous one, justifies the applicability of the saddle point method, and results in the same asymptotic formulas for $I_{n}(z;a)$ and $I_{n}'(z;a)$. An illustration is depicted in Figure~\ref{fig:conf_ell0_ellinf_merge_charlier}.

\begin{figure}[htb]
\begin{center}
\begin{tikzpicture}[x=0.75pt,y=0.75pt,yscale=-1,xscale=1]
\clip   (32.8,6) rectangle (571.8,228) ;
\draw  [color={rgb, 255:red, 65; green, 117; blue, 5 }  ,draw opacity=1 ][fill={rgb, 255:red, 65; green, 117; blue, 5 }  ,fill opacity=0.5 ][line width=2.25]  (179.8,104) .. controls (188.8,52) and (252.8,78) .. (264.8,109) .. controls (272.8,138) and (177.3,194) .. (151.8,103) .. controls (182.8,101) and (152.8,104) .. (179.8,104) -- cycle ;
\draw  [color={rgb, 255:red, 208; green, 2; blue, 27 }  ,draw opacity=1 ][fill={rgb, 255:red, 208; green, 2; blue, 27 }  ,fill opacity=0.5 ][line width=2.25]  (316.8,110) .. controls (329.8,44) and (413.8,64) .. (449.8,119) .. controls (399.8,162) and (327.8,165) .. (316.8,110) -- cycle ;
\draw    (110.8,103) -- (587.8,104) ;
\draw [color={rgb, 255:red, 139; green, 87; blue, 42 }  ,draw opacity=1 ][line width=2.25]    (17.8,102) .. controls (19.47,100.35) and (21.14,100.36) .. (22.8,102.03) .. controls (24.46,103.7) and (26.13,103.71) .. (27.8,102.06) .. controls (29.47,100.4) and (31.14,100.41) .. (32.8,102.08) .. controls (34.46,103.75) and (36.13,103.76) .. (37.8,102.11) .. controls (39.47,100.46) and (41.14,100.47) .. (42.8,102.14) .. controls (44.46,103.81) and (46.13,103.82) .. (47.8,102.17) .. controls (49.47,100.51) and (51.14,100.52) .. (52.8,102.19) .. controls (54.46,103.86) and (56.13,103.87) .. (57.8,102.22) .. controls (59.47,100.57) and (61.14,100.58) .. (62.8,102.25) .. controls (64.46,103.92) and (66.13,103.93) .. (67.8,102.28) .. controls (69.47,100.62) and (71.14,100.63) .. (72.8,102.3) .. controls (74.46,103.97) and (76.13,103.98) .. (77.8,102.33) .. controls (79.47,100.68) and (81.14,100.69) .. (82.8,102.36) .. controls (84.46,104.03) and (86.13,104.04) .. (87.8,102.39) .. controls (89.47,100.73) and (91.14,100.74) .. (92.8,102.41) .. controls (94.46,104.08) and (96.13,104.09) .. (97.8,102.44) .. controls (99.47,100.79) and (101.14,100.8) .. (102.8,102.47) .. controls (104.46,104.14) and (106.13,104.15) .. (107.8,102.5) .. controls (109.47,100.84) and (111.14,100.85) .. (112.8,102.52) .. controls (114.46,104.19) and (116.13,104.2) .. (117.8,102.55) .. controls (119.47,100.9) and (121.14,100.91) .. (122.8,102.58) .. controls (124.46,104.25) and (126.13,104.26) .. (127.8,102.61) .. controls (129.47,100.96) and (131.14,100.97) .. (132.8,102.64) .. controls (134.46,104.31) and (136.13,104.32) .. (137.8,102.66) .. controls (139.47,101.01) and (141.14,101.02) .. (142.8,102.69) .. controls (144.46,104.36) and (146.13,104.37) .. (147.8,102.72) .. controls (149.47,101.07) and (151.14,101.08) .. (152.8,102.75) .. controls (154.46,104.42) and (156.13,104.43) .. (157.8,102.77) .. controls (159.47,101.12) and (161.14,101.13) .. (162.8,102.8) .. controls (164.46,104.47) and (166.13,104.48) .. (167.8,102.83) .. controls (169.47,101.18) and (171.14,101.19) .. (172.8,102.86) .. controls (174.46,104.53) and (176.13,104.54) .. (177.8,102.88) .. controls (179.47,101.23) and (181.14,101.24) .. (182.8,102.91) .. controls (184.46,104.58) and (186.13,104.59) .. (187.8,102.94) .. controls (189.47,101.29) and (191.14,101.3) .. (192.8,102.97) .. controls (194.46,104.64) and (196.13,104.65) .. (197.8,102.99) -- (198.8,103) -- (198.8,103) ;
\draw    (360.8,4) -- (359.8,227) ;
\draw    (199.8,96) -- (199.8,110) ;
\draw  [color={rgb, 255:red, 74; green, 144; blue, 226 }  ,draw opacity=1 ][fill={rgb, 255:red, 74; green, 144; blue, 226 }  ,fill opacity=0.51 ][line width=2.25]  (541.8,8) .. controls (565.8,-17) and (626.8,-35) .. (658.8,0) .. controls (659.8,43) and (660.8,190) .. (659.8,232) .. controls (645.8,279) and (536.8,257) .. (513.8,230) .. controls (493.8,191) and (497.8,182) .. (449.8,119) .. controls (509.8,71) and (509.8,44) .. (541.8,8) -- cycle ;

\draw  [color={rgb, 255:red, 144; green, 19; blue, 254 }  ,draw opacity=1 ][dash pattern={on 5.63pt off 4.5pt}][line width=1.5]  (298.8,121) .. controls (302.8,222) and (436.8,200) .. (449.8,119) .. controls (462.8,38) and (294.8,20) .. (298.8,121) -- cycle ;

\draw (186,115.4) node [anchor=north west][inner sep=0.75pt]    {$-1$};
\draw (366,113.4) node [anchor=north west][inner sep=0.75pt]    {$0$};
\draw (189,154.4) node [anchor=north west][inner sep=0.75pt]  [color={rgb, 255:red, 65; green, 117; blue, 5 }  ,opacity=1 ]  {$\ell _{-1}$};
\draw (366,152.4) node [anchor=north west][inner sep=0.75pt]  [color={rgb, 255:red, 208; green, 2; blue, 27 }  ,opacity=1 ]  {$\ell _{0}$};
\draw (52,16.4) node [anchor=north west][inner sep=0.75pt]    {$\Re f( \xi ;z) =\Re f( \xi _{+};z)$};
\draw (477,198.4) node [anchor=north west][inner sep=0.75pt]  [color={rgb, 255:red, 74; green, 144; blue, 226 }  ,opacity=1 ]  {$\ell _{\infty }$};
\draw (468,115.4) node [anchor=north west][inner sep=0.75pt]  [color={rgb, 255:red, 0; green, 0; blue, 0 }  ,opacity=1 ]  {$\xi _{+}$};
\draw (302,43.4) node [anchor=north west][inner sep=0.75pt]  [color={rgb, 255:red, 65; green, 117; blue, 5 }  ,opacity=1 ]  {$\textcolor[rgb]{0.56,0.07,1}{\gamma }$};
\end{tikzpicture}
\end{center}
\caption{The topological configuration when $\ell_{0}$ and $\ell_{\infty}$ merge first. The dashed curve is a possible choice of the Jordan curve $\gamma$ satisfying the assumption~(v) of Theorem~\ref{thm:saddle-point}.} \label{fig:conf_ell0_ellinf_merge_charlier}
\end{figure}
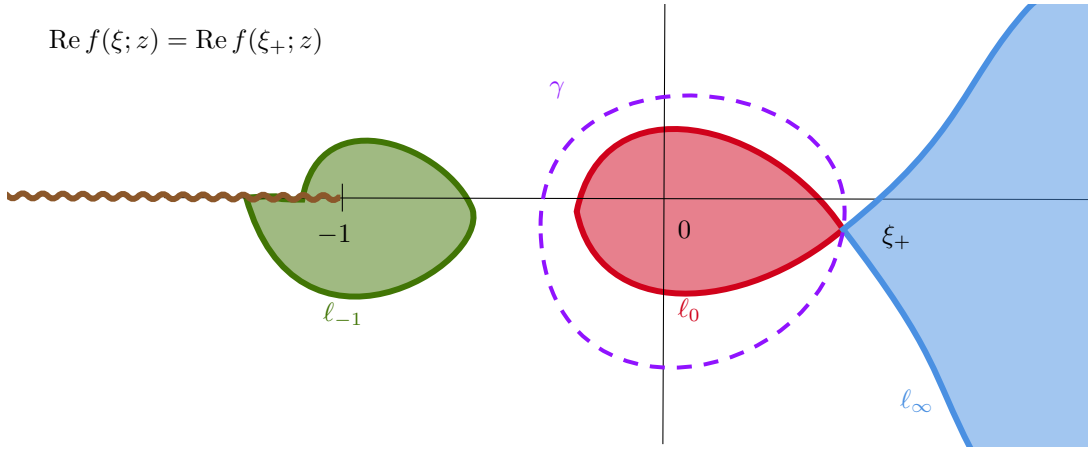
c) \textbf{Components $\ell_{-1}$ and $\ell_{\infty}$ merge first:} 
This configuration is impossible since more than two crossings of the level curves with the imaginary line are inevitable, however, as it follows from Lemma~\ref{lem:crossings_imag_line_charlier}, exactly two crossings occur. The situation is illustrated in Figure~\ref{fig:conf_ell-1_ellinf_merge_charlier}.
\begin{figure}[htb]
\begin{center}
\begin{tikzpicture}[x=0.75pt,y=0.75pt,yscale=-1,xscale=1]
\clip   (32.8,6) rectangle (571.8,228) ;
\draw  [color={rgb, 255:red, 65; green, 117; blue, 5 }  ,draw opacity=1 ][fill={rgb, 255:red, 65; green, 117; blue, 5 }  ,fill opacity=0.5 ][line width=2.25]  (337.93,75.93) .. controls (320.27,109.6) and (289.27,117.93) .. (252.6,150.6) .. controls (217.27,190.6) and (180.6,207.27) .. (182.47,195.33) .. controls (97.27,193.27) and (96.6,194.6) .. (25.27,193.27) .. controls (22.6,155.27) and (-8.89,112.77) .. (30.6,93.93) .. controls (117.93,17.27) and (236.6,9.93) .. (337.93,75.93) -- cycle ;
\draw  [color={rgb, 255:red, 208; green, 2; blue, 27 }  ,draw opacity=1 ][fill={rgb, 255:red, 208; green, 2; blue, 27 }  ,fill opacity=0.5 ][line width=2.25]  (341.27,209.27) .. controls (313.93,203.93) and (318.6,155.27) .. (343.93,151.93) .. controls (369.27,148.6) and (368.6,214.6) .. (341.27,209.27) -- cycle ;
\draw    (107.8,196) -- (584.8,197) ;
\draw [color={rgb, 255:red, 139; green, 87; blue, 42 }  ,draw opacity=1 ][line width=2.25]    (8.8,195) .. controls (10.47,193.34) and (12.13,193.34) .. (13.8,195.01) .. controls (15.47,196.68) and (17.13,196.68) .. (18.8,195.02) .. controls (20.47,193.36) and (22.13,193.36) .. (23.8,195.03) .. controls (25.47,196.7) and (27.13,196.7) .. (28.8,195.04) .. controls (30.47,193.38) and (32.13,193.38) .. (33.8,195.05) .. controls (35.47,196.72) and (37.13,196.72) .. (38.8,195.06) .. controls (40.47,193.4) and (42.13,193.4) .. (43.8,195.07) .. controls (45.47,196.74) and (47.13,196.74) .. (48.8,195.08) .. controls (50.47,193.42) and (52.13,193.42) .. (53.8,195.09) .. controls (55.47,196.76) and (57.13,196.76) .. (58.8,195.1) .. controls (60.47,193.44) and (62.13,193.44) .. (63.8,195.11) .. controls (65.47,196.78) and (67.13,196.78) .. (68.8,195.12) .. controls (70.47,193.45) and (72.13,193.45) .. (73.8,195.12) .. controls (75.47,196.79) and (77.13,196.79) .. (78.8,195.13) .. controls (80.47,193.47) and (82.13,193.47) .. (83.8,195.14) .. controls (85.47,196.81) and (87.13,196.81) .. (88.8,195.15) .. controls (90.47,193.49) and (92.13,193.49) .. (93.8,195.16) .. controls (95.47,196.83) and (97.13,196.83) .. (98.8,195.17) .. controls (100.47,193.51) and (102.13,193.51) .. (103.8,195.18) .. controls (105.47,196.85) and (107.13,196.85) .. (108.8,195.19) .. controls (110.47,193.53) and (112.13,193.53) .. (113.8,195.2) .. controls (115.47,196.87) and (117.13,196.87) .. (118.8,195.21) .. controls (120.47,193.55) and (122.13,193.55) .. (123.8,195.22) .. controls (125.47,196.89) and (127.13,196.89) .. (128.8,195.23) .. controls (130.47,193.57) and (132.13,193.57) .. (133.8,195.24) .. controls (135.47,196.91) and (137.13,196.91) .. (138.8,195.25) .. controls (140.47,193.59) and (142.13,193.59) .. (143.8,195.26) .. controls (145.47,196.93) and (147.13,196.93) .. (148.8,195.27) .. controls (150.47,193.61) and (152.13,193.61) .. (153.8,195.28) .. controls (155.47,196.95) and (157.13,196.95) .. (158.8,195.29) .. controls (160.47,193.63) and (162.13,193.63) .. (163.8,195.3) .. controls (165.47,196.97) and (167.13,196.97) .. (168.8,195.31) .. controls (170.47,193.65) and (172.13,193.65) .. (173.8,195.32) .. controls (175.47,196.99) and (177.13,196.99) .. (178.8,195.33) -- (182.47,195.33) -- (182.47,195.33) ;
\draw    (342.8,6) -- (341.8,229) ;
\draw    (183.47,187.67) -- (183.47,201.67) ;
\draw  [color={rgb, 255:red, 74; green, 144; blue, 226 }  ,draw opacity=1 ][fill={rgb, 255:red, 74; green, 144; blue, 226 }  ,fill opacity=0.51 ][line width=2.25]  (365.27,5.93) .. controls (382.6,-45.27) and (626.8,-35) .. (658.8,0) .. controls (659.8,43) and (660.8,190) .. (659.8,232) .. controls (645.8,279) and (595.6,237.73) .. (572.6,210.73) .. controls (545.6,179.73) and (497.68,140.5) .. (459.93,121.4) .. controls (409.93,96.73) and (394.6,103.33) .. (337.93,75.93) .. controls (356.6,31.4) and (357.27,34.6) .. (365.27,5.93) -- cycle ;

\draw (154,204.73) node [anchor=north west][inner sep=0.75pt]    {$-1$};
\draw (359.33,206.07) node [anchor=north west][inner sep=0.75pt]    {$0$};
\draw (237.67,173.73) node [anchor=north west][inner sep=0.75pt]  [color={rgb, 255:red, 65; green, 117; blue, 5 }  ,opacity=1 ]  {$\ell _{-1}$};
\draw (366,152.4) node [anchor=north west][inner sep=0.75pt]  [color={rgb, 255:red, 208; green, 2; blue, 27 }  ,opacity=1 ]  {$\ell _{0}$};
\draw (38.67,9.07) node [anchor=north west][inner sep=0.75pt]    {$\Re f( \xi;z) =\Re f( \xi _{+};z)$};
\draw (497.67,161.07) node [anchor=north west][inner sep=0.75pt]  [color={rgb, 255:red, 74; green, 144; blue, 226 }  ,opacity=1 ]  {$\ell _{\infty }$};
\draw (317.33,42.07) node [anchor=north west][inner sep=0.75pt]  {$\xi _{+}$};
\end{tikzpicture}
\end{center}
\caption{The topological configuration when $\ell_{-1}$ and $\ell_{\infty}$ merge first. There are necessarily more than two crossings of the level curves with the imaginary line.} \label{fig:conf_ell-1_ellinf_merge_charlier}
\end{figure}
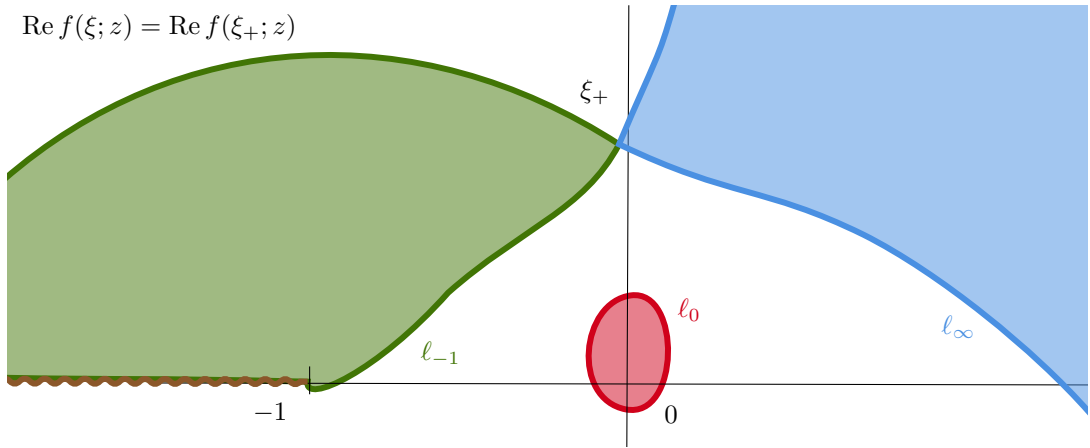

B) \textbf{The case $a<\Re z$:} When $a<\Re z$, one infers from~\eqref{eq:re_f_expanded} the limit relations
\[
 \lim_{\xi\to0} \Re f(\xi;z)=\lim_{\Re\xi\to\infty} \Re f(\xi;z)=-\infty
\]
and
\[
 \lim_{\xi\to-1} \Re f(\xi;z)=\lim_{\Re\xi\to-\infty} \Re f(\xi;z)=\infty.
\]
Therefore, as $c\to-\infty$, the level curves given by the equation~\eqref{eq:re f_eq_c} consist of just two components $\ell_{0}$ and $\ell_{\infty}$, see Figure~\ref{fig:conf-c<<0_case2_charlier}.
\begin{figure}[htb]
\begin{center}
\begin{tikzpicture}[x=0.75pt,y=0.75pt,yscale=-1,xscale=1]
\clip   (32.8,6) rectangle (571.8,228);

\draw  [color={rgb, 255:red, 208; green, 2; blue, 27 }  ,draw opacity=1 ][fill={rgb, 255:red, 208; green, 2; blue, 27 }  ,fill opacity=0.5 ][line width=2.25]  (327.8,104) .. controls (332.8,50) and (415.8,66) .. (420.8,104) .. controls (425.8,142) and (322.8,158) .. (327.8,104) -- cycle ;
\draw    (110.8,103) -- (460.8,103) ;
\draw [color={rgb, 255:red, 139; green, 87; blue, 42 }  ,draw opacity=1 ][line width=2.25]    (27.8,103) .. controls (29.47,101.33) and (31.13,101.33) .. (32.8,103) .. controls (34.47,104.67) and (36.13,104.67) .. (37.8,103) .. controls (39.47,101.33) and (41.13,101.33) .. (42.8,103) .. controls (44.47,104.67) and (46.13,104.67) .. (47.8,103) .. controls (49.47,101.33) and (51.13,101.33) .. (52.8,103) .. controls (54.47,104.67) and (56.13,104.67) .. (57.8,103) .. controls (59.47,101.33) and (61.13,101.33) .. (62.8,103) .. controls (64.47,104.67) and (66.13,104.67) .. (67.8,103) .. controls (69.47,101.33) and (71.13,101.33) .. (72.8,103) .. controls (74.47,104.67) and (76.13,104.67) .. (77.8,103) .. controls (79.47,101.33) and (81.13,101.33) .. (82.8,103) .. controls (84.47,104.67) and (86.13,104.67) .. (87.8,103) .. controls (89.47,101.33) and (91.13,101.33) .. (92.8,103) .. controls (94.47,104.67) and (96.13,104.67) .. (97.8,103) .. controls (99.47,101.33) and (101.13,101.33) .. (102.8,103) .. controls (104.47,104.67) and (106.13,104.67) .. (107.8,103) .. controls (109.47,101.33) and (111.13,101.33) .. (112.8,103) .. controls (114.47,104.67) and (116.13,104.67) .. (117.8,103) .. controls (119.47,101.33) and (121.13,101.33) .. (122.8,103) .. controls (124.47,104.67) and (126.13,104.67) .. (127.8,103) .. controls (129.47,101.33) and (131.13,101.33) .. (132.8,103) .. controls (134.47,104.67) and (136.13,104.67) .. (137.8,103) .. controls (139.47,101.33) and (141.13,101.33) .. (142.8,103) .. controls (144.47,104.67) and (146.13,104.67) .. (147.8,103) .. controls (149.47,101.33) and (151.13,101.33) .. (152.8,103) .. controls (154.47,104.67) and (156.13,104.67) .. (157.8,103) .. controls (159.47,101.33) and (161.13,101.33) .. (162.8,103) .. controls (164.47,104.67) and (166.13,104.67) .. (167.8,103) .. controls (169.47,101.33) and (171.13,101.33) .. (172.8,103) .. controls (174.47,104.67) and (176.13,104.67) .. (177.8,103) .. controls (179.47,101.33) and (181.13,101.33) .. (182.8,103) .. controls (184.47,104.67) and (186.13,104.67) .. (187.8,103) .. controls (189.47,101.33) and (191.13,101.33) .. (192.8,103) .. controls (194.47,104.67) and (196.13,104.67) .. (197.8,103) -- (198.8,103) -- (198.8,103) ;
\draw    (360.8,12) -- (359.8,227) ;
\draw    (199.8,96) -- (199.8,110) ;
\draw  [color={rgb, 255:red, 74; green, 144; blue, 226 }  ,draw opacity=1 ][fill={rgb, 255:red, 74; green, 144; blue, 226 }  ,fill opacity=0.51 ][line width=2.25]  (534.8,2) .. controls (558.8,-23) and (626.8,-35) .. (658.8,0) .. controls (659.8,43) and (660.8,190) .. (659.8,232) .. controls (645.8,279) and (548.8,260) .. (525.8,233) .. controls (511.8,194) and (483.8,161) .. (478.8,118) .. controls (477.8,81) and (519.8,46) .. (534.8,2) -- cycle ;

\draw (186,115.4) node [anchor=north west][inner sep=0.75pt]    {$-1$};
\draw (342,111.4) node [anchor=north west][inner sep=0.75pt]    {$0$};
\draw (394,142.4) node [anchor=north west][inner sep=0.75pt]  [color={rgb, 255:red, 208; green, 2; blue, 27 }  ,opacity=1 ]  {$\ell _{0}$};
\draw (51,20.4) node [anchor=north west][inner sep=0.75pt]    {$\Re f( \xi;z ) =c\sim -\infty $};
\draw (460,149.4) node [anchor=north west][inner sep=0.75pt]  [color={rgb, 255:red, 74; green, 144; blue, 226 }  ,opacity=1 ]  {$\ell _{\infty }$};
\end{tikzpicture}
\end{center}
\caption{An illustration of the topological configuration of the level curves~$\ell_{0}$ and $\ell_{\infty}$ when $c$ approaches $-\infty$ and $\Re z>a$. The filled regions denotes the areas where the value of $\Re f(\xi;z)$ is less than the level~$c$ of the level curves.} \label{fig:conf-c<<0_case2_charlier}
\end{figure}
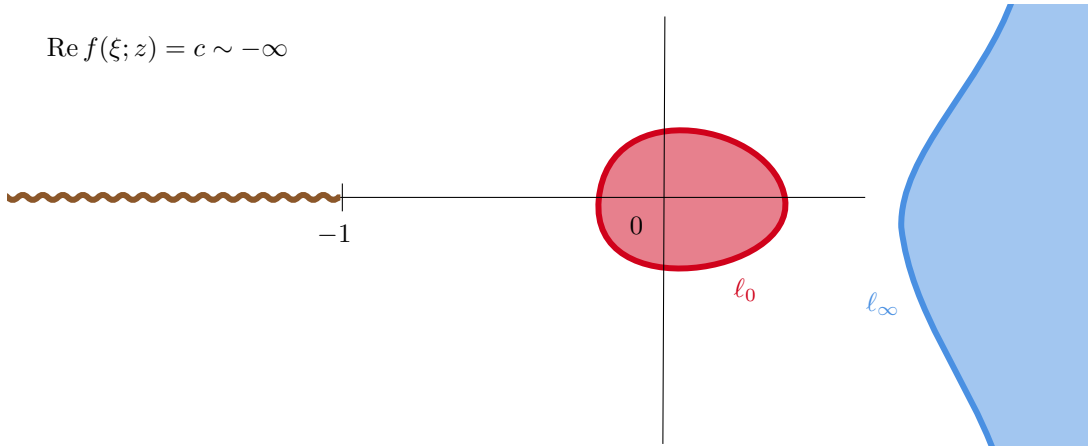
As $c\to\infty$, we have to end up with the two components $\ell_{-1}$ and $\ell_{-\infty}$. There are again three possible scenarios how the transitions can proceed. Either the two initial components $\ell_{0}$ and $\ell_{\infty}$ merge at~$\xi_{+}$ and then splits at~$\xi_{-}$ to form $\ell_{-1}$ and $\ell_{-\infty}$. Or one of the initials $\ell_{0}$, $\ell_{\infty}$ splits at~$\xi_{+}$ to form $\ell_{-1}$ as a product and the remaining components merge at~$\xi_{-}$ to form~$\ell_{-\infty}$. We will again discuss all three possible scenarios.

\textbf{a) $\ell_{0}$ and $\ell_{\infty}$ merge at $\xi_{+}$:} The situation is depicted in Figure~\ref{fig:conf_ell0_ellinf_merge-case2-parab}. It allows to apply the saddle point method in a~straightforward fashion resulting again in the formula~\eqref{eq:ratio_I_n_asympt}.

\begin{figure}[htb]
\begin{center}
\begin{tikzpicture}[x=0.75pt,y=0.75pt,yscale=-1,xscale=1]
\clip   (32.8,6) rectangle (571.8,228) ;
\draw  [color={rgb, 255:red, 208; green, 2; blue, 27 }  ,draw opacity=1 ][fill={rgb, 255:red, 208; green, 2; blue, 27 }  ,fill opacity=0.5 ][line width=2.25]  (338.8,102) .. controls (341.27,54.2) and (394.6,51.53) .. (446.6,82.2) .. controls (411.93,142.87) and (337.93,145.53) .. (338.8,102) -- cycle ;
\draw    (110.8,103) -- (587.8,104) ;
\draw [color={rgb, 255:red, 139; green, 87; blue, 42 }  ,draw opacity=1 ][line width=2.25]    (21.8,103) .. controls (23.47,101.33) and (25.13,101.33) .. (26.8,103) .. controls (28.47,104.67) and (30.13,104.67) .. (31.8,103) .. controls (33.47,101.33) and (35.13,101.33) .. (36.8,103) .. controls (38.47,104.67) and (40.13,104.67) .. (41.8,103) .. controls (43.47,101.33) and (45.13,101.33) .. (46.8,103) .. controls (48.47,104.67) and (50.13,104.67) .. (51.8,103) .. controls (53.47,101.33) and (55.13,101.33) .. (56.8,103) .. controls (58.47,104.67) and (60.13,104.67) .. (61.8,103) .. controls (63.47,101.33) and (65.13,101.33) .. (66.8,103) .. controls (68.47,104.67) and (70.13,104.67) .. (71.8,103) .. controls (73.47,101.33) and (75.13,101.33) .. (76.8,103) .. controls (78.47,104.67) and (80.13,104.67) .. (81.8,103) .. controls (83.47,101.33) and (85.13,101.33) .. (86.8,103) .. controls (88.47,104.67) and (90.13,104.67) .. (91.8,103) .. controls (93.47,101.33) and (95.13,101.33) .. (96.8,103) .. controls (98.47,104.67) and (100.13,104.67) .. (101.8,103) .. controls (103.47,101.33) and (105.13,101.33) .. (106.8,103) .. controls (108.47,104.67) and (110.13,104.67) .. (111.8,103) .. controls (113.47,101.33) and (115.13,101.33) .. (116.8,103) .. controls (118.47,104.67) and (120.13,104.67) .. (121.8,103) .. controls (123.47,101.33) and (125.13,101.33) .. (126.8,103) .. controls (128.47,104.67) and (130.13,104.67) .. (131.8,103) .. controls (133.47,101.33) and (135.13,101.33) .. (136.8,103) .. controls (138.47,104.67) and (140.13,104.67) .. (141.8,103) .. controls (143.47,101.33) and (145.13,101.33) .. (146.8,103) .. controls (148.47,104.67) and (150.13,104.67) .. (151.8,103) .. controls (153.47,101.33) and (155.13,101.33) .. (156.8,103) .. controls (158.47,104.67) and (160.13,104.67) .. (161.8,103) .. controls (163.47,101.33) and (165.13,101.33) .. (166.8,103) .. controls (168.47,104.67) and (170.13,104.67) .. (171.8,103) .. controls (173.47,101.33) and (175.13,101.33) .. (176.8,103) .. controls (178.47,104.67) and (180.13,104.67) .. (181.8,103) .. controls (183.47,101.33) and (185.13,101.33) .. (186.8,103) .. controls (188.47,104.67) and (190.13,104.67) .. (191.8,103) .. controls (193.47,101.33) and (195.13,101.33) .. (196.8,103) -- (198.8,103) -- (198.8,103) ;
\draw    (360.8,4) -- (359.8,227) ;
\draw    (199.8,96) -- (199.8,110) ;
\draw  [color={rgb, 255:red, 74; green, 144; blue, 226 }  ,draw opacity=1 ][fill={rgb, 255:red, 74; green, 144; blue, 226 }  ,fill opacity=0.51 ][line width=2.25]  (481.27,3.53) .. controls (505.27,-21.47) and (626.13,-35) .. (658.13,0) .. controls (659.13,43) and (660.13,190) .. (659.13,232) .. controls (645.13,279) and (580.27,254.53) .. (557.27,227.53) .. controls (535.93,188.2) and (506.6,114.87) .. (446.6,82.2) .. controls (467.93,44.87) and (471.27,38.2) .. (481.27,3.53) -- cycle ;

\draw  [color={rgb, 255:red, 144; green, 19; blue, 254 }  ,draw opacity=1 ][dash pattern={on 5.63pt off 4.5pt}][line width=1.5]  (298.8,121) .. controls (315.8,200) and (496.6,185.53) .. (446.6,82.2) .. controls (396.6,-21.13) and (281.8,42) .. (298.8,121) -- cycle ;

\draw (186,115.4) node [anchor=north west][inner sep=0.75pt]    {$-1$};
\draw (366,113.4) node [anchor=north west][inner sep=0.75pt]    {$0$};
\draw (393.33,134.4) node [anchor=north west][inner sep=0.75pt]  [color={rgb, 255:red, 208; green, 2; blue, 27 }  ,opacity=1 ]  {$\ell _{0}$};
\draw (48,26.4) node [anchor=north west][inner sep=0.75pt]    {$\Re f( \xi;z) =\Re f( \xi _{+};z)$};
\draw (488.33,149.4) node [anchor=north west][inner sep=0.75pt]  [color={rgb, 255:red, 74; green, 144; blue, 226 }  ,opacity=1 ]  {$\ell _{\infty }$};
\draw (435,42.73) node [anchor=north west][inner sep=0.75pt]  [color={rgb, 255:red, 0; green, 0; blue, 0 }  ,opacity=1 ]  {$\xi _{+}$};
\draw (302.67,30.73) node [anchor=north west][inner sep=0.75pt]  [color={rgb, 255:red, 65; green, 117; blue, 5 }  ,opacity=1 ]  {$\textcolor[rgb]{0.56,0.07,1}{\gamma }$};
\end{tikzpicture}
\end{center}
\caption{An illustration of the topological configuration when $\ell_{0}$ and $\ell_{\infty}$ merge at $\xi_{+}$. The Jordan curve $\gamma$ fulfills the assumption~(v) of Theorem~\ref{thm:saddle-point}.} \label{fig:conf_ell0_ellinf_merge-case2-parab}
\end{figure}
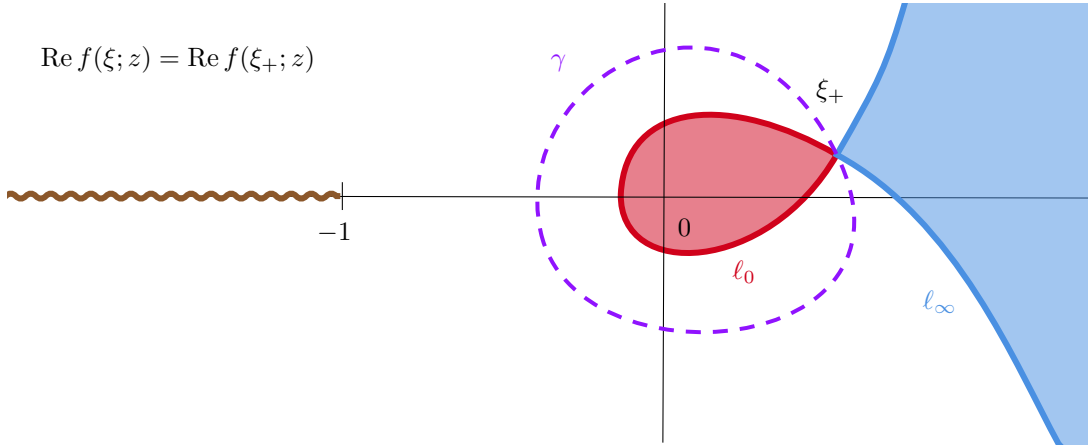

\textbf{b) $\ell_{\infty}$ splits at $\xi_{+}$:}
This configuration does not occur. Otherwise, one gets a contradiction with Lemma~\ref{lem:crossings_imag_line_charlier} since the level curves would have to cross the imaginary line more than twice, see Figure~\ref{fig:conf_ellinf_split_charlier}.
\begin{figure}[htb]
\begin{center}
\begin{tikzpicture}[x=0.75pt,y=0.75pt,yscale=-1,xscale=1]
\clip   (32.8,6) rectangle (571.8,228);
\draw  [color={rgb, 255:red, 74; green, 144; blue, 226 }  ,draw opacity=1 ][fill={rgb, 255:red, 74; green, 144; blue, 226 }  ,fill opacity=0.51 ][line width=2.25]  (288.7,50.75) .. controls (278.73,117) and (258.7,140.08) .. (272.7,176.75) .. controls (210.7,174.75) and (159.7,176.75) .. (61.2,175.25) .. controls (66.2,137.25) and (178.2,32.75) .. (288.7,50.75) -- cycle ;
\draw  [color={rgb, 255:red, 208; green, 2; blue, 27 }  ,draw opacity=1 ][fill={rgb, 255:red, 208; green, 2; blue, 27 }  ,fill opacity=0.5 ][line width=2.25]  (348.7,178.75) .. controls (336.7,156.25) and (377.7,126.25) .. (387.7,150.75) .. controls (397.7,175.25) and (360.7,201.25) .. (348.7,178.75) -- cycle ;
\draw    (98.8,176.33) -- (575.8,177.33) ;
\draw [color={rgb, 255:red, 139; green, 87; blue, 42 }  ,draw opacity=1 ][line width=2.25]    (31.2,175.75) .. controls (32.87,174.09) and (34.54,174.1) .. (36.2,175.77) .. controls (37.86,177.44) and (39.53,177.45) .. (41.2,175.8) .. controls (42.87,174.14) and (44.54,174.15) .. (46.2,175.82) .. controls (47.86,177.49) and (49.53,177.5) .. (51.2,175.84) .. controls (52.87,174.19) and (54.54,174.2) .. (56.2,175.87) .. controls (57.86,177.54) and (59.53,177.55) .. (61.2,175.89) .. controls (62.87,174.23) and (64.54,174.24) .. (66.2,175.91) .. controls (67.86,177.58) and (69.53,177.59) .. (71.2,175.94) .. controls (72.87,174.28) and (74.54,174.29) .. (76.2,175.96) .. controls (77.86,177.63) and (79.53,177.64) .. (81.2,175.98) .. controls (82.87,174.33) and (84.54,174.34) .. (86.2,176.01) .. controls (87.86,177.68) and (89.53,177.69) .. (91.2,176.03) .. controls (92.87,174.37) and (94.54,174.38) .. (96.2,176.05) .. controls (97.86,177.72) and (99.53,177.73) .. (101.2,176.08) .. controls (102.87,174.42) and (104.54,174.43) .. (106.2,176.1) .. controls (107.86,177.77) and (109.53,177.78) .. (111.2,176.12) .. controls (112.87,174.47) and (114.54,174.48) .. (116.2,176.15) .. controls (117.86,177.82) and (119.53,177.83) .. (121.2,176.17) .. controls (122.87,174.51) and (124.54,174.52) .. (126.2,176.19) .. controls (127.86,177.86) and (129.53,177.87) .. (131.2,176.22) .. controls (132.87,174.56) and (134.54,174.57) .. (136.2,176.24) .. controls (137.86,177.91) and (139.53,177.92) .. (141.2,176.26) .. controls (142.87,174.61) and (144.54,174.62) .. (146.2,176.29) .. controls (147.86,177.96) and (149.53,177.97) .. (151.2,176.31) .. controls (152.87,174.65) and (154.54,174.66) .. (156.2,176.33) .. controls (157.86,178) and (159.53,178.01) .. (161.2,176.36) .. controls (162.87,174.7) and (164.54,174.71) .. (166.2,176.38) .. controls (167.86,178.05) and (169.53,178.06) .. (171.2,176.4) .. controls (172.87,174.75) and (174.54,174.76) .. (176.2,176.43) .. controls (177.86,178.1) and (179.53,178.11) .. (181.2,176.45) .. controls (182.87,174.79) and (184.54,174.8) .. (186.2,176.47) .. controls (187.86,178.14) and (189.53,178.15) .. (191.2,176.5) .. controls (192.87,174.84) and (194.54,174.85) .. (196.2,176.52) .. controls (197.86,178.19) and (199.53,178.2) .. (201.2,176.54) .. controls (202.87,174.89) and (204.54,174.9) .. (206.2,176.57) .. controls (207.86,178.24) and (209.53,178.25) .. (211.2,176.59) .. controls (212.87,174.93) and (214.54,174.94) .. (216.2,176.61) .. controls (217.86,178.28) and (219.53,178.29) .. (221.2,176.64) .. controls (222.87,174.98) and (224.54,174.99) .. (226.2,176.66) .. controls (227.86,178.33) and (229.53,178.34) .. (231.2,176.68) .. controls (232.87,175.03) and (234.54,175.04) .. (236.2,176.71) .. controls (237.86,178.38) and (239.53,178.39) .. (241.2,176.73) .. controls (242.87,175.07) and (244.54,175.08) .. (246.2,176.75) .. controls (247.86,178.42) and (249.53,178.43) .. (251.2,176.78) .. controls (252.87,175.12) and (254.54,175.13) .. (256.2,176.8) .. controls (257.86,178.47) and (259.53,178.48) .. (261.2,176.82) .. controls (262.87,175.17) and (264.54,175.18) .. (266.2,176.85) .. controls (267.86,178.52) and (269.53,178.53) .. (271.2,176.87) .. controls (272.87,175.21) and (274.54,175.22) .. (276.2,176.89) .. controls (277.86,178.56) and (279.53,178.57) .. (281.2,176.92) .. controls (282.87,175.26) and (284.54,175.27) .. (286.2,176.94) .. controls (287.86,178.61) and (289.53,178.62) .. (291.2,176.96) .. controls (292.87,175.31) and (294.54,175.32) .. (296.2,176.99) -- (298.8,177) -- (298.8,177) ;
\draw    (360.8,4) -- (359.8,227) ;
\draw    (299.3,170) -- (299.3,184) ;
\draw  [color={rgb, 255:red, 74; green, 144; blue, 226 }  ,draw opacity=1 ][fill={rgb, 255:red, 74; green, 144; blue, 226 }  ,fill opacity=0.51 ][line width=2.25]  (294.8,2.35) .. controls (328.3,-6.65) and (593.3,-28.15) .. (625.3,6.85) .. controls (626.3,49.85) and (615.8,191.35) .. (574.8,227.35) .. controls (503.8,127.65) and (418.3,117.65) .. (397.8,114.15) .. controls (380.3,109.15) and (371.3,67.15) .. (288.7,50.75) .. controls (292.8,14.35) and (290.8,31.35) .. (294.8,2.35) -- cycle ;

\draw (288.5,191.4) node [anchor=north west][inner sep=0.75pt]    {$-1$};
\draw (375.5,194.4) node [anchor=north west][inner sep=0.75pt]    {$0$};
\draw (393.33,134.4) node [anchor=north west][inner sep=0.75pt]  [color={rgb, 255:red, 208; green, 2; blue, 27 }  ,opacity=1 ]  {$\ell _{0}$};
\draw (48,26.4) node [anchor=north west][inner sep=0.75pt]    {$\Re f(\xi;z) =\Re f(\xi _{+};z)$};
\draw (319,74.07) node [anchor=north west][inner sep=0.75pt]  [color={rgb, 255:red, 74; green, 144; blue, 226 }  ,opacity=1 ]  {$\ell _{\infty }$};
\draw (261.67,23.4) node [anchor=north west][inner sep=0.75pt]  [color={rgb, 255:red, 0; green, 0; blue, 0 }  ,opacity=1 ]  {$\xi _{+}$};
\end{tikzpicture}
\end{center}
\caption{An illustration of the topological configuration when $\ell_{\infty}$ splits at $\xi_{+}$.} \label{fig:conf_ellinf_split_charlier}
\end{figure}
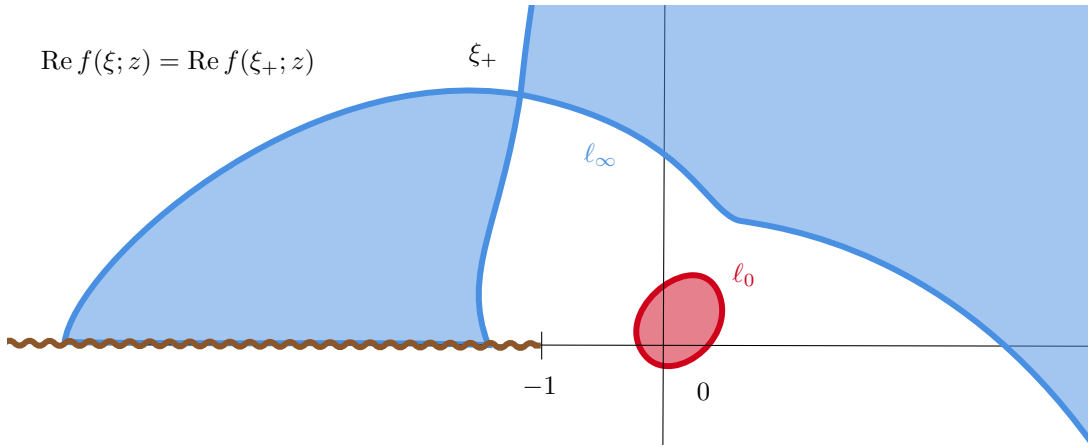

\textbf{c) $\ell_{0}$ splits at $\xi_{+}$:} The last configuration requires a more subtle analysis since a behavior of the level curves in a neighborhood of the cut $(-\infty,-1)$ has to be taken into account. We divide the final discussion into 3 parts.

\textbf{c1) No interference from the branch cut:} If the curve $l_0$ splits at $\xi_+$ so that no interference from the branch cut $(-\infty,-1)$ occurs, see Figure~\ref{fig:conf_ell0_split_charlier}, the saddle point method is applicable resulting in~\eqref{eq:ratio_I_n_asympt} again.
\begin{figure}[htb]
\begin{center}
\begin{tikzpicture}[x=0.75pt,y=0.75pt,yscale=-1,xscale=1]

\clip   (32.8,6) rectangle (571.8,228) ;
\draw  [color={rgb, 255:red, 208; green, 2; blue, 27 }  ,draw opacity=1 ][fill={rgb, 255:red, 208; green, 2; blue, 27 }  ,fill opacity=0.5 ][line width=2.25]  (294.2,83.65) .. controls (210.7,46.15) and (150.7,92.15) .. (101.2,147.65) .. controls (154.7,147.15) and (256.7,147.65) .. (293.7,150.15) .. controls (279.7,135.65) and (276.7,124.15) .. (294.2,83.65) -- cycle ;
\draw  [color={rgb, 255:red, 208; green, 2; blue, 27 }  ,draw opacity=1 ][fill={rgb, 255:red, 208; green, 2; blue, 27 }  ,fill opacity=0.5 ][line width=2.25]  (410.7,107.65) .. controls (412.7,46.65) and (341.2,24.65) .. (294.2,83.65) .. controls (334.7,110.15) and (319.2,161.65) .. (345.2,170.15) .. controls (382.7,178.15) and (410.7,140.15) .. (410.7,107.65) -- cycle ;
\draw    (122.8,148.5) -- (599.8,149.5) ;
\draw [color={rgb, 255:red, 139; green, 87; blue, 42 }  ,draw opacity=1 ][line width=2.25]    (29.7,147.65) .. controls (31.37,145.99) and (33.04,146) .. (34.7,147.67) .. controls (36.37,149.34) and (38.03,149.34) .. (39.7,147.68) .. controls (41.37,146.02) and (43.04,146.03) .. (44.7,147.7) .. controls (46.37,149.37) and (48.03,149.37) .. (49.7,147.71) .. controls (51.37,146.05) and (53.04,146.06) .. (54.7,147.73) .. controls (56.37,149.4) and (58.03,149.4) .. (59.7,147.74) .. controls (61.37,146.08) and (63.04,146.09) .. (64.7,147.76) .. controls (66.37,149.43) and (68.03,149.43) .. (69.7,147.77) .. controls (71.37,146.11) and (73.04,146.12) .. (74.7,147.79) .. controls (76.37,149.46) and (78.03,149.46) .. (79.7,147.8) .. controls (81.37,146.14) and (83.04,146.15) .. (84.7,147.82) .. controls (86.36,149.49) and (88.03,149.5) .. (89.7,147.84) .. controls (91.37,146.18) and (93.03,146.18) .. (94.7,147.85) .. controls (96.36,149.52) and (98.03,149.53) .. (99.7,147.87) .. controls (101.37,146.21) and (103.03,146.21) .. (104.7,147.88) .. controls (106.36,149.55) and (108.03,149.56) .. (109.7,147.9) .. controls (111.37,146.24) and (113.03,146.24) .. (114.7,147.91) .. controls (116.36,149.58) and (118.03,149.59) .. (119.7,147.93) .. controls (121.37,146.27) and (123.03,146.27) .. (124.7,147.94) .. controls (126.36,149.61) and (128.03,149.62) .. (129.7,147.96) .. controls (131.37,146.3) and (133.04,146.31) .. (134.7,147.98) .. controls (136.37,149.65) and (138.03,149.65) .. (139.7,147.99) .. controls (141.37,146.33) and (143.04,146.34) .. (144.7,148.01) .. controls (146.37,149.68) and (148.03,149.68) .. (149.7,148.02) .. controls (151.37,146.36) and (153.04,146.37) .. (154.7,148.04) .. controls (156.37,149.71) and (158.03,149.71) .. (159.7,148.05) .. controls (161.37,146.39) and (163.04,146.4) .. (164.7,148.07) .. controls (166.37,149.74) and (168.03,149.74) .. (169.7,148.08) .. controls (171.37,146.42) and (173.04,146.43) .. (174.7,148.1) .. controls (176.37,149.77) and (178.03,149.77) .. (179.7,148.11) .. controls (181.37,146.45) and (183.04,146.46) .. (184.7,148.13) .. controls (186.36,149.8) and (188.03,149.81) .. (189.7,148.15) .. controls (191.37,146.49) and (193.03,146.49) .. (194.7,148.16) .. controls (196.36,149.83) and (198.03,149.84) .. (199.7,148.18) .. controls (201.37,146.52) and (203.03,146.52) .. (204.7,148.19) .. controls (206.36,149.86) and (208.03,149.87) .. (209.7,148.21) .. controls (211.37,146.55) and (213.03,146.55) .. (214.7,148.22) .. controls (216.36,149.89) and (218.03,149.9) .. (219.7,148.24) .. controls (221.37,146.58) and (223.03,146.58) .. (224.7,148.25) .. controls (226.36,149.92) and (228.03,149.93) .. (229.7,148.27) .. controls (231.37,146.61) and (233.03,146.61) .. (234.7,148.28) .. controls (236.36,149.95) and (238.03,149.96) .. (239.7,148.3) .. controls (241.37,146.64) and (243.04,146.65) .. (244.7,148.32) .. controls (246.37,149.99) and (248.03,149.99) .. (249.7,148.33) .. controls (251.37,146.67) and (253.04,146.68) .. (254.7,148.35) .. controls (256.37,150.02) and (258.03,150.02) .. (259.7,148.36) .. controls (261.37,146.7) and (263.04,146.71) .. (264.7,148.38) .. controls (266.37,150.05) and (268.03,150.05) .. (269.7,148.39) .. controls (271.37,146.73) and (273.04,146.74) .. (274.7,148.41) .. controls (276.37,150.08) and (278.03,150.08) .. (279.7,148.42) .. controls (281.37,146.76) and (283.04,146.77) .. (284.7,148.44) .. controls (286.37,150.11) and (288.03,150.11) .. (289.7,148.45) .. controls (291.37,146.79) and (293.04,146.8) .. (294.7,148.47) .. controls (296.36,150.14) and (298.03,150.15) .. (299.7,148.49) -- (304.3,148.5) -- (304.3,148.5) ;
\draw    (344.3,4.5) -- (343.3,227.5) ;
\draw    (306.3,141.5) -- (306.3,155.5) ;
\draw  [color={rgb, 255:red, 144; green, 19; blue, 254 }  ,draw opacity=1 ][dash pattern={on 5.63pt off 4.5pt}][line width=1.5]  (294.2,83.65) .. controls (291.2,183.15) and (396.2,212.15) .. (439.2,133.15) .. controls (482.2,54.15) and (297.2,-15.85) .. (294.2,83.65) -- cycle ;
\draw  [color={rgb, 255:red, 74; green, 144; blue, 226 }  ,draw opacity=1 ][fill={rgb, 255:red, 74; green, 144; blue, 226 }  ,fill opacity=0.51 ][line width=2.25]  (527,-9) .. controls (551,-34) and (619,-46) .. (651,-11) .. controls (652,32) and (653,179) .. (652,221) .. controls (638,268) and (541,249) .. (518,222) .. controls (504,183) and (476,150) .. (471,107) .. controls (470,70) and (512,35) .. (527,-9) -- cycle ;

\draw (284,160.4) node [anchor=north west][inner sep=0.75pt]    {$-1$};
\draw (352,151.9) node [anchor=north west][inner sep=0.75pt]    {$0$};
\draw (417.33,95.9) node [anchor=north west][inner sep=0.75pt]  [color={rgb, 255:red, 208; green, 2; blue, 27 }  ,opacity=1 ]  {$\ell _{0}$};
\draw (48,26.4) node [anchor=north west][inner sep=0.75pt]    {$\Re f(\xi;z) =\Re f(\xi _{+};z)$};
\draw (253,42.23) node [anchor=north west][inner sep=0.75pt]  [color={rgb, 255:red, 0; green, 0; blue, 0 }  ,opacity=1 ]  {$\xi _{+}$};
\draw (296.17,20.23) node [anchor=north west][inner sep=0.75pt]  [color={rgb, 255:red, 65; green, 117; blue, 5 }  ,opacity=1 ]  {$\textcolor[rgb]{0.56,0.07,1}{\gamma }$};
\draw (466,169.4) node [anchor=north west][inner sep=0.75pt]  [color={rgb, 255:red, 74; green, 144; blue, 226 }  ,opacity=1 ]  {$\ell _{\infty }$};
\end{tikzpicture}
\end{center}
\caption{An illustration of the topological configuration when $\ell_{0}$ splits at $\xi_{+}$ without any interference from the branch cut. The Jordan curve $\gamma$ fulfills the assumption~(v) of Theorem~\ref{thm:saddle-point}.}\label{fig:conf_ell0_split_charlier}
\end{figure}
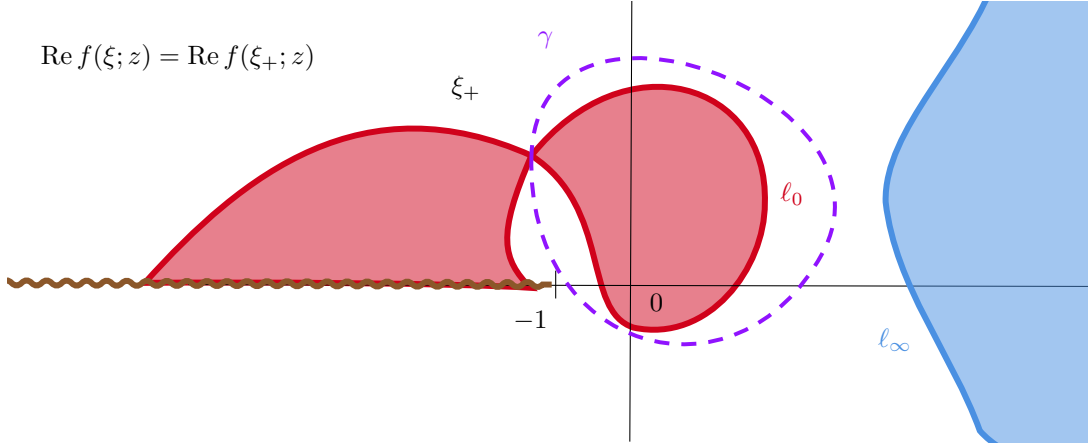

\textbf{c2) Same side interference:} Suppose an interference from the branch cut $(-\infty,-1)$ does occur. If it happens from ``the same side'', as illustrated in Figure~\ref{fig:conf_ell0_split_same_side_charlier}, one gets a contradiction with Lemma~\ref{lem:crossings_the_cut_charlier} according to which the level curves can approach the cut $(-\infty,-1)$ at most at two points from above and from below, too. Thus, the considered configuration is impossible. 
\begin{figure}[htb]
\begin{center}
\begin{tikzpicture}[x=0.75pt,y=0.75pt,yscale=-1,xscale=1]

\clip   (32.8,6) rectangle (571.8,228) ;
\draw  [color={rgb, 255:red, 208; green, 2; blue, 27 }  ,draw opacity=1 ][fill={rgb, 255:red, 208; green, 2; blue, 27 }  ,fill opacity=0.5 ][line width=2.25]  (262.36,77.88) .. controls (220.61,59.13) and (115.56,61.08) .. (101.2,147.65) .. controls (154.7,147.15) and (192.96,145.78) .. (229.96,148.28) .. controls (213.16,127.48) and (251.96,107.08) .. (262.36,77.88) -- cycle ;
\draw  [color={rgb, 255:red, 208; green, 2; blue, 27 }  ,draw opacity=1 ][fill={rgb, 255:red, 208; green, 2; blue, 27 }  ,fill opacity=0.5 ][line width=2.25]  (410.7,107.65) .. controls (412.7,46.65) and (290.76,15.08) .. (262.36,77.88) .. controls (290.36,85.08) and (258.76,132.28) .. (259.16,148.28) .. controls (267.16,149.08) and (276.76,149.08) .. (281.56,148.68) .. controls (289.56,85.88) and (332.2,165.9) .. (345.2,170.15) .. controls (362.36,178.68) and (410.7,140.15) .. (410.7,107.65) -- cycle ;
\draw    (122.8,148.5) -- (599.8,149.5) ;
\draw [color={rgb, 255:red, 139; green, 87; blue, 42 }  ,draw opacity=1 ][line width=2.25]    (29.7,147.65) .. controls (31.37,145.99) and (33.04,146) .. (34.7,147.67) .. controls (36.37,149.34) and (38.03,149.34) .. (39.7,147.68) .. controls (41.37,146.02) and (43.04,146.03) .. (44.7,147.7) .. controls (46.37,149.37) and (48.03,149.37) .. (49.7,147.71) .. controls (51.37,146.05) and (53.04,146.06) .. (54.7,147.73) .. controls (56.37,149.4) and (58.03,149.4) .. (59.7,147.74) .. controls (61.37,146.08) and (63.04,146.09) .. (64.7,147.76) .. controls (66.37,149.43) and (68.03,149.43) .. (69.7,147.77) .. controls (71.37,146.11) and (73.04,146.12) .. (74.7,147.79) .. controls (76.37,149.46) and (78.03,149.46) .. (79.7,147.8) .. controls (81.37,146.14) and (83.04,146.15) .. (84.7,147.82) .. controls (86.36,149.49) and (88.03,149.5) .. (89.7,147.84) .. controls (91.37,146.18) and (93.03,146.18) .. (94.7,147.85) .. controls (96.36,149.52) and (98.03,149.53) .. (99.7,147.87) .. controls (101.37,146.21) and (103.03,146.21) .. (104.7,147.88) .. controls (106.36,149.55) and (108.03,149.56) .. (109.7,147.9) .. controls (111.37,146.24) and (113.03,146.24) .. (114.7,147.91) .. controls (116.36,149.58) and (118.03,149.59) .. (119.7,147.93) .. controls (121.37,146.27) and (123.03,146.27) .. (124.7,147.94) .. controls (126.36,149.61) and (128.03,149.62) .. (129.7,147.96) .. controls (131.37,146.3) and (133.04,146.31) .. (134.7,147.98) .. controls (136.37,149.65) and (138.03,149.65) .. (139.7,147.99) .. controls (141.37,146.33) and (143.04,146.34) .. (144.7,148.01) .. controls (146.37,149.68) and (148.03,149.68) .. (149.7,148.02) .. controls (151.37,146.36) and (153.04,146.37) .. (154.7,148.04) .. controls (156.37,149.71) and (158.03,149.71) .. (159.7,148.05) .. controls (161.37,146.39) and (163.04,146.4) .. (164.7,148.07) .. controls (166.37,149.74) and (168.03,149.74) .. (169.7,148.08) .. controls (171.37,146.42) and (173.04,146.43) .. (174.7,148.1) .. controls (176.37,149.77) and (178.03,149.77) .. (179.7,148.11) .. controls (181.37,146.45) and (183.04,146.46) .. (184.7,148.13) .. controls (186.36,149.8) and (188.03,149.81) .. (189.7,148.15) .. controls (191.37,146.49) and (193.03,146.49) .. (194.7,148.16) .. controls (196.36,149.83) and (198.03,149.84) .. (199.7,148.18) .. controls (201.37,146.52) and (203.03,146.52) .. (204.7,148.19) .. controls (206.36,149.86) and (208.03,149.87) .. (209.7,148.21) .. controls (211.37,146.55) and (213.03,146.55) .. (214.7,148.22) .. controls (216.36,149.89) and (218.03,149.9) .. (219.7,148.24) .. controls (221.37,146.58) and (223.03,146.58) .. (224.7,148.25) .. controls (226.36,149.92) and (228.03,149.93) .. (229.7,148.27) .. controls (231.37,146.61) and (233.03,146.61) .. (234.7,148.28) .. controls (236.36,149.95) and (238.03,149.96) .. (239.7,148.3) .. controls (241.37,146.64) and (243.04,146.65) .. (244.7,148.32) .. controls (246.37,149.99) and (248.03,149.99) .. (249.7,148.33) .. controls (251.37,146.67) and (253.04,146.68) .. (254.7,148.35) .. controls (256.37,150.02) and (258.03,150.02) .. (259.7,148.36) .. controls (261.37,146.7) and (263.04,146.71) .. (264.7,148.38) .. controls (266.37,150.05) and (268.03,150.05) .. (269.7,148.39) .. controls (271.37,146.73) and (273.04,146.74) .. (274.7,148.41) .. controls (276.37,150.08) and (278.03,150.08) .. (279.7,148.42) .. controls (281.37,146.76) and (283.04,146.77) .. (284.7,148.44) .. controls (286.37,150.11) and (288.03,150.11) .. (289.7,148.45) .. controls (291.37,146.79) and (293.04,146.8) .. (294.7,148.47) .. controls (296.36,150.14) and (298.03,150.15) .. (299.7,148.49) -- (304.3,148.5) -- (304.3,148.5) ;
\draw    (344.3,4.5) -- (343.3,227.5) ;
\draw    (306.3,141.5) -- (306.3,155.5) ;
\draw  [color={rgb, 255:red, 74; green, 144; blue, 226 }  ,draw opacity=1 ][fill={rgb, 255:red, 74; green, 144; blue, 226 }  ,fill opacity=0.51 ][line width=2.25]  (527,-9) .. controls (551,-34) and (619,-46) .. (651,-11) .. controls (652,32) and (653,179) .. (652,221) .. controls (638,268) and (541,249) .. (518,222) .. controls (504,183) and (476,150) .. (471,107) .. controls (470,70) and (512,35) .. (527,-9) -- cycle ;

\draw (284,160.4) node [anchor=north west][inner sep=0.75pt]    {$-1$};
\draw (352,151.9) node [anchor=north west][inner sep=0.75pt]    {$0$};
\draw (417.33,95.9) node [anchor=north west][inner sep=0.75pt]  [color={rgb, 255:red, 208; green, 2; blue, 27 }  ,opacity=1 ]  {$\ell _{0}$};
\draw (48,26.4) node [anchor=north west][inner sep=0.75pt]    {$\Re f(\xi;z) =\Re f(\xi _{+};z)$};
\draw (253,42.23) node [anchor=north west][inner sep=0.75pt]  [color={rgb, 255:red, 0; green, 0; blue, 0 }  ,opacity=1 ]  {$\xi _{+}$};
\draw (459,174.4) node [anchor=north west][inner sep=0.75pt]  [color={rgb, 255:red, 74; green, 144; blue, 226 }  ,opacity=1 ]  {$\ell _{\infty }$};
\end{tikzpicture}
\end{center}
\caption{An illustration of the topological configuration when $\ell_{0}$ splits at $\xi_{+}$ with the ''same side'' interference from the branch cut. } \label{fig:conf_ell0_split_same_side_charlier}
\end{figure}
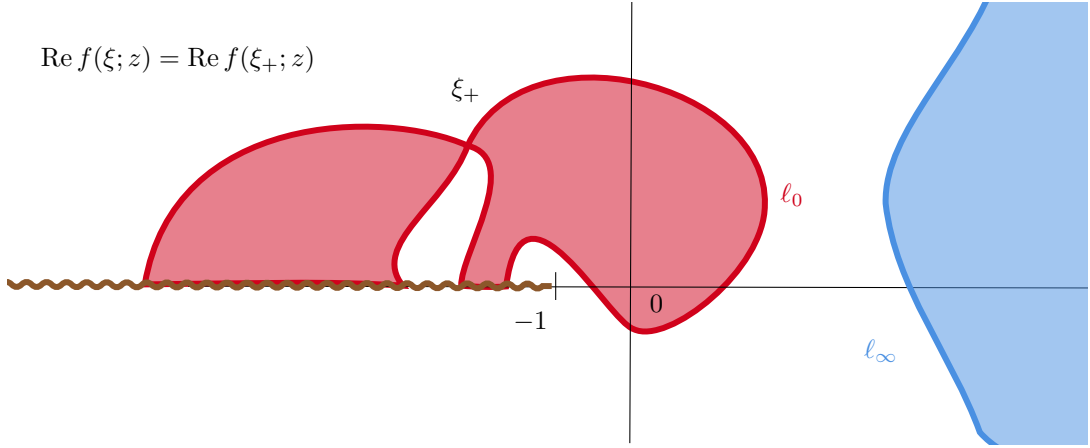

\textbf{c3) Opposite side interference:} The last configuration that is not excluded by Lemma~\ref{lem:crossings_the_cut_charlier} is na ``opposite side'' interference from the cut $(-\infty,-1)$ as shown in Figure~\ref{fig:conf_ell0_split_opposite_side_charlier}. If such a~configuration occurs, the saddle point method cannot be readily applied since no Jordan curve located in $\C\setminus((-\infty,-1]\cup\{0\})$ with $0$ in its interior fulfills the assumption~(v) of Theorem~\ref{thm:saddle-point} with the saddle point~$\xi_{+}$.
Nevertheless, we show that the formula~\eqref{eq:ratio_I_n_asympt} remains true even in this case.

\begin{figure}[htb]
\begin{center}
\begin{tikzpicture}[x=0.75pt,y=0.75pt,yscale=-1,xscale=1]

\clip   (32.8,6) rectangle (571.8,228);
\draw  [color={rgb, 255:red, 208; green, 2; blue, 27 }  ,draw opacity=1 ][fill={rgb, 255:red, 208; green, 2; blue, 27 }  ,fill opacity=0.5 ][line width=2.25]  (294.2,83.65) .. controls (210.7,46.15) and (150.7,92.15) .. (101.2,147.65) .. controls (154.7,147.15) and (256.7,147.65) .. (293.7,150.15) .. controls (279.7,135.65) and (276.7,124.15) .. (294.2,83.65) -- cycle ;
\draw  [color={rgb, 255:red, 208; green, 2; blue, 27 }  ,draw opacity=1 ][fill={rgb, 255:red, 208; green, 2; blue, 27 }  ,fill opacity=0.5 ][line width=2.25]  (410.7,107.65) .. controls (412.7,46.65) and (341.2,24.65) .. (294.2,83.65) .. controls (348.2,125.05) and (318.7,194.55) .. (256.2,149.55) .. controls (236.7,148.05) and (242.2,148.05) .. (219.2,148.05) .. controls (227.2,207.05) and (308.89,188.21) .. (355.7,174.05) .. controls (402.51,159.89) and (410.7,123.9) .. (410.7,107.65) -- cycle ;
\draw    (122.8,148.5) -- (599.8,149.5) ;
\draw [color={rgb, 255:red, 139; green, 87; blue, 42 }  ,draw opacity=1 ][line width=2.25]    (29.7,147.65) .. controls (31.37,145.99) and (33.04,146) .. (34.7,147.67) .. controls (36.37,149.34) and (38.03,149.34) .. (39.7,147.68) .. controls (41.37,146.02) and (43.04,146.03) .. (44.7,147.7) .. controls (46.37,149.37) and (48.03,149.37) .. (49.7,147.71) .. controls (51.37,146.05) and (53.04,146.06) .. (54.7,147.73) .. controls (56.37,149.4) and (58.03,149.4) .. (59.7,147.74) .. controls (61.37,146.08) and (63.04,146.09) .. (64.7,147.76) .. controls (66.37,149.43) and (68.03,149.43) .. (69.7,147.77) .. controls (71.37,146.11) and (73.04,146.12) .. (74.7,147.79) .. controls (76.37,149.46) and (78.03,149.46) .. (79.7,147.8) .. controls (81.37,146.14) and (83.04,146.15) .. (84.7,147.82) .. controls (86.36,149.49) and (88.03,149.5) .. (89.7,147.84) .. controls (91.37,146.18) and (93.03,146.18) .. (94.7,147.85) .. controls (96.36,149.52) and (98.03,149.53) .. (99.7,147.87) .. controls (101.37,146.21) and (103.03,146.21) .. (104.7,147.88) .. controls (106.36,149.55) and (108.03,149.56) .. (109.7,147.9) .. controls (111.37,146.24) and (113.03,146.24) .. (114.7,147.91) .. controls (116.36,149.58) and (118.03,149.59) .. (119.7,147.93) .. controls (121.37,146.27) and (123.03,146.27) .. (124.7,147.94) .. controls (126.36,149.61) and (128.03,149.62) .. (129.7,147.96) .. controls (131.37,146.3) and (133.04,146.31) .. (134.7,147.98) .. controls (136.37,149.65) and (138.03,149.65) .. (139.7,147.99) .. controls (141.37,146.33) and (143.04,146.34) .. (144.7,148.01) .. controls (146.37,149.68) and (148.03,149.68) .. (149.7,148.02) .. controls (151.37,146.36) and (153.04,146.37) .. (154.7,148.04) .. controls (156.37,149.71) and (158.03,149.71) .. (159.7,148.05) .. controls (161.37,146.39) and (163.04,146.4) .. (164.7,148.07) .. controls (166.37,149.74) and (168.03,149.74) .. (169.7,148.08) .. controls (171.37,146.42) and (173.04,146.43) .. (174.7,148.1) .. controls (176.37,149.77) and (178.03,149.77) .. (179.7,148.11) .. controls (181.37,146.45) and (183.04,146.46) .. (184.7,148.13) .. controls (186.36,149.8) and (188.03,149.81) .. (189.7,148.15) .. controls (191.37,146.49) and (193.03,146.49) .. (194.7,148.16) .. controls (196.36,149.83) and (198.03,149.84) .. (199.7,148.18) .. controls (201.37,146.52) and (203.03,146.52) .. (204.7,148.19) .. controls (206.36,149.86) and (208.03,149.87) .. (209.7,148.21) .. controls (211.37,146.55) and (213.03,146.55) .. (214.7,148.22) .. controls (216.36,149.89) and (218.03,149.9) .. (219.7,148.24) .. controls (221.37,146.58) and (223.03,146.58) .. (224.7,148.25) .. controls (226.36,149.92) and (228.03,149.93) .. (229.7,148.27) .. controls (231.37,146.61) and (233.03,146.61) .. (234.7,148.28) .. controls (236.36,149.95) and (238.03,149.96) .. (239.7,148.3) .. controls (241.37,146.64) and (243.04,146.65) .. (244.7,148.32) .. controls (246.37,149.99) and (248.03,149.99) .. (249.7,148.33) .. controls (251.37,146.67) and (253.04,146.68) .. (254.7,148.35) .. controls (256.37,150.02) and (258.03,150.02) .. (259.7,148.36) .. controls (261.37,146.7) and (263.04,146.71) .. (264.7,148.38) .. controls (266.37,150.05) and (268.03,150.05) .. (269.7,148.39) .. controls (271.37,146.73) and (273.04,146.74) .. (274.7,148.41) .. controls (276.37,150.08) and (278.03,150.08) .. (279.7,148.42) .. controls (281.37,146.76) and (283.04,146.77) .. (284.7,148.44) .. controls (286.37,150.11) and (288.03,150.11) .. (289.7,148.45) .. controls (291.37,146.79) and (293.04,146.8) .. (294.7,148.47) .. controls (296.36,150.14) and (298.03,150.15) .. (299.7,148.49) -- (304.3,148.5) -- (304.3,148.5) ;
\draw    (344.3,4.5) -- (343.3,227.5) ;
\draw    (306.3,141.5) -- (306.3,155.5) ;
\draw  [color={rgb, 255:red, 74; green, 144; blue, 226 }  ,draw opacity=1 ][fill={rgb, 255:red, 74; green, 144; blue, 226 }  ,fill opacity=0.51 ][line width=2.25]  (527,-9) .. controls (551,-34) and (619,-46) .. (651,-11) .. controls (652,32) and (653,179) .. (652,221) .. controls (638,268) and (541,249) .. (518,222) .. controls (504,183) and (476,150) .. (471,107) .. controls (470,70) and (512,35) .. (527,-9) -- cycle ;
\draw [color={rgb, 255:red, 144; green, 19; blue, 254 }  ,draw opacity=1 ][line width=1.5]  [dash pattern={on 5.63pt off 4.5pt}]  (18.8,119) .. controls (134.8,92) and (317.8,-17) .. (304.3,148.5) ;
\draw [color={rgb, 255:red, 65; green, 117; blue, 5 }  ,draw opacity=1 ][line width=1.5]  [dash pattern={on 5.63pt off 4.5pt}]  (25.8,120) .. controls (154.8,114) and (305.8,124) .. (304.3,148.5) .. controls (307.8,175) and (92.8,189) .. (25.8,186) ;
\draw [color={rgb, 255:red, 65; green, 117; blue, 5 }  ,draw opacity=1 ][line width=1.5]  [dash pattern={on 5.63pt off 4.5pt}]  (25.8,120) .. controls (154.8,114) and (431.8,94) .. (429.8,147) .. controls (431.8,221) and (104.8,186) .. (25.8,186) ;

\draw (303,163.4) node [anchor=north west][inner sep=0.75pt]    {$-1$};
\draw (352,151.9) node [anchor=north west][inner sep=0.75pt]    {$0$};
\draw (417.33,95.9) node [anchor=north west][inner sep=0.75pt]  [color={rgb, 255:red, 208; green, 2; blue, 27 }  ,opacity=1 ]  {$\ell _{0}$};
\draw (48,26.4) node [anchor=north west][inner sep=0.75pt]    {$\Re f( \xi ,z) =\Re f( \xi _{+} ,z)$};
\draw (281,52.23) node [anchor=north west][inner sep=0.75pt]  [color={rgb, 255:red, 0; green, 0; blue, 0 }  ,opacity=1 ]  {$\xi _{+}$};
\draw (472,178.4) node [anchor=north west][inner sep=0.75pt]  [color={rgb, 255:red, 74; green, 144; blue, 226 }  ,opacity=1 ]  {$\ell _{\infty }$};
\draw (226.17,30.4) node [anchor=north west][inner sep=0.75pt]  [color={rgb, 255:red, 65; green, 117; blue, 5 }  ,opacity=1 ]  {$\textcolor[rgb]{0.56,0.07,1}{\gamma^{(1)}_{+} }$};
\draw (139.17,156.4) node [anchor=north west][inner sep=0.75pt]  [color={rgb, 255:red, 65; green, 117; blue, 5 }  ,opacity=1 ]  {$\textcolor[rgb]{0.25,0.46,0.02}{\gamma ^{( 1)}}$};
\draw (383.17,190.4) node [anchor=north west][inner sep=0.75pt]  [color={rgb, 255:red, 65; green, 117; blue, 5 }  ,opacity=1 ]  {$\textcolor[rgb]{0.25,0.46,0.02}{\gamma ^{( 2)}}$};
\end{tikzpicture}
\end{center}
\caption{An illustration of the topological configuration when $\ell_{0}$ splits at $\xi_{+}$ with the ``opposite side'' interference from the branch cut.} \label{fig:conf_ell0_split_opposite_side_charlier}
\end{figure}
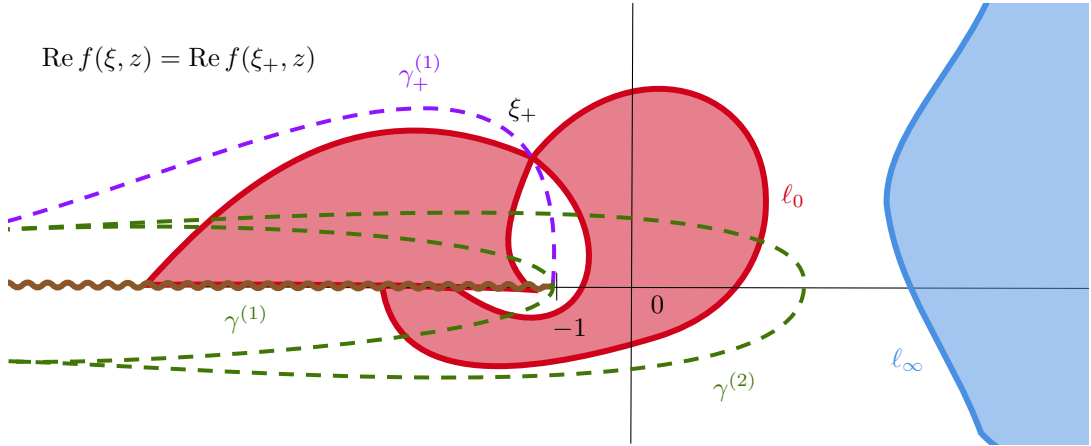

Recall $z\in\Omega_{-}$ is fixed, i.e., $\Re f(\xi_{+};z)<\Re f(\xi_{-};z)$. For brevity, we temporarily denote
\[
 \psi_{n}(\xi):=g(\xi)e^{-nf(\xi;z)}.
\]
Further, we define two particular curves. Let $\gamma^{(1)}$ denote a simple curve going from $-\infty+\ii$ in the upper half-plane $\Im\xi>0$ towards the point $-1$, where it crosses the real line, and continuous to $-\infty-\ii$ in the lower half-plane $\Im\xi<0$. Similarly, let $\gamma^{(2)}$ be a simple curve starting from $-\infty-\ii$, crossing the real line at a positive number, and continuing to $-\infty+\ii$. See Figure~\ref{fig:conf_ell0_split_opposite_side_charlier}.

An inspection of the asymptotic behavior of $\psi_{n}(\xi)$, for $\xi\to-1$ and $\Re\xi\to-\infty$, together with the assumptions $a>0$ and $a<\Re z$ shows that $\psi_{n}$ is integrable along $\gamma^{(1)}$ as well as~$\gamma^{(2)}$. Then a limit argument yields
\begin{equation}
 I_{n}(z;a)=\oint_{\gamma}\psi_{n}(z)\dd z=\int_{\gamma^{(1)}}\psi_{n}(z)\dd z+\int_{\gamma^{(2)}}\psi_{n}(z)\dd z,
\label{eq:I_split_gam1_gam2}
\end{equation}
for any Jordan curve $\gamma$ located in $\C\setminus((-\infty,-1]\cup\{0\})$ with $0$ in its interior. Next, we separately deduce the asymptotic behavior of both integrals on the right-hand side of~\eqref{eq:I_split_gam1_gam2}, for $n\to\infty$.

Let us denote by $\gamma_{+}^{(1)}$ and $\gamma_{-}^{(1)}$ parts of the curve $\gamma^{(1)}$ located in the upper and lower half-plane, respectively. In addition, suppose the saddle point $\xi_{+}$ is located in the upper half-plane, i.e., $\Im\xi_{+}>0$, as shown in Figure~\ref{fig:conf_ell0_split_opposite_side_charlier}. Then $\gamma_{+}^{(1)}$ can be homotopically deformed to a curve which fulfills assumption~(v) of Theorem~\ref{thm:saddle-point} and the saddle point method yields the asymptotic formula
\begin{equation}
 \int_{\gamma^{(1)}_{+}}\psi_{n}(z)\dd z=\psi_{n}(\xi_{+})\sqrt{\frac{2\pi}{n\partial_{\xi}^{2}f(\xi_{+};z)}}\left(1+O\left(\frac{1}{\sqrt{n}}\right)\right), \quad n\to\infty.
\label{eq:asympt_psi_gamma_1_plus}
\end{equation}

It turns out that the integral of $\psi_{n}$ along $\gamma_{-}^{(1)}$ is proportional to the integral of $\psi_{n}$ along $\gamma_{+}^{(1)}$. Indeed, by deforming $\gamma_{\pm}^{(1)}$ to the cut $(-\infty,-1)$ from above and below, respectively, we obtain
\[
\int_{\gamma^{(1)}_{\pm}}\psi_{n}(z)\dd z=\pm\int_{-\infty}^{-1}\psi_{n}(x\pm\ii0)\dd x=
\mp e^{\mp\ii\pi n(a-z)}\int_{-\infty}^{-1}\frac{e^{anx}}{x^{n+1}(-1-x)^{n(a-z)+1}}\dd x,
\]
where 
\[
\psi_{n}(x\pm\ii0):=\lim_{\substack{\xi\to x \\ \Im\xi\gtrless0}}\psi_{n}(\xi).
\]
Thus,
\[
\int_{\gamma^{(1)}_{-}}\psi_{n}(z)\dd z=-e^{-2\ii\pi n(a-z)}\int_{\gamma^{(1)}_{+}}\psi_{n}(z)\dd z,
\]
which together with~\eqref{eq:asympt_psi_gamma_1_plus} yields
\[
 \int_{\gamma^{(1)}}\psi_{n}(z)\dd z=\left(1-e^{-2\ii\pi n(a-z)}\right)\psi_{n}(\xi_{+})\sqrt{\frac{2\pi}{n\partial_{\xi}^{2}f(\xi_{+};z)}}\left(1+O\left(\frac{1}{\sqrt{n}}\right)\right), \quad n\to\infty.
\]
If $\xi_{+}$ is located in the lower-half plane, i.e, $\Im\xi_{+}<0$, one proceeds analogously and gets
\[
 \int_{\gamma^{(1)}}\psi_{n}(z)\dd z=\left(1-e^{2\ii\pi n(a-z)}\right)\psi_{n}(\xi_{+})\sqrt{\frac{2\pi}{n\partial_{\xi}^{2}f(\xi_{+};z)}}\left(1+O\left(\frac{1}{\sqrt{n}}\right)\right), \quad n\to\infty.
\]
In summary, we have the asymptotic formula
\begin{equation}
 \int_{\gamma^{(1)}}\psi_{n}(z)\dd z=\omega(\xi_{+};z)\psi_{n}(\xi_{+})\sqrt{\frac{2\pi}{n\partial_{\xi}^{2}f(\xi_{+};z)}}\left(1+O\left(\frac{1}{\sqrt{n}}\right)\right), \quad n\to\infty,
\label{eq:asympt_psi_gamma_1}
\end{equation}
where
\[
 \omega(\xi_{+};z):=1-\exp\left(-2\ii\pi n(a-z)\sign(\Im\xi_{+})\right)
\]
is a non-vanishing factor.

To deduce the asymptotic behavior of the integral of $\psi_{n}$ along $\gamma^{(2)}$, we let the level curves evolve to the next topologically different picture which occurs when $\Re f(\xi;z)=\Re f(\xi_{-};z)$. At this point, the two components $\ell_{0}$ and $\ell_{\infty}$ has to merge together at the sub-dominant saddle point $\xi_{-}$, see Figure~\ref{fig:conf_ell0_ellinf_merge_case2-charlier}. Then the curve $\gamma^{(2)}$ can be homotopically deformed so that the assumption~(v) of Theorem~\ref{thm:saddle-point} is fulfilled and the saddle point method implies the asymptotic formula 
\begin{equation}
 \int_{\gamma^{(2)}}\psi_{n}(z)\dd z=\psi_{n}(\xi_{-})\sqrt{\frac{2\pi}{n\partial_{\xi}^{2}f(\xi_{-};z)}}\left(1+O\left(\frac{1}{\sqrt{n}}\right)\right), \quad n\to\infty.
\label{eq:asympt_psi_gamma_2}
\end{equation}
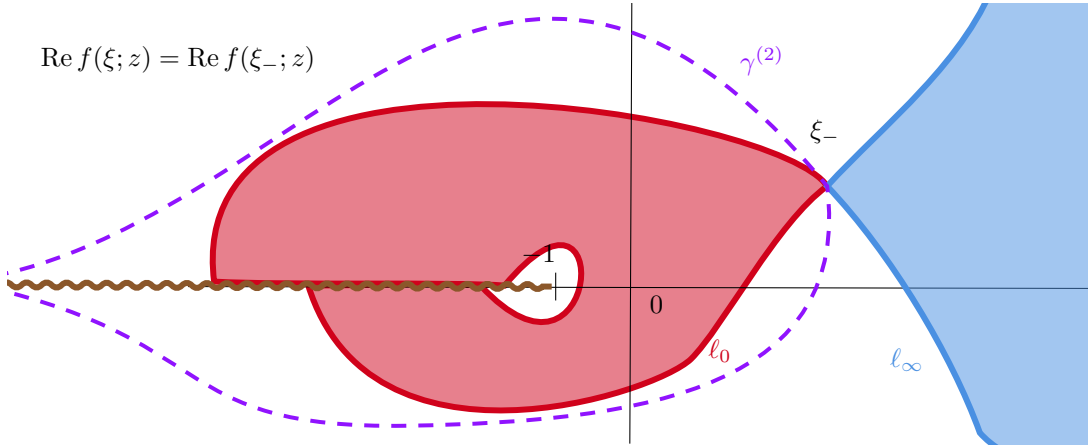
\begin{figure}[htb]
\begin{center}
\begin{tikzpicture}[x=0.75pt,y=0.75pt,yscale=-1,xscale=1]

\clip   (32.8,6) rectangle (571.8,228) ;
\draw  [color={rgb, 255:red, 208; green, 2; blue, 27 }  ,draw opacity=1 ][fill={rgb, 255:red, 208; green, 2; blue, 27 }  ,fill opacity=0.5 ][line width=2.25]  (441.8,98) .. controls (418.8,62) and (118.8,8) .. (135.8,146) .. controls (171.8,148) and (255.8,146) .. (280.8,148) .. controls (337.8,79) and (328.8,213) .. (268.8,148) .. controls (249.3,146.5) and (205.8,148) .. (182.8,148) .. controls (213.8,248) and (354.8,203) .. (373.8,185) .. controls (389.8,169) and (416.8,116) .. (441.8,98) -- cycle ;
\draw    (122.8,148.5) -- (599.8,149.5) ;
\draw [color={rgb, 255:red, 139; green, 87; blue, 42 }  ,draw opacity=1 ][line width=2.25]    (29.7,147.65) .. controls (31.37,145.99) and (33.04,146) .. (34.7,147.67) .. controls (36.37,149.34) and (38.03,149.34) .. (39.7,147.68) .. controls (41.37,146.02) and (43.04,146.03) .. (44.7,147.7) .. controls (46.37,149.37) and (48.03,149.37) .. (49.7,147.71) .. controls (51.37,146.05) and (53.04,146.06) .. (54.7,147.73) .. controls (56.37,149.4) and (58.03,149.4) .. (59.7,147.74) .. controls (61.37,146.08) and (63.04,146.09) .. (64.7,147.76) .. controls (66.37,149.43) and (68.03,149.43) .. (69.7,147.77) .. controls (71.37,146.11) and (73.04,146.12) .. (74.7,147.79) .. controls (76.37,149.46) and (78.03,149.46) .. (79.7,147.8) .. controls (81.37,146.14) and (83.04,146.15) .. (84.7,147.82) .. controls (86.36,149.49) and (88.03,149.5) .. (89.7,147.84) .. controls (91.37,146.18) and (93.03,146.18) .. (94.7,147.85) .. controls (96.36,149.52) and (98.03,149.53) .. (99.7,147.87) .. controls (101.37,146.21) and (103.03,146.21) .. (104.7,147.88) .. controls (106.36,149.55) and (108.03,149.56) .. (109.7,147.9) .. controls (111.37,146.24) and (113.03,146.24) .. (114.7,147.91) .. controls (116.36,149.58) and (118.03,149.59) .. (119.7,147.93) .. controls (121.37,146.27) and (123.03,146.27) .. (124.7,147.94) .. controls (126.36,149.61) and (128.03,149.62) .. (129.7,147.96) .. controls (131.37,146.3) and (133.04,146.31) .. (134.7,147.98) .. controls (136.37,149.65) and (138.03,149.65) .. (139.7,147.99) .. controls (141.37,146.33) and (143.04,146.34) .. (144.7,148.01) .. controls (146.37,149.68) and (148.03,149.68) .. (149.7,148.02) .. controls (151.37,146.36) and (153.04,146.37) .. (154.7,148.04) .. controls (156.37,149.71) and (158.03,149.71) .. (159.7,148.05) .. controls (161.37,146.39) and (163.04,146.4) .. (164.7,148.07) .. controls (166.37,149.74) and (168.03,149.74) .. (169.7,148.08) .. controls (171.37,146.42) and (173.04,146.43) .. (174.7,148.1) .. controls (176.37,149.77) and (178.03,149.77) .. (179.7,148.11) .. controls (181.37,146.45) and (183.04,146.46) .. (184.7,148.13) .. controls (186.36,149.8) and (188.03,149.81) .. (189.7,148.15) .. controls (191.37,146.49) and (193.03,146.49) .. (194.7,148.16) .. controls (196.36,149.83) and (198.03,149.84) .. (199.7,148.18) .. controls (201.37,146.52) and (203.03,146.52) .. (204.7,148.19) .. controls (206.36,149.86) and (208.03,149.87) .. (209.7,148.21) .. controls (211.37,146.55) and (213.03,146.55) .. (214.7,148.22) .. controls (216.36,149.89) and (218.03,149.9) .. (219.7,148.24) .. controls (221.37,146.58) and (223.03,146.58) .. (224.7,148.25) .. controls (226.36,149.92) and (228.03,149.93) .. (229.7,148.27) .. controls (231.37,146.61) and (233.03,146.61) .. (234.7,148.28) .. controls (236.36,149.95) and (238.03,149.96) .. (239.7,148.3) .. controls (241.37,146.64) and (243.04,146.65) .. (244.7,148.32) .. controls (246.37,149.99) and (248.03,149.99) .. (249.7,148.33) .. controls (251.37,146.67) and (253.04,146.68) .. (254.7,148.35) .. controls (256.37,150.02) and (258.03,150.02) .. (259.7,148.36) .. controls (261.37,146.7) and (263.04,146.71) .. (264.7,148.38) .. controls (266.37,150.05) and (268.03,150.05) .. (269.7,148.39) .. controls (271.37,146.73) and (273.04,146.74) .. (274.7,148.41) .. controls (276.37,150.08) and (278.03,150.08) .. (279.7,148.42) .. controls (281.37,146.76) and (283.04,146.77) .. (284.7,148.44) .. controls (286.37,150.11) and (288.03,150.11) .. (289.7,148.45) .. controls (291.37,146.79) and (293.04,146.8) .. (294.7,148.47) .. controls (296.36,150.14) and (298.03,150.15) .. (299.7,148.49) -- (304.3,148.5) -- (304.3,148.5) ;
\draw    (344.3,4.5) -- (343.3,227.5) ;
\draw    (306.3,141.5) -- (306.3,155.5) ;
\draw  [color={rgb, 255:red, 74; green, 144; blue, 226 }  ,draw opacity=1 ][fill={rgb, 255:red, 74; green, 144; blue, 226 }  ,fill opacity=0.51 ][line width=2.25]  (527,-9) .. controls (551,-34) and (619,-46) .. (651,-11) .. controls (652,32) and (653,179) .. (652,221) .. controls (638,268) and (541,249) .. (518,222) .. controls (504,183) and (473.8,132) .. (441.8,98) .. controls (469.8,67) and (512,35) .. (527,-9) -- cycle ;
\draw  [color={rgb, 255:red, 144; green, 19; blue, 254 }  ,draw opacity=1 ][dash pattern={on 5.63pt off 4.5pt}][line width=1.5]  (9.8,146) .. controls (141.8,173) and (109.05,220.5) .. (232.17,218.5) .. controls (355.3,216.5) and (453.8,198) .. (441.8,98) .. controls (272.8,-107) and (179.8,124) .. (9.8,146) -- cycle ;

\draw (288,126.4) node [anchor=north west][inner sep=0.75pt]    {$-1$};
\draw (352,151.9) node [anchor=north west][inner sep=0.75pt]    {$0$};
\draw (381.33,173.9) node [anchor=north west][inner sep=0.75pt]  [color={rgb, 255:red, 208; green, 2; blue, 27 }  ,opacity=1 ]  {$\ell _{0}$};
\draw (48,26.4) node [anchor=north west][inner sep=0.75pt]    {$\Re f( \xi;z) =\Re f( \xi _{-};z)$};
\draw (472,178.4) node [anchor=north west][inner sep=0.75pt]  [color={rgb, 255:red, 74; green, 144; blue, 226 }  ,opacity=1 ]  {$\ell _{\infty }$};
\draw (432,63.23) node [anchor=north west][inner sep=0.75pt]  [color={rgb, 255:red, 0; green, 0; blue, 0 }  ,opacity=1 ]  {$\xi _{-}$};
\draw (397.17,26.4) node [anchor=north west][inner sep=0.75pt]  [color={rgb, 255:red, 65; green, 117; blue, 5 }  ,opacity=1 ]  {$\textcolor[rgb]{0.56,0.07,1}{\gamma ^{( 2)}}$};
\end{tikzpicture}
\end{center}
\caption{An illustration of the topological configuration when $\ell_{0}$ and $\ell_{\infty}$ merge at $\xi_{-}$. The curve $\gamma^{(2)}$ fulfills the assumption~(v) of Theorem~\ref{thm:saddle-point}.} \label{fig:conf_ell0_ellinf_merge_case2-charlier}
\end{figure}

Since $\Re\xi_{+}<\Re\xi_{-}$, the right-hand side of~\eqref{eq:asympt_psi_gamma_2} decays exponentially faster when compared to the right-hand side of~\eqref{eq:asympt_psi_gamma_1}. It follows that the asymptotic behavior of~\eqref{eq:I_split_gam1_gam2}, for $n\to\infty$, is governed by the integral along the curve~$\gamma^{(1)}$. Hence, we may conclude that 
\[
I_{n}(z;a)=\omega(\xi_{+};z)\psi_{n}(\xi_{+})\sqrt{\frac{2\pi}{n\partial_{\xi}^{2}f(\xi_{+};z)}}\left(1+O\left(\frac{1}{\sqrt{n}}\right)\right), \quad n\to\infty.
\]
Very much the same approach, with the function $g$ replaced by $-g(\partial_{z} f)$, yields the asymptotic formula
\[
I_{n}'(z;a)=\frac{\partial f}{\partial z}(\xi_{+};z)\omega(\xi_{+};z)\psi_{n}(\xi_{+})\sqrt{\frac{2\pi}{n\partial_{\xi}^{2}f(\xi_{+};z)}}\left(1+O\left(\frac{1}{\sqrt{n}}\right)\right), \quad n\to\infty.
\]
Consequently, we end up with the desired ratio asymptotics~\eqref{eq:ratio_I_n_asympt} again.

C) \textbf{The case $a=\Re z$:} So far we have established the ratio limit formula~\eqref{eq:ratio_I_n_asympt}, and hence Theorem~\ref{thm:cauchy_transf_charlier}, for all $z\in\Omega_{-}$ such that $\Re z\neq a$. If $a\geq 1$, $\Omega_{-}$ has no intersection with the vertical line defined by equation $\Re z=a$. On the other hand, if $a\in(0,1)$, the limiting Cauchy transform $C_{\mu}$ extends continuously to the vertical cut $\Re z=a$ of $\Omega_{-}$. Consequently, the formula for $C_{\mu}(z)$ of Theorem~\ref{thm:cauchy_transf_charlier} remains true also for $z\in\Omega_{-}$ with $\Re z=a$.

In the end, we point out that, if one assumes $z\in\Omega_{+}$, rather than $z\in\Omega_{-}$, the whole discussion proceeds the same with interchanged roles of the saddle points $\xi_{+}$ and $\xi_{-}$. It follows the formula for $C_{\mu}(z)$ of Theorem~\ref{thm:cauchy_transf_charlier} for $z\in\Omega_{+}$, and the proof of Theorem~\ref{thm:cauchy_transf_charlier} is completed.

\section{An application to complex sampling Jacobi matrices}\label{sec:sampling_jacobi}

\subsection{A solvable case of complex sampling Jacobi matrices}

Results of Theorem~\ref{thm:charlier} can be equivalently interpreted as a description of the asymptotic eigenvalue distribution of a~concrete instance of the Jacobi matrix $J_{n}=J_{n}(\alpha,\beta)$, whose diagonal and off-diagonal entries are given by values of two continuous functions $\alpha,\beta:[0,1]\to\C$ sampled on the equidistant grid, i.e., 
\[
 \left(J_{n}(\alpha,\beta)\right)_{k,k}=\beta\left(\tfrac{k}{n}\right) \quad\mbox{ and }\quad 
 \left(J_{n}(\alpha,\beta)\right)_{k,k+1}=\left(J_{n}(\alpha,\beta)\right)_{k+1,k}=\alpha\left(\tfrac{k}{n}\right),
\]
for $k=1,\dots,n$ and $k=1,\dots,n-1$, respectively. Indeed, it follows readily from equations~\eqref{eq:jacobi_matrix_entries_charlier} and~\eqref{eq:char_pol_charlier} that roots of $P_{n}^{\mathrm{C}}$ coincide with eigenvalues of $J_{n}(\alpha,\beta)$ with sampling functions
\begin{equation}
 \beta(x)=x \quad\mbox{ and }\quad \alpha(x)=\ii\sqrt{ax}.
\label{eq:beta_alpha_charlier}
\end{equation}

A determination of the asymptotic eigenvalue distribution of $J_{n}(\alpha,\beta)$ has been studied in an even greater generality for matrices which are not necessarily tridiagonal and appear under various names such as generalized Toeplitz, Toeplitz-like, variable coefficient Toeplitz, locally Toeplitz or Kac--Murdock--Szeg{\" o} matrices. The last name refers to a pioneering work~\cite{kac-mur-sze_53} due to Kac, Murdock, and Szeg{\" o}, who derived the asymptotic eigenvalue distribution for certain generalized \textbf{self-adjoint} Toeplitz matrices. Later on, Tilli, with a motivation coming from a numerical analysis of ODEs, introduced the concept of locally Toeplitz sequences and rediscovered the result of~\cite{kac-mur-sze_53} in~\cite{til_98} in a slightly different setting. The concept has been further developed to a theory of generalized locally Toeplitz sequences~\cite{serra-capizzano17,serra-capizzano18}. Similarly, Kuijlaars and Van Assche, with a motivation coming from the theory of orthogonal polynomials, deduced the asymptotic eigenvalue distribution of matrices $J_{n}(\alpha,\beta)$ in the self-adjoint setting in~\cite{kui-van_99}.

Methods applied in the above mentioned works heavily depend on the self-adjointness of the studied matrices. In case of $J_{n}(\alpha,\beta)$, it is equivalent to the assumption that the sampling functions $\alpha$ and $\beta$ are real-valued. Much less is known when the self-adjointness or normality is not assumed. To our best knowledge, among the reasonable general classes of non-normal matrices, for which the asymptotic eigenvalue distribution has been found, there belong only the banded Toeplitz matrices~\cite{bot-gru_05,hir_67,sch-spi_60} and Toeplitz matrices with a~rational symbol~\cite{day_75a,day_75b}. Few more works relevant in the non-self-adjoint setting are~\cite{gol-ser-cap_07, til_99}. It has been conjectured in~\cite{bou-loy-tyl_18} (in a greater generality) that the eigenvalues of complex sampling Jacobi matrices $J_{n}(\alpha,\beta)$, with complex-valued $\alpha,\beta\in C([0,1])$, cluster in sets that exhibit the same properties as in the case of complex banded Toeplitz matrices, for $n\to\infty$. More concretely, the asymptotic eigenvalue distribution is expected to exist and should be supported in a~connected set which is a union of finite number of pairwise disjoint open analytic arcs and finite number of the so-called exceptional points, see~\cite[Chp.~11]{bot-gru_05} for details. Theorem~\ref{thm:charlier} is in agreement with this conjecture.

\subsection{Concluding remarks}

\begin{enumerate}[1.]
\item Our motivation to analyze the concrete sequence of polynomials $P_{n}^{\mathrm{C}}$ stems from the above open problem in a hope that solvable examples may shed some light to the problem in general. Nevertheless, it is obvious that the applied method could be hardly worked out in a more general setting. First, a relatively simple generating function formula implying the contour integral representation~\eqref{eq:integr_repre_spm_charlier} is needed. Second, it is obvious that the discussion of possible configurations from Subsection~\ref{subsec:spm_charlier} would become too complicated, for example, in cases of more than two simple saddle points. In other words, the assumption~(v) of Theorem~\ref{thm:saddle-point} is a significant limiting factor to the method.  

\item On the other hand, the saddle point method provides us with the asymptotic behavior of $P_{n}^{\mathrm{C}}$ and its derivative, for $n\to\infty$, which is more information than actually needed when the goal is a determination of the limiting Cauchy transform. To this end, we need only to know an asymptotic formula for the ratio~\eqref{eq:cauchy_transf_root_count}. Usual strategies for a derivation of the ratio asymptotics, which are based on limiting formulas for powers of traces of analyzed matrices, see~\cite{kac-mur-sze_53}, or formulas for the limiting logarithmic potential in a neighborhood of infinity, see~\cite{kui-van_99}, are not applicable in the non-self-adjoint setting because a~compactly supported measure in $\C$ is not determined by its moment sequence or the~logarithmic potential (nor Cauchy transform) known in a neighborhood of infinity.

\item One can observe that the asymptotic distribution of roots of~$P_{n}^{\mathrm{C}}$ is, in a sense, determined by the complex function $f$ given by~\eqref{eq:def_f_charlier}. Although this is a common feature in several more solvable models, see for example~\cite{bla-sta_jmaa20}, it is not clear if it can be generalized to some extend. Even if such a function $f$ would exist in a greater generality, it is not obvious how it would be encoded in the determining functions~\eqref{eq:beta_alpha_charlier}. Moreover, only certain parts of the curves in the $z$-plane given by equation~\eqref{eq:re_f_eq} form the support of the limiting measure. Which parts should be relevant in general remain also unclear. 

\item There are at least two more recent and advanced methods that were successfully applied in similar problems. First, it is a fairly general GRS theory related to critical measures and trajectories of quadratic differentials~\cite{mar-fin-rak_11, mar-fin-rak_16, rak_12}. Second, the nonlinear steepest descent method of the Riemann--Hilbert problem~\cite{dei-zho_93,dei_99} proved to be another powerful tool in asymptotic analysis of polynomial families. Neither method, however, does seem to be readily applicable when a sequence of matrices $J_{n}(\alpha,\beta)$ is given. Rather than recurrence coefficients (determined by the functions $\alpha$, $\beta$), the usual starting point for an application of the aformentioned techniques is a~possibly non-Hermitian orthogonality relation. 

\item On the other hand, methods of the Riemann--Hilbert problem could be applied to deduce an asymptotic behavior of zeros of the Charlier polynomials $C_{n}^{(a_{n})}(nz)$ assuming, for example, that $\lim_{n\to\infty}a_{n}/n= a<0$, similarly as it has been done for the Laguerre polynomials. It would be an interesting generalization of the current study.

\end{enumerate}

\section*{Acknowledgement}
The research of P.~B. was supported by the grant No.~201/12/G028 of the Czech Science Foundation.
F.~{\v S}. acknowledges financial support of the EXPRO grant No.~20-17749X of the Czech Science Foundation.

\bibliographystyle{acm}

\begin{thebibliography}{10}

\bibitem{aba-bog_16}
{\sc Abathun, A., and B\o~gvad, R.}
\newblock Asymptotic distribution of zeros of a certain class of hypergeometric
  polynomials.
\newblock {\em Comput. Methods Funct. Theory 16}, 2 (2016), 167--185.

\bibitem{aba-bog_18}
{\sc Abathun, A., and B\o~gvad, R.}
\newblock Zeros of a certain class of {G}auss hypergeometric polynomials.
\newblock {\em Czechoslovak Math. J. 68(143)}, 4 (2018), 1021--1031.

\bibitem{ati-mar-fin-mar-gon-tha_14}
{\sc Atia, M.~J., Mart\'{\i}nez-Finkelshtein, A., Mart\'{\i}nez-Gonz\'{a}lez,
  P., and Thabet, F.}
\newblock Quadratic differentials and asymptotics of {L}aguerre polynomials
  with varying complex parameters.
\newblock {\em J. Math. Anal. Appl. 416}, 1 (2014), 52--80.

\bibitem{bla-sta_jmaa20}
{\sc Blaschke, P., and \v{S}tampach, F.}
\newblock The asymptotic zero distribution of {L}ommel polynomials as functions
  of their order with a variable complex argument.
\newblock {\em J. Math. Anal. Appl. 490}, 1 (2020), 124238, 19.

\bibitem{rui-wong_94}
{\sc Bo, R., and Wong, R.}
\newblock Uniform asymptotic expansion of {C}harlier polynomials.
\newblock {\em Methods Appl. Anal. 1}, 3 (1994), 294--313.

\bibitem{bot-gru_05}
{\sc B\"{o}ttcher, A., and Grudsky, S.~M.}
\newblock {\em Spectral properties of banded {T}oeplitz matrices}.
\newblock Society for Industrial and Applied Mathematics (SIAM), Philadelphia,
  PA, 2005.

\bibitem{bou-loy-tyl_18}
{\sc Bourget, A., Loya, A.~A., and McMillen, T.}
\newblock Spectral asymptotics for {K}ac-{M}urdock-{S}zeg\"{o} matrices.
\newblock {\em Jpn. J. Math. 13}, 1 (2018), 67--107.

\bibitem{boy-goh_07}
{\sc Boyer, R., and Goh, W. M.~Y.}
\newblock On the zero attractor of the {E}uler polynomials.
\newblock {\em Adv. in Appl. Math. 38}, 1 (2007), 97--132.

\bibitem{boy-goh_08}
{\sc Boyer, R.~P., and Goh, W. M.~Y.}
\newblock Polynomials associated with partitions: asymptotics and zeros.
\newblock In {\em Special functions and orthogonal polynomials}, vol.~471 of
  {\em Contemp. Math.} Amer. Math. Soc., Providence, RI, 2008, pp.~33--45.

\bibitem{boy-goh_10}
{\sc Boyer, R.~P., and Goh, W. M.~Y.}
\newblock Appell polynomials and their zero attractors.
\newblock In {\em Gems in experimental mathematics}, vol.~517 of {\em Contemp.
  Math.} Amer. Math. Soc., Providence, RI, 2010, pp.~69--96.

\bibitem{dai-won_08}
{\sc Dai, D., and Wong, R.}
\newblock Global asymptotics for {L}aguerre polynomials with large negative
  parameter---a {R}iemann-{H}ilbert approach.
\newblock {\em Ramanujan J. 16}, 2 (2008), 181--209.

\bibitem{day_75a}
{\sc Day, K.~M.}
\newblock Measures associated with {T}oeplitz matrices generated by the
  {L}aurent expansion of rational functions.
\newblock {\em Trans. Amer. Math. Soc. 209\/} (1975), 175--183.

\bibitem{day_75b}
{\sc Day, K.~M.}
\newblock Toeplitz matrices generated by the {L}aurent series expansion of an
  arbitrary rational function.
\newblock {\em Trans. Amer. Math. Soc. 206\/} (1975), 224--245.

\bibitem{dei-zho_93}
{\sc Deift, P., and Zhou, X.}
\newblock A steepest descent method for oscillatory {R}iemann-{H}ilbert
  problems. {A}symptotics for the {MK}d{V} equation.
\newblock {\em Ann. of Math. (2) 137}, 2 (1993), 295--368.

\bibitem{dei_99}
{\sc Deift, P.~A.}
\newblock {\em Orthogonal polynomials and random matrices: a
  {R}iemann-{H}ilbert approach}, vol.~3 of {\em Courant Lecture Notes in
  Mathematics}.
\newblock New York University, Courant Institute of Mathematical Sciences, New
  York; American Mathematical Society, Providence, RI, 1999.

\bibitem{dia-ori_11}
{\sc D\'{\i}az~Mendoza, C., and Orive, R.}
\newblock The {S}zeg\"{o} curve and {L}aguerre polynomials with large negative
  parameters.
\newblock {\em J. Math. Anal. Appl. 379}, 1 (2011), 305--315.

\bibitem{dri-jor_03}
{\sc Driver, K., and Jordaan, K.}
\newblock Asymptotic zero distribution of a class of {$_3F_2$} hypergeometric
  functions.
\newblock {\em Indag. Math. (N.S.) 14}, 3-4 (2003), 319--327.

\bibitem{dri-joh_07}
{\sc Driver, K.~A., and Johnston, S.~J.}
\newblock Asymptotic zero distribution of a class of hypergeometric
  polynomials.
\newblock {\em Quaest. Math. 30}, 2 (2007), 219--230.

\bibitem{dur-gui_01}
{\sc Duren, P.~L., and Guillou, B.~J.}
\newblock Asymptotic properties of zeros of hypergeometric polynomials.
\newblock {\em J. Approx. Theory 111}, 2 (2001), 329--343.

\bibitem{serra-capizzano17}
{\sc Garoni, C., and Serra-Capizzano, S.}
\newblock {\em Generalized locally {T}oeplitz sequences: theory and
  applications. {V}ol. {I}}.
\newblock Springer, Cham, 2017.

\bibitem{serra-capizzano18}
{\sc Garoni, C., and Serra-Capizzano, S.}
\newblock {\em Generalized locally {T}oeplitz sequences: theory and
  applications. {V}ol. {II}}.
\newblock Springer, Cham, 2018.

\bibitem{goh_ca98}
{\sc Goh, W. M.~Y.}
\newblock Plancherel-{R}otach asymptotics for the {C}harlier polynomials.
\newblock {\em Constr. Approx. 14}, 2 (1998), 151--168.

\bibitem{gol-ser-cap_07}
{\sc Golinskii, L., and Serra-Capizzano, S.}
\newblock The asymptotic properties of the spectrum of nonsymmetrically
  perturbed {J}acobi matrix sequences.
\newblock {\em J. Approx. Theory 144}, 1 (2007), 84--102.

\bibitem{hir_67}
{\sc Hirschman, Jr., I.~I.}
\newblock The spectra of certain {T}oeplitz matrices.
\newblock {\em Illinois J. Math. 11\/} (1967), 145--159.

\bibitem{hua-lin-zha_21}
{\sc Huang, X.-M., Lin, Y., and Zhao, Y.-Q.}
\newblock Asymptotics of the {C}harlier polynomials via difference equation
  methods.
\newblock {\em Anal. Appl. (Singap.) 19}, 4 (2021), 679--713.

\bibitem{kac-mur-sze_53}
{\sc Kac, M., Murdock, W.~L., and Szeg\"{o}, G.}
\newblock On the eigenvalues of certain {H}ermitian forms.
\newblock {\em J. Rational Mech. Anal. 2\/} (1953), 767--800.

\bibitem{koe-les-swa_10}
{\sc Koekoek, R., Lesky, P.~A., and Swarttouw, R.~F.}
\newblock {\em Hypergeometric orthogonal polynomials and their
  {$q$}-analogues}.
\newblock Springer Monographs in Mathematics. Springer-Verlag, Berlin, 2010.
\newblock With a foreword by Tom H. Koornwinder.

\bibitem{kui-mar-fin_04}
{\sc Kuijlaars, A. B.~J., and Mart\'{\i}nez-Finkelshtein, A.}
\newblock Strong asymptotics for {J}acobi polynomials with varying nonstandard
  parameters.
\newblock {\em J. Anal. Math. 94\/} (2004), 195--234.

\bibitem{kui-mcl_01}
{\sc Kuijlaars, A. B.~J., and McLaughlin, K. T.-R.}
\newblock Riemann-{H}ilbert analysis for {L}aguerre polynomials with large
  negative parameter.
\newblock {\em Comput. Methods Funct. Theory 1}, 1, [On table of contents:
  2002] (2001), 205--233.

\bibitem{kui-mcl_04}
{\sc Kuijlaars, A. B.~J., and McLaughlin, K. T.-R.}
\newblock Asymptotic zero behavior of {L}aguerre polynomials with negative
  parameter.
\newblock {\em Constr. Approx. 20}, 4 (2004), 497--523.

\bibitem{kui-van_99}
{\sc Kuijlaars, A. B.~J., and Van~Assche, W.}
\newblock The asymptotic zero distribution of orthogonal polynomials with
  varying recurrence coefficients.
\newblock {\em J. Approx. Theory 99}, 1 (1999), 167--197.

\bibitem{lop-tem_04}
{\sc L\'{o}pez, J.~L., and Temme, N.~M.}
\newblock Convergent asymptotic expansions of {C}harlier, {L}aguerre and
  {J}acobi polynomials.
\newblock {\em Proc. Roy. Soc. Edinburgh Sect. A 134}, 3 (2004), 537--555.

\bibitem{mar-fin-mar-gon-ori_99}
{\sc Mart\'{\i}nez-Finkelshtein, A., Mart\'{\i}nez-Gonz\'{a}lez, P., and Orive,
  R.}
\newblock Zeros of {J}acobi polynomials with varying non-classical parameters.
\newblock In {\em Special functions ({H}ong {K}ong, 1999)}. World Sci. Publ.,
  River Edge, NJ, 2000, pp.~98--113.

\bibitem{mar-fin-mar-gon-ori_01}
{\sc Mart\'{\i}nez-Finkelshtein, A., Mart\'{\i}nez-Gonz\'{a}lez, P., and Orive,
  R.}
\newblock On asymptotic zero distribution of {L}aguerre and generalized
  {B}essel polynomials with varying parameters.
\newblock In {\em Proceedings of the {F}ifth {I}nternational {S}ymposium on
  {O}rthogonal {P}olynomials, {S}pecial {F}unctions and their {A}pplications
  ({P}atras, 1999)\/} (2001), vol.~133, pp.~477--487.

\bibitem{mar-fin-ori_05}
{\sc Mart\'{\i}nez-Finkelshtein, A., and Orive, R.}
\newblock Riemann-{H}ilbert analysis of {J}acobi polynomials orthogonal on a
  single contour.
\newblock {\em J. Approx. Theory 134}, 2 (2005), 137--170.

\bibitem{mar-fin-rak_11}
{\sc Mart\'{\i}nez-Finkelshtein, A., and Rakhmanov, E.~A.}
\newblock Critical measures, quadratic differentials, and weak limits of zeros
  of {S}tieltjes polynomials.
\newblock {\em Comm. Math. Phys. 302}, 1 (2011), 53--111.

\bibitem{mar-fin-rak_16}
{\sc Mart\'{\i}nez-Finkelshtein, A., and Rakhmanov, E.~A.}
\newblock Do orthogonal polynomials dream of symmetric curves?
\newblock {\em Found. Comput. Math. 16}, 6 (2016), 1697--1736.

\bibitem{mat_02}
{\sc Matsumoto, Y.}
\newblock {\em An introduction to {M}orse theory}, vol.~208 of {\em
  Translations of Mathematical Monographs}.
\newblock American Mathematical Society, Providence, RI, 2002.
\newblock Translated from the 1997 Japanese original by Kiki Hudson and
  Masahico Saito, Iwanami Series in Modern Mathematics.

\bibitem{mil_63}
{\sc Milnor, J.}
\newblock {\em Morse theory}.
\newblock Based on lecture notes by M. Spivak and R. Wells. Annals of
  Mathematics Studies, No. 51. Princeton University Press, Princeton, N.J.,
  1963.

\bibitem{olv_97}
{\sc Olver, F. W.~J.}
\newblock {\em Asymptotics and special functions}.
\newblock AKP Classics. A K Peters, Ltd., Wellesley, MA, 1997.
\newblock Reprint of the 1974 original [Academic Press, New York; MR0435697 (55
  \# 8655)].

\bibitem{rak_12}
{\sc Rakhmanov, E.~A.}
\newblock Orthogonal polynomials and {$S$}-curves.
\newblock In {\em Recent advances in orthogonal polynomials, special functions,
  and their applications}, vol.~578 of {\em Contemp. Math.} Amer. Math. Soc.,
  Providence, RI, 2012, pp.~195--239.

\bibitem{saf-tot_97}
{\sc Saff, E.~B., and Totik, V.}
\newblock {\em Logarithmic potentials with external fields}, vol.~316 of {\em
  Grundlehren der mathematischen Wissenschaften [Fundamental Principles of
  Mathematical Sciences]}.
\newblock Springer-Verlag, Berlin, 1997.
\newblock Appendix B by Thomas Bloom.

\bibitem{sch-spi_60}
{\sc Schmidt, P., and Spitzer, F.}
\newblock The {T}oeplitz matrices of an arbitrary {L}aurent polynomial.
\newblock {\em Math. Scand. 8\/} (1960), 15--38.

\bibitem{son-won_17}
{\sc Song, Z., and Wong, R.}
\newblock Asymptotics of pseudo-{J}acobi polynomials with varying parameters.
\newblock {\em Stud. Appl. Math. 139}, 1 (2017), 179--217.

\bibitem{til_98}
{\sc Tilli, P.}
\newblock Locally {T}oeplitz sequences: spectral properties and applications.
\newblock {\em Linear Algebra Appl. 278}, 1-3 (1998), 91--120.

\bibitem{til_99}
{\sc Tilli, P.}
\newblock Some results on complex {T}oeplitz eigenvalues.
\newblock {\em J. Math. Anal. Appl. 239}, 2 (1999), 390--401.

\bibitem{vla_84}
{\sc Vladimirov, V.~S.}
\newblock {\em Equations of mathematical physics}.
\newblock ``Mir'', Moscow, 1984.
\newblock Translated from the Russian by Eugene Yankovsky [E. Yankovski\u\i ].

\bibitem{sta_21}
{\sc \v{S}tampach, F.}
\newblock Asymptotic behavior and zeros of the {B}ernoulli polynomials of the
  second kind.
\newblock {\em J. Approx. Theory 262\/} (2021), Paper No. 105517, 28.

\bibitem{wan-qui-won_13}
{\sc Wang, J., Qiu, W., and Wong, R.}
\newblock Global asymptotics for {M}eixner-{P}ollaczek polynomials with a
  varying parameter.
\newblock {\em Stud. Appl. Math. 130}, 4 (2013), 345--392.

\bibitem{wan-won_jmaa16}
{\sc Wang, X.-S., and Wong, R.}
\newblock Asymptotics of {R}acah polynomials with varying parameters.
\newblock {\em J. Math. Anal. Appl. 436}, 2 (2016), 1149--1164.

\bibitem{won_01}
{\sc Wong, R.}
\newblock {\em Asymptotic approximations of integrals}, vol.~34 of {\em
  Classics in Applied Mathematics}.
\newblock Society for Industrial and Applied Mathematics (SIAM), Philadelphia,
  PA, 2001.
\newblock Corrected reprint of the 1989 original.

\bibitem{won-zha_06}
{\sc Wong, R., and Zhang, W.}
\newblock Uniform asymptotics for {J}acobi polynomials with varying large
  negative parameters---a {R}iemann-{H}ilbert approach.
\newblock {\em Trans. Amer. Math. Soc. 358}, 6 (2006), 2663--2694.

\bibitem{zhou-sri-wan_12}
{\sc Zhou, J.-R., Srivastava, H.~M., and Wang, Z.-G.}
\newblock Asymptotic distributions of the zeros of a family of hypergeometric
  polynomials.
\newblock {\em Proc. Amer. Math. Soc. 140}, 7 (2012), 2333--2346.

\end{thebibliography}

\end{document}